\documentclass[12pt]{amsart}
\usepackage{amssymb,tikz-cd,amsmath}
\usepackage{mathrsfs}
\usepackage{verbatim}
\usepackage{mathtools}
\newcommand{\Pro}{\mathcal{P}}
\newcommand{\Fl}{\mathcal{F}l}
\newcommand{\C}{\mathbb{C}}
\newcommand{\Z}{\mathbb{Z}}
\newcommand{\GL}{\operatorname{GL}}
\newcommand{\Hcal}{H}
\newcommand{\gr}{\operatorname{gr}}
\newcommand{\Dcal}{\mathcal{D}}
\newcommand{\End}{\operatorname{End}}
\newcommand{\Loc}{\mathsf{Loc}}
\newcommand{\Hom}{\operatorname{Hom}}
\newcommand{\F}{\mathbb{F}}
\newcommand{\Ocal}{\mathcal{O}}
\newcommand{\Ecal}{\mathcal{E}}

\newcommand{\h}{\mathfrak{t}}
\newcommand{\g}{\mathfrak{g}}
\newcommand{\Q}{\mathbb{Q}}
\renewcommand{\GL}{\operatorname{GL}}

\newcommand{\Fr}{\operatorname{Fr}}
\newcommand{\Lcal}{\mathcal{L}}
\newcommand{\Ring}{\mathsf{R}}
\newcommand{\Field}{\mathsf{F}}
\newtheorem{Thm}{Theorem}[section]
\newtheorem{Prop}[Thm]{Proposition}
\newtheorem{Cor}[Thm]{Corollary}
\newtheorem{Lem}[Thm]{Lemma}
\theoremstyle{definition}
\newtheorem{Ex}[Thm]{Example}
\newtheorem{defi}[Thm]{Definition}
\newtheorem{Rem}[Thm]{Remark}
\newtheorem{Conj}[Thm]{Conjecture}
\newtheorem{Claim}[Thm]{Claim}
\numberwithin{equation}{section}
\title{Affine Springer Fibers, Procesi bundles, and Cherednik algebras}
\author{Pablo Boixeda Alvarez and Ivan Losev}
\address{P.B.A:  School of Mathematics, IAS, Princeton, NJ USA}
\email{pboixeda@ias.edu}
\address{I.L.: Department
of Mathematics, Yale University, New Haven CT USA \& School of Mathematics, IAS, Princeton, NJ USA}
\email{ivan.loseu@gmail.com}
\begin{document}
\begin{abstract}
Let $\mathfrak{g}$ be a semisimple Lie algebra, $\mathfrak{t}$ its Cartan subalgebra and $W$ the Weyl group.
    The goal of this paper is to prove an isomorphism between suitable completions of the equivariant Borel-Moore homology of certain affine Springer fibers for $\mathfrak{g}$ and the global sections of a bundle related to a Procesi bundle on the smooth locus of a partial resolution of $(\mathfrak{t}\oplus \mathfrak{t}^*)/W$. We deduce some applications of our isomorphism including a conditional application to the center of the small quantum group. Our main method is to compare certain bimodules over rational and trigonometric Cherednik algebras.
\end{abstract}
\maketitle
\tableofcontents
\section{Introduction}\label{intro}
Let $G$ be a connected reductive algebraic group over $\C$, $\g$ be its Lie algebra, $\h\subset \g$ a Cartan subalgebra, $T$ the corresponding maximal torus, and $W$  the Weyl group. Pick a nonnegative integer $d$.

The goal of this paper is to relate two different geometric objects, ``coherent'' and
``constructible'', constructed from these data.

First, we describe the ``coherent'' object. Consider the Poisson variety
$Y:=(\h\oplus \h^*)/W$.
We will choose a suitable partial Poisson resolution $X$ of $Y$, Section \ref{SS_Y_partial_resol}.
For example, in the case of $G=\operatorname{GL}_n$, the variety $X$ is going to be the Hilbert scheme of points in $\C^2$.
When $\g$ is simply laced, $X$ is going to be the so called $\Q$-factorial terminalization of $Y$,
see \cite{BCHM} for the general construction or \cite[Section 2.2]{SRA_der}
for a discussion in the present settings.  In types $B/C$, $F_4$, and $G_2$ we get some intermediate partial resolution. See Section \ref{SS_Y_partial_resol} for details. In all cases,
we are going to have $\operatorname{codim}_{X}X^{sing}\geqslant 4$.

The smooth locus $X^{reg}$ comes with several important vector bundles.  There is a ``Procesi bundle'' $\Pro^{reg}$ that will be constructed in Section
\ref{SS_Procesi_sheaves} based on results from
\cite{SRA_der}.   One important property of $\Pro^{reg}$ we need right now is that its endomorphism algebra is
\begin{equation}\label{eq:H_definition} H:=\C[\h\oplus \h^*]\#W\end{equation}
In the case when $G=\operatorname{GL}_n(\C)$, we recover Haiman's
Procesi bundle on the Hilbert scheme, \cite{Haiman}. When $\g$ is simply laced
and hence $X$ is a $\Q$-factorial terminalization of $Y$, $\Pro^{reg}$ is the restriction of the Procesi sheaf $\Pro$ on $X$, \cite[Section 4]{SRA_der} to $X^{reg}$. In types $B/C,F_4,G_2$, we consider the sheaf $\Pro$  obtained by  the pushforward of  the Procesi sheaf from the $\Q$-factorial terminalization to $X$ and then restrict it to $X^{reg}$. The sign invariants in $\Pro^{reg}$ is a line bundle to be denoted by $\Ocal^{reg}(1)$. Its $d$th tensor power will be denoted by $\Ocal^{reg}(d)$.

So, for $d\in \Z_{>0}$, we can consider the
$H$-bimodule
\begin{equation}\label{eq:bimodule1}
B_d:=\Gamma(X^{reg}, \Pro^{reg,*}\otimes \Ocal^{reg}(d)\otimes \Pro^{reg}).
\end{equation}
This is the first of the two  objects we are interested in.

There is a number of reasons to be interested in the bimodule $B_d$. First, consider  the case when $G=\operatorname{GL}_n$. The bimodule $B_d$ is closely related to the $d$th power of the so called $\nabla$ operator on symmetric polynomials, compare to \cite{CM}. It would be interesting to see whether this observation can be generalized to the case of general $G$.

Another reason to care about $B_d$ is that these bimodules (or their variants) are expected to appear in a variety of other contexts. The subject of this paper is their connection to the affine Springer theory. Another prospective appearance is the study of character sheaves on semisimple Lie algebras and the usual Springer theory: the bimodules $B_d$ are expected to be related to the central elements $T_{w_0}^{2d}$ in the Hecke category. A related appearance should be in the study of invariants of torus knots, see
\cite{GH}.

The second object we care about, the ``constructible'' one, is  the equivariant Borel-Moore homology of a suitable affine Springer fiber for the group $G$.

%The goal of this note is to sketch an %isomorphism between:
%\begin{enumerate}
%\item The space of global sections %$\Gamma(H_n^\times, \Pro\otimes %\Pro)$,
%where $\Pro$ is the Procesi bundle on %$H_n$,
%\item and the equivariant BM homology %group $H^{BM}_{T}(\Fl_e)$ of a %suitable
%affine Springer fiber $\Fl_e$.
%\end{enumerate}

Fix a regular element $s\in \h^*$. Let $t$ be an indeterminate so that
we can form the loop algebra $\g((t))$.  Consider the element $e_d:=st^d\in \g((t))$.
This element gives rise to the affine Springer fiber $\Fl_{e_d}$ in the affine flag variety $\Fl$ for $G$, sometimes it is called an {\it equivalued unramified} affine Springer fiber. The maximal torus $T$, the centralizer of $s$, acts on $\Fl_{e_d}$. So we can consider the equivariant Borel-Moore homology $H^{BM}_T(\Fl_{e_d})$.
%Below we will often write $e$ instead of $e_1$.

It turns out that $H^{BM}_T(\Fl_{e_d})$ also carries a bimodule structure but for a somewhat different algebra. Namely, let $T^\vee$ denote the Langlands dual torus. Consider the algebra $H^\times:=\C[T^* T^\vee]\# W$. The algebra $H^\times$ acts on $H^{BM}_T(\Fl_{e_d})$ from the left by what we call the CS (Chern-Springer) action, such an action exists for any homogenous affine Springer fiber, as the construction in Section \ref{SS_affine_Springer_reminder} shows.
%\cite[Theorem 1.2.6.(1)]{OY}\footnote{I've checked the reference and I'm not sure it's a correct one. It think this algebra should act just on any affine Springer fiber}.
For our particular choices of $e_d$ we also have a commuting $H^\times$-action that we call the ECM (equivariant-centralizer-monodromy) action.

Now we explain a relation between $H$ and $H^\times$.
For $G=\GL_n$, the algebra $H^\times$ is a localization of $H$. In the general case, the algebras $H$ and $H^\times$ share a common ``completion''. Namely, set \begin{equation}\label{completion:def}
    H^\wedge:= H\otimes_{\C[\h]}\C[\h]^{\wedge_0},
\end{equation}
where $\C[\h]^{\wedge_0}$ is the completion of $\C[\h]$ at zero. The same algebra $H^\wedge$ arises as $H^\times\otimes_{\C[T]}\C[T]^{\wedge_1}$, where we identify $\C[\h]^{\wedge_0}$ with $\C[T]^{\wedge_1}$
by means of $\exp: \h\rightarrow T$.
Then we consider
$$B_d^\wedge:=B_d\otimes_{\C[\h]}\C[\h]^{\wedge_0},$$
an $H^\wedge$-bimodule, as well as
$$H^T_{BM}(\Fl_{e_d})^\wedge:=H^T_{BM}(\Fl_{e_d})\otimes_{\C[T]}\C[T]^{\wedge_1},$$
also an $H^\wedge$-bimodule.

\begin{Thm}\label{Thm:iso}
There is an $H^\wedge$-bilinear
isomorphism $B^\wedge_d\xrightarrow{\sim} H^{BM}_T(\Fl_{e_d})^\wedge$.
\end{Thm}

Note that both sides are graded: $H^{BM}_T(\Fl_{e_d})^{\wedge}$ is graded by the homological degree, and $B_d^\wedge$ from a $\C^\times$-equivariant structure on $\Pro^{reg,*}\otimes\mathcal{O}^{reg}(d)\otimes \Pro^{reg}$ that will be explained in
Section \ref{SS_Cherednik_deform_constr}. We will see below that one can achieve that the isomorphism in Theorem \ref{Thm:iso} is grading preserving.

Now we explain how Theorem \ref{Thm:iso} relates to the previous work. In \cite{Kivinen}, Kivinen studied the spherical version of $\Fl_{e_d}$ and  proved a spherical version of Theorem \ref{Thm:iso} in the case of  $G=\operatorname{GL}_n$.
``Spherical'' means that $B_d$ is replaced with $\epsilon B_d$
for the trivial idempotent $\epsilon$ in $\C W=\C S_n$. On the level of Springer fibers, this means
that we take the Springer fiber in the affine Grassmannian instead of the affine flag variety.
Also note that Kivinen works with localizations, which is only possible for $G=\operatorname{GL}_n$.
In fact, one can prove an analog of Theorem \ref{Thm:iso} for localizations using the methods of this paper but we are not going to discuss this. In fact, one can prove a version of Theorem \ref{Thm:iso} for $B_d$ itself and a suitable modification of $\Fl_{e_d}$, but it will be proved elsewhere.

The bimodule $B_1$ for $G=\operatorname{GL}_n$ also appears in the recent paper of Carlsson and
Mellit, \cite[Conjecture 3.7]{CM}. We will deduce that conjecture from Theorem \ref{Thm:iso} combined with other statements that are used in its proof in Section \ref{SS_proofs_type_A}.

%We also would like to mention that Theorem \ref{Thm:iso} can be viewed as an instance of
%``coherent-constructible correspondence''. This type of correspondences is of great importance
%in the geometric representation theory, for example in Langlands duality.

Here is another important  application of Theorem \ref{Thm:iso}. Let $\C_{triv}$ denote
the one-dimensional irreducible representation of $H^\wedge$, where $\h$ and $\h^*$ act by $0$ and $W$ acts via the trivial representation.

\begin{Thm}\label{Thm:dim}
We have
$$\dim B_d\otimes_{H}\C_{triv}=
\dim H_T^{BM}(\Fl_{e_d})\otimes_{H^\times}\C_{triv}=(dh+1)^{\dim \h},$$
where $h$ denotes the Coxeter number of
$W$. Moreover, as a $W$-module, $B_d\otimes_{H}\C_{triv}$ is isomorphic to $\C(\Lambda_0/(dh+1)\Lambda_0)$, where we write $\Lambda_0$ for the co-root lattice.
\end{Thm}

In fact, we show that the first dimension is $\geqslant (dh+1)^{\dim \h}$, while the second dimension is $\leqslant (dh+1)^{\dim \h}$. The latter is done by using an argument similar to one in \cite{BBASV}.

Now we explain a reason to be interested in $H^{BM}_T(\Fl_{e_d})$.
It is expected that for $d=1$ this bimodule  is closely related to  the center of the principal block of the small quantum group $\mathfrak{u}_\epsilon(\g^\vee)$, where $\epsilon$ is an odd root of unity,  \cite{BBASV}.  We remark that
$$(H^T_{BM}(\Fl_{e_1})^{\wedge}\otimes_{\C[T^*T^\vee]^\wedge}\C_{triv})^*=H^*(\Fl_{e_1})^\Lambda,$$
where $\Lambda$ stands for the character lattice of $T$. Let $G^\vee$ denote the Langlands
dual group and $T^\vee$ is its maximal torus. Let $Z$ denote the center of
$\mathfrak{u}_\epsilon(\g^\vee)$. The group $G^\vee$ acts on $Z$ by algebra automorphisms.
The main conjecture of \cite{BBASV} relates the subalgebra $Z^{T^\vee}$ of $Z$
to the cohomology of $\Fl_{e_1}$ (there are also connections of the equivariant cohomology to the center but we are not going to discuss that).
Namely, it is conjectured in \cite{BBASV} that $Z^{T^\vee}$ is isomorphic to
$H^*(\Fl_e)^\Lambda$.
%, $Z^{G^\vee}$ coincides with
%the image of $H^*(\Fl)$ in %$H^*(\Fl_e)^\Lambda$ under the pullback map.
Modulo
the conjecture from \cite{BBASV}, Theorem \ref{Thm:dim} shows
that the dimension of the $W$-invariant part in $Z^{T^{\vee}}$ has dimension
$(h+1)^{\dim \h}$.

For $G=\operatorname{SL}_n$, we can say more.
Using Theorem \ref{Thm:iso} combined with Haiman's $n!$ theorem, \cite{Haiman}, one can show that,
modulo the conjecture from \cite{BBASV}, $W$ acts trivially on $Z^{T^\vee}$. This implies that
$G^\vee$ acts trivially on $Z$, so $\dim Z=(n+1)^{n-1}$. This will confirm a conjecture
from \cite{LY}. See Section \ref{SS_proof_center} for details.

%are expected  very closely  related to the center of the principal block in the small quantum group %$\mathfrak{u}_q(\g)$. We will see that in type A these two spaces coincide,
%which has important consequences for the center.

Now we explain two key ideas of the proof of Theorem
\ref{Thm:iso}. First, unsurprisingly, we use the induction on $d$. Our second, and main, idea
is to use a one-parameter deformation: it turns out that we can deform both $B_d$ and $H^T_{BM}(\Fl_{e_d})$.
For a complex number $c$, we can consider the rational Cherednik algebra $H_{\hbar,c}$ over $\C[\hbar]$, see
Section \ref{SS_RCA}, deforming $H$, and the trigonometric Cherednik algebra $H^\times_{\hbar,c}$,
see Section \ref{SS_TCA}, deforming $H^\times$.
The $H$-bimodule $B_d$ deforms to a bimodule over $H_{\hbar,d}$ (acting on the left) and $H_{\hbar,0}$ (acting on the right).
This is achieved by quantizing the Procesi bundle $\Pro^{reg}$ and the  line bundle
$\mathcal{O}^{reg}(1)$. The bimodule
$H^T_{BM}(\Fl_{e_d})$  deforms
to a bimodule over $H^\times_{\hbar,d}$
and $H^\times_{\hbar,0}$. The deformation in this case is done by considering the equivariant BM homology for
$T\times \C^\times$, where $\C^\times$ acts by the loop rotation. Note that, for each $c\in \C$, the algebras $H_{\hbar,c}, H^\times_{\hbar,c}$ share common partial completions (at $0$ and $1$, respectively).
We will see that we have a deformed version of the isomorphism from Theorem \ref{Thm:iso}, which turns out to be easier to establish.

In fact, the representations of rational  Cherednik algebras appeared in the context
of affine Springer theory previously, \cite{OY}. In particular, for a suitable
``elliptic'' element $e'_d$ (different
from $e_d$), it was shown that $H^{BM}_{\C^\times}(\Fl_{e'_d})$ admits a filtration with an action of $H^\times_{\hbar,d+1/h}$ on the associated graded space turning  $\operatorname{gr}H^{BM}_{\C^\times}(\Fl_{e'_d})$ into a deformation
of the unique irreducible finite dimensional module of the quotient $H_{d+1/h}:=
H_{\hbar,d+1/h}/(\hbar-1)$. Our construction and techniques that go into the proof
of the main result are very different from those of \cite{OY}. 
%The dimension of this module is also $(dh+1)^{\dim \h}$.
%Techniques that go into the proof of the deformed version of Theorem \ref{Thm:iso}
%are very different from those of \cite{OY}. Also,
%while $\dim H^T_{BM}(\Fl_{e_d})\otimes_{H^\times}\C_{triv}$ has the right dimension, it does not
%seem to be directly related to $H^{BM}_{\C^\times}(\Fl_{e'_d})$. It is an interesting question
%to establish a formal connection between $H^{BM}_{\C^\times}(\Fl_{e'_d})$ and
%a suitable modification of $H_{T\times \C^\times}^{BM}(\Fl_{e_d})$.

We finish the introduction by describing the content of the paper. In Section \ref{S_Hilbert_background} we discuss generalities on partial Poisson resolutions of $Y=(\h\oplus \h^*)/W$, Procesi sheaves on them and rational and trigonometric Cherednik algebras. This section mostly contains known results and their easy modifications.

In Section \ref{S_Procesi_bimodule_deformation} we construct a deformation of $B_d$. A key result used in the construction is  that the pushforward from $X^{reg}$ to $X$ of the vector bundle $\Pro^{reg,*}\otimes \Ocal^{reg}(d)\otimes \Pro^{reg}$ is a Cohen-Macaulay sheaf without higher cohomology. Two key ingredients for this result are the construction of the Procesi sheaves via quantizations in characteristic $p$ and the following claim of independent interests: the pushforward to $X$ of a line bundle on $X^{reg}$ is Cohen-Macaulay.

In Section \ref{S_BM_background} we provide some background on the equivariant Borel-Moore homology and equivalued unramified affine Springer fibers. This section does not contain any new results.

In Section \ref{S_BM_actions} we construct actions of $H^\times_{\hbar,d}, H^\times_{\hbar,0}$ on $H^{BM}_{T\times \C^\times}(\Fl_{e_d})$ and establish some properties of the resulting bimodule. A key technique is the localization theorem for equivariant BM homology. This section relies
on Appendix by the authors and Kivinen to check the relations for the action of $H^\times_{\hbar,d}$.

In Section \ref{S_proof_1} we prove Theorems \ref{Thm:iso} and \ref{Thm:dim}. And then in Section \ref{S_small_quantum} we discuss applications of own main results to conjectures of Carlsson and Mellit and to the center of the small quantum group.

{\bf Acknowledgements}. We would like to thank Roman Bezrukavnikov, Erik Carlsson, Evgeny Gorsky, Oscar Kivinen,
and Yoshinori Namikawa  for stimulating discussions. We also thank Oscar for working on the appendix together
with the authors. The work of P.B.A. was partially supported by the NSF under grant DMS-1926686.
The work of I.L. was partially supported by the NSF
under grant DMS-2001139.

\section{Procesi sheaves and Cherednik algebras}\label{S_Hilbert_background}
In this section we recall various generalities related to the algebras $H=\C[T^*\h^*]\#W,
H^\times=\C[T^*T^\vee]\#W$, their
deformations -- the rational and trigonometric Cherednik algebras,
and the bimodule $B_d$. In particular, we discuss a partial resolution $X$
of $Y$, and Procesi sheaves on $X$.

\subsection{Partial Poisson resolutions of $Y$}\label{SS_Y_partial_resol}
Let $Y=(\h\oplus \h^*)/W$. The goal of this section is to construct a partial Poisson resolution $X$ of $Y$ mentioned in the introduction.

The variety $Y$ is a conical symplectic singularity. As such, it admits a $\Q$-factorial terminalization, to be denoted by $\tilde{X}$, see \cite{BCHM} or \cite[Section 2.2]{SRA_der}. This is another, generally, singular symplectic variety together with a projective birational morphism
$\rho: \tilde{X}\rightarrow Y$. The variety $\tilde{X}$ is $\Q$-factorial and has terminal singularities. In particular, $\operatorname{codim}_{\tilde{X}}\tilde{X}^{sing}\geqslant 4$,
\cite{Namikawa_symplectic}. We remark that $\tilde{X}$ is not unique.

Note that $Y$ carries a natural action of $(\C^\times)^2$, by dilations of $\h$ and
of $\h^*$. This action lifts to $\tilde{X}$ making $\rho$ equivariant, compare to
\cite[Proposition A.7]{Namikawa_flop}. We will also consider the contracting torus $\{(t,t)| t\in \C^\times\}\subset (\C^\times)^2$.
The Poisson bracket on $\Ocal_{\tilde{X}}$ has weight $-2$ with respect to the action of this torus.

We will need to understand the structure of the exceptional divisor $D$ of $\tilde{X}\rightarrow Y$. For each irreducible component of this divisor, its image in $Y$ is the closure of a codimension $2$ leaf, see the proof of \cite[Proposition 2.14]{orbit_method}.
Such leaves are in bijection with conjugacy classes reflections in $W$.
All formal slices to these leaves in $Y$ are of type $A_1$. Therefore the preimage of the closure of such a leaf is irreducible. So we get a bijection between the conjugacy classes of reflections in $W$ and the irreducible components of the exceptional divisor. For a reflection $s$ we write $D_s$ for the corresponding component. So in the class group we have $D=\sum D_s$, where the sum is taken over the representatives of conjugacy classes.

We proceed to defining a partial resolution $X$ of $Y$.

When $\g$ is simple and  simply laced, we set $X:=\tilde{X}$.
For example, for $\g=\mathfrak{sl}_n$, we get a slight modification of $\operatorname{Hilb}_n(\C^2)$, the Hilbert scheme of $n$ points in $\C^2$. Namely, this variety maps to $(\C^n\oplus \C^{n*})/S_n$ and our $X$ is the preimage of $(\h\oplus \h^*)/S_n$. Note that in this case $X$ is smooth (and symplectic).

Assume again that $\g$ is simple and simply laced.
Note that since $\tilde{X}$ is $\Q$-factorial, there is
$\ell>0$ such that the line bundle $\mathcal{O}(\ell D)$ on $\tilde{X}^{reg}$ extends to  a line bundle on $\tilde{X}$. The extension, also
denoted by $\Ocal(\ell D)$, is ample.

Now consider the case when $W$ is of type $B_n,F_4$ or $G_2$. In this case, there are two codimension $2$ symplectic leaves in $Y$, corresponding to the two conjugacy classes of reflections.  We consider the bundle $\mathcal{O}(D)$ on $\tilde{X}^{reg}$ associated to the divisor $D$.  Again, we can find $\ell$ such that $\mathcal{O}(\ell D)$ extends to $\tilde{X}$. But now
$\mathcal{O}(\ell D)$ may fail to be ample (for any choice of $\tilde{X}$). For example, this is the case in type $B_n$ for $n>1$. We choose $\tilde{X}$
so that this bundle lies in the closure of the ample cone of $\tilde{X}$.

\begin{Prop}\label{Prop:X_non_simply_laced}
There is an irreducible singular symplectic variety $X$ with projective birational morphisms
$\bar{\rho}:\tilde{X}\rightarrow X$ and $\underline{\rho}:X\rightarrow Y$ such that
\begin{itemize}
    \item[(i)] $\operatorname{codim}_{X}X^{sing}\geqslant 4$,
    \item[(ii)] For some $\ell>0$, the bundle $\mathcal{O}(\ell D)$ is lifted from an ample line bundle on $X$.
\end{itemize}
\end{Prop}
\begin{proof}
Intermediate partial resolutions between $\tilde{X}$ and $Y$ (that are normal, hence singular symplectic) are classified by faces of the ample
cone of $\tilde{X}$ in such a way that for a given face, $C_0$, for any rational point $\chi\in C_0$,
a positive rational multiple of $\chi$ is the 1st Chern class  of an ample line bundle  on the
corresponding partial resolution.  This follows, for example, \cite[Section 3-2]{KMM}. Note that in \cite{KMM} the result is stated in terms of the nef cone. By a theorem of Kleiman, in our situation, the nef cone is dual to the ample cone, and so we can talk about the ample cone instead.

In particular, we get a unique partial resolution $X$ satisfying (ii). We need to show that it satisfies
(i) as well. Assume the contrary: $\operatorname{codim}_{X}X^{sing}=2$. Since
$\operatorname{codim}_{\tilde{X}}\tilde{X}^{sing}=4$, an irreducible component of $\bar{\rho}^{-1}(X^{sing})$ is a divisor.
On the other hand, as argued in the proof of \cite[Proposition 2.14]{orbit_method},
the image in $Y$ of an irreducible divisor  under $\rho$ either intersects $Y^{reg}$ or coincides
with the closure of a codimension $2$ leaf. It follows that a codimension $2$ leaf in
$X$ maps to a codimension $2$ leaf in $Y$. This contradicts the claim that
some multiple of $D$ corresponds to an ample line bundle on $X$.
So $X$ satisfies (i) as well, which finishes the proof.
\end{proof}

We note that, by the construction, $(\C^\times)^2$ acts on $X$ and the morphisms
$\bar{\rho},\underline{\rho}$ are equivariant.
%We also note that the components
%of the exceptional divisor for $\underline{\rho}:X\rightarrow Y$ are in a natural
%bijection with those for $\rho:\tilde{X}\rightarrow Y$. We will abuse the notation
%and write $D$ for the exceptional divisor of $X\rightarrow Y$ and $D_s$ for its component
%corresponding to a reflection $s\in S$ (recall that $D_s$ recovers $s$ uniquely up to
%conjugacy).

\begin{Rem}\label{Rem:X_type_BC}
When $W$ is of type $B_n$, the varieties $\tilde{X}$ (which is actually smooth) and $Y$ can be realized as Nakajima quiver
varieties for the affine quiver of type $\tilde{A}_2$ with dimension vector $n\delta$
and unit framing at the extending vertex $0$. We can construct $X$ as a quiver variety
as well: the character defining stability can be shown to be $(0,1)$. We will not use
this observation.
\end{Rem}

\subsection{Procesi sheaves}\label{SS_Procesi_sheaves}
The goal of this section is to produce a Procesi sheaf on $X$. The case of Procesi sheaves on $\tilde{X}$ was handled in \cite[Section 4]{SRA_der}.

Let us recall the construction of the latter. We can reduce $\tilde{X}$ mod $p$ for $p\gg 0$. Namely, set $\F:=\overline{\F}_p$. Then we can define the reduction $\tilde{X}_\F$ to $\F$. Since $p$
is sufficiently large, $\tilde{X}_\F$ is a singular symplectic variety with $\operatorname{codim}_{\tilde{X}_\F}\tilde{X}_\F^{sing}\geqslant 4$ and vanishing higher cohomology of the structure sheaf. In \cite[Section 4.2]{SRA_der} the second named author  constructed a  filtered quantization $\mathcal{D}_{\F}$ of the structure sheaf $\Ocal_{\tilde{X}_\F}$ whose global sections is $\mathbb{A}(\h_\F\oplus \h_\F^*)^W$, where $\mathbb{A}$ stands for the Weyl algebra of a symplectic vector space. Consider the Frobenius
morphism $\Fr: \tilde{X}_\F\rightarrow \tilde{X}_\F^{(1)}$ and the pushforward $\Fr_* \mathcal{D}_\F$. The restriction of this sheaf of algebras to the regular locus is an Azumaya algebra,
\cite[Lemma 4.3]{SRA_der}. Consider the completion $\F[Y^{(1)}]^{\wedge_0}$
of $\F[Y^{(1)}]$ at $0$. We denote its spectrum by $Y^{(1)\wedge}_\F$. Consider the scheme
\begin{equation}\label{eq:X_wedge}
\tilde{X}_\F^{(1),\wedge}:=Y^{(1)\wedge}_\F\times_{Y^{(1)}_\F}\tilde{X}^{(1)}_\F.
\end{equation}
It was shown in
\cite[Section 4.3]{SRA_der} that the restriction of $\Fr_*\mathcal{D}_\F$ to the regular locus in
$\tilde{X}_\F^{(1),\wedge}$ splits. Moreover, it was shown there that
we can find a Morita equivalent sheaf of algebras $\mathcal{A}_\F$ on
$\tilde{X}_\F^{(1),\wedge}$ whose global sections
are $\F[\h^{(1)}\oplus \h^{(1)*}]^{\wedge_0}\# W$.
Let $\epsilon$ denote the averaging idempotent in
$\F W$. Set $\tilde{\Pro}^\wedge_\F:=\mathcal{A}_\F \epsilon$. Then the restriction of $\tilde{\Pro}_\F^\wedge$ to
$\tilde{X}_\F^{(1),\wedge,reg}$ is a splitting bundle for the Azumaya algebra  $$\mathcal{A}_\F|_{\tilde{X}_\F^{(1),\wedge,reg}}.$$ Also note that $\mathcal{A}_\F$ is a maximal Cohen-Macaulay sheaf that coincides with the endomorphism sheaf of $\tilde{\Pro}^\wedge_\F$. Note that, by the construction, we have $$\epsilon\tilde{\Pro}^{\wedge}_\F=
\Ocal_{\tilde{X}_\F^{(1),\wedge}}.$$

Consider the contracting $\F^\times$-action on $X_\F^{(1)\wedge}$. Then $\tilde{\Pro}^\wedge_\F$
can be shown to admit an $\F^\times$-equivariant structure. Using this, we can extend $\tilde{\Pro}^\wedge_\F$ to an $\F^\times$-equivariant maximal Cohen-Macaulay sheaf on $\tilde{X}^{(1)}_\F$ to be denoted by $\tilde{\Pro}_\F$, see \cite[Lemma 4.6]{SRA_der}.
By the same lemma, we can modify the $\F^\times$-equivariant structure on $\tilde{\Pro}_\F$ so that we get a graded algebra isomorphism $\operatorname{End}(\tilde{\Pro}_\F)\xrightarrow{\sim} \F[\h^{(1)}\oplus \h^{(1)*}]\# W$.

Finally, we can lift $\tilde{\Pro}_\F$ to characteristic $0$, \cite[Section 4.4]{SRA_der}.
We get a maximal Cohen-Macaulay sheaf $\tilde{\Pro}$ on $\tilde{X}$ with the following properties:
\begin{itemize}
    \item[(i)] We have a graded algebra isomorphism $\End(\tilde{\Pro})\xrightarrow{\sim} H$,
    \item[(ii)] $\mathcal{E}nd(\tilde{\Pro})$ is a maximal Cohen-Macaulay module,
    \item[(iii)] $H^i(\tilde{X}, \mathcal{E}nd(\tilde{\Pro}))=0$ for $i>0$.
    \item[(iv)] $\epsilon \tilde{\Pro}\xrightarrow{\sim}\Ocal_{\tilde{X}}$,
    a $\C^\times$-equivariant isomorphism.
\end{itemize}

Sheaves $\tilde{\Pro}$ satisfying (i)-(iv) are called {\it Procesi sheaves} on $\tilde{X}$.

We note that, for the same reason as in \cite[Lemma 4.6]{SRA_der},  $\tilde{\Pro}$ can also be made equivariant with respect to $(\C^\times)^2$ and the isomorphisms in (i) and (iv) can assumed to be $(\C^\times)^2$-equivariant. As remarked in \cite[Remark 4.8]{SRA_der}, the argument in \cite{Losev_Procesi} classifying the Procesi bundles in the smooth case carries over to the singular case. So the bundles $\tilde{\Pro}$ on $\tilde{X}$ satisfying (i)-(iv) are classified by the elements of the Namikawa-Weyl group of $Y$ introduced in \cite{Namikawa2}. We will denote this group by $W_Y$.
This group is $\prod_{s}(\Z/2\Z)$, where $s$ runs over representatives of conjugacy classes of reflections in $W$.
Below, Section \ref{SS_RCA}, we will recall how the classification of Procesi sheaves works.

To finish the section, we discuss Procesi sheaves on $X$. Recall the birational
projective morphism $\bar{\rho}:\tilde{X}\rightarrow X$ from
Proposition \ref{Prop:X_non_simply_laced}. Set
$$\Pro:=\bar{\rho}_* \tilde{\Pro}.$$

\begin{Lem}\label{Lem:Pro_X}
The sheaf $\Pro$ on $X$ has properties completely analogous to (i)-(iv).
\end{Lem}
\begin{proof}
First of all, note that $R\overline{\rho}_*
\Ocal_{\tilde{X}}=\Ocal_X$ because $\tilde{X},X$ are singular symplectic and $\overline{\rho}$
is birational and projective. For similar reasons,  $H^i(X, \Ocal_X)=0$ for all $i>0$.  So the same is true
over $\F$ (assuming, as always, that $p\gg 0$).
Therefore $R^i\overline{\rho}_* \mathcal{D}=0$
for $i>0$, the sheaf $\overline{\rho}_* \mathcal{D}$ is a filtered quantization of
$\Ocal_{X_\F}$ and has no higher cohomology.
From here we deduce that $\overline{\rho}_*\mathcal{A}_\F$ is a maximal Cohen-Macaulay sheaf without higher cohomology.
Moreover, $\overline{\rho}_* \tilde{\Pro}_\F^{\wedge}= \epsilon \overline{\rho}_*\mathcal{A}_\F$. Now note that $\Pro$
is obtained from $\overline{\rho}_* \tilde{\Pro}_\F^{\wedge}$ in the same way as
$\tilde{\Pro}$ is obtained from $\tilde{\Pro}_\F^{\wedge}$. It follows that the natural homomorphism $\operatorname{End}(\tilde{\Pro})\rightarrow \operatorname{End}(\Pro)$ is an isomorphism yielding (i).
Conditions (ii) and (iii) also follow, while (iv) is immediate from the construction of $\Pro$.
\end{proof}

\subsection{Rational Cherednik algebras}\label{SS_RCA}
Let us write $S$ for the set of reflections in
$W$. Let $c:S\rightarrow \C$ be a $W$-invariant function. Let $\hbar$ be an independent variable. Then we can define the rational Cherednik algebra $H_{\hbar,c}$ as the quotient of $T(\h\oplus \h^*)[\hbar]\# W$ by the following relations
$$[x,x']=[y,y']=0, [y,x]=\hbar \left(\langle y,x\rangle-\sum_{s\in S} c(s)\langle y,\alpha_s^\vee\rangle\langle x,\alpha_s\rangle s\right).$$
Here $x,x'\in \h, y,y'\in \h^*$ and $\alpha_s,\alpha_s^\vee$ denote the positive root and the positive coroot corresponding to a reflection $s$.  For example, $H_{\hbar,0}=D_\hbar(\h^*)\# W$, where we write
$D_\hbar(\h^*)$ for the algebra of homogenized differential operators on $\h^*$.

We will write $H_c$ for the specialization of $H_{\hbar,c}$ to $\hbar=1$.

Now we will discuss a connection between the rational Cherednik algebras and Procesi sheaves.
We start with the Procesi sheaves on $\tilde{X}$,
the case treated in \cite[Section 5.1]{SRA_der}.

The formal quantizations  of $\tilde{X}^{reg}$ with a compatible action of the contracting torus are classified by the points of $H^2(\tilde{X}^{reg},\C)$, \cite[Section 2.3]{quant_iso}.
We note that the 1st Chern class map induces an isomorphism $\C\otimes_{\Z}\operatorname{Pic}(\tilde{X}^{reg},\C)\xrightarrow{\sim} H^2(\tilde{X}^{reg},\C)$, both spaces have dimensions equal to the number of conjugacy classes of reflections in $W$.
The quantizations of $\tilde{X}$ are in a natural bijection with those of $\tilde{X}^{reg}$
via push-forward and pullback, see \cite[Proposition 3.4]{BPW}. Let us write $\tilde{\mathcal{D}}_{\hbar,\lambda}$ for the formal quantization of $\tilde{X}$ corresponding to $\lambda$.
Note that $\tilde{\Dcal}_{\hbar,\lambda}$ also has an action of the torus $(\C^\times)^2$, and the action of the Hamiltonian subtorus $\{(t,t^{-1})| t\in \C^\times\}$ is still Hamiltonian.

The algebra of global sections $\Gamma(\tilde{\Dcal}_{\hbar,\lambda})$ is related to the rational Cherednik algebra $H_{\hbar,c}$
as follows. Consider the spherical subalgebra $\epsilon H_{\hbar,c}\epsilon$, a graded quantization of $\C[Y]$.  We can consider the subalgebra
$\Gamma(\tilde{\Dcal}_{\hbar,\lambda})^{fin}$ of $\C^\times$-locally finite elements in
$\Gamma(\tilde{\Dcal}_{\hbar,\lambda})$ with respect to the contracting $\C^\times$-action. Then we have  $$\Gamma(\tilde{\Dcal}_{\hbar,\lambda})^{fin}\cong  \epsilon H_{\hbar,c_\lambda} \epsilon,$$ where $c_\lambda$ is computed as follows.
 The Chern classes of the line bundles $\mathcal{O}(D_s)$ form a basis in $H^2(\tilde{X}^{reg},\C)$. Let $\lambda_s$ be the coefficient of the basis element corresponding to $s$ in $\lambda$.

\begin{defi}\label{defi:c_from_lambda}
By definition, $c_\lambda$ sends $s\in S$ to $\lambda_s-\frac{1}{2}$.
\end{defi}

Note that the Namikawa-Weyl group $W_Y$ acts on $H^2(\tilde{X}^{reg},\C)$ by changing signs of the coordinates $\lambda_s$, this follows, for example, from
\cite[Section 3.6]{orbit_method}. In particular, we get a $W_Y$-action on the affine space of parameters $c$. Two $W_Y$-conjugate parameters give rise to the same algebra $\Gamma(\Dcal_{\hbar,\lambda})^{fin}$,
\cite[Proposition 3.10]{BPW}.

Now we discuss a connection of rational Cherednik algebras with Procesi sheaves,
established in \cite{quant_iso} in a special case and in \cite{SRA_der} in the general case. See,
in particular, \cite[Section 5.1]{SRA_der}. Let $\tilde{\Pro}^{reg}$ denote the restriction of $\tilde{\Pro}$ to $\tilde{X}^{reg}$, this is a vector bundle.  Let $\tilde{\Dcal}_{\hbar,\lambda}^{reg}$ be the restriction of $\tilde{\Dcal}_{\hbar,\lambda}$. Since $\mathcal{E}nd(\tilde{\Pro})$ is a maximal Cohen-Macaulay module that has no higher cohomology,
see Section \ref{SS_Procesi_sheaves}, we see that $\Gamma(\mathcal{E}nd(\tilde{\Pro}^{reg}))=H$
and $H^i(\mathcal{E}nd(\tilde{\Pro}^{reg}))=0$
for $i=1,2$. So we have a unique quantization of $\tilde{\Pro}^{reg}$ to a left $\tilde{\Dcal}_{\hbar,\lambda}^{reg}$-module to be denoted by $\tilde{\Pro}_{\hbar,\lambda}^{reg}$. This quantization is $(\C^\times)^2$-equivariant. Set
\begin{equation}\label{eq:ecal_reg_defn} \tilde{\Ecal}^{reg}_{\hbar,\lambda}:=
\mathcal{E}nd_{\tilde{\Dcal}_{\hbar,\lambda}^{reg}}
(\tilde{\Pro}_{\hbar,\lambda}^{reg})^{opp}.
\end{equation}
This is a sheaf of $\C[[\hbar]]$-algebras on $\tilde{X}^{reg}$ deforming $\mathcal{E}nd(\tilde{\Pro}^{reg})^{opp}$.

As was argued in
\cite[Section 5.1]{SRA_der}, \begin{equation}\label{eq:Procesi_RCA_iso}\Gamma(\tilde{\Ecal}^{reg}_{\hbar,\lambda})^{fin}\xrightarrow{\sim} H_{\hbar,c'(\lambda)}
\end{equation}
for an affine map $\lambda\mapsto c'(\lambda)$.
By multiplying the source and the target of (\ref{eq:Procesi_RCA_iso}) with $\epsilon$ on the left and on the right, we get $\Gamma(\tilde{\Dcal}_{\hbar,\lambda}^{reg})^{fin}\xrightarrow{\sim} \epsilon H_{\hbar,c'(\lambda)} \epsilon$ that gives the identity endomorphism of
$\C[Y]$ after taking the quotient by $\hbar=0$. So, as argued in \cite[Section 5.1]{SRA_der},
we have an element $w\in W_Y$ such that $c'(\lambda)=w c_\lambda$. This element depends on the choice of $\tilde{\Pro}$. This defines a bijection between the set of possible Procesi bundles and $W_Y$, already mentioned in Section
\ref{SS_Procesi_sheaves}. This was proved in \cite[Theorem 1.1]{Losev_Procesi} in the case when $\tilde{X}$ is smooth and carries over to the general case verbatim.

We will always choose $\tilde{\Pro}$ corresponding to the unit element in $W_Y$.

Let $\epsilon_-$ denote the sign idempotent in $\C W$. Using the previous discussion we can describe
$\tilde{\Pro}^{reg} \epsilon_-$, the sign component of
$\tilde{\Pro}^{reg}$.

\begin{Lem}\label{Lem:sgn_component}
We have $c_1(\tilde{\Pro}^{reg}\epsilon_-)=\frac{1}{2}c_1(\mathcal{O}(D))$.
\end{Lem}
\begin{proof}
Let $s$ be a reflection in $W$. Pick a
point $y\in Y$ lying in the symplectic leaf corresponding to $s$. We set $Y^{\wedge_y}:=\operatorname{Spec}(\C[Y]^{\wedge_y})$ and $X^{\wedge_y}:=Y^\wedge_y\times_{Y}\tilde{X}$.

Pick $\lambda\in H^2(\tilde{X}^{reg},\C)$ and set $c:=c_\lambda$.
We can also consider the completion $H_{\hbar,c}^{\wedge_y}$. As was checked in
\cite[Section 3.3]{BE}, this is a matrix algebra of size $|W|/2$ over $\underline{H}_{\hbar,c(s)}^{\wedge_0}$ the  completed rational Cherednik algebra for
$(\h, \langle s \rangle)$ with parameter $c(s)$.
On the other hand, analogously to \cite[Proposition 4.1]{Losev_Procesi}, $\tilde{\Pro}^{\wedge_y}:=\tilde{\Pro}|_{\tilde{X}^{\wedge_y}}$, coincides with
$\Hom_{\C\{1,s\}}(\C W, \underline{\Pro}^{\wedge_0})$, where $\underline{\Pro}$ is the Procesi bundle over $(\h\oplus \h^*)^s\times T^*\mathbb{P}^1$.
Let $i:
X^{\wedge_y}\rightarrow \tilde{X}^{reg}$ be the embedding.
Then we have the pullback map
$$i^*:H^2(\tilde{X}^{reg},\C)=H^2_{DR}(\tilde{X}^{reg})\rightarrow H^2_{DR}(X^{\wedge_y})
\cong H^2\left((\h\oplus \h^*)^s\times T^*\mathbb{P}^1,\C\right)=\C$$
The isomorphism $\End(\Pro_{\hbar,\lambda})\xrightarrow{\sim}
H^{\wedge_\hbar}_{\hbar,\lambda}$ gives rise to
an isomorphism $\End(\underline{\Pro}^{\wedge_0}_{\hbar,i^*(\lambda)})\xrightarrow{\sim}
\underline{H}^{\wedge_0}_{\hbar,c(s)}$.
By Definition \ref{defi:c_from_lambda}, the isomorphism of parameter spaces corresponding to $\underline{\Pro}$ sends $i^*(\lambda)$ to
$i^*(\lambda)-\frac{1}{2}$. The two possibilities for $\underline{\Pro}$ are $\Ocal\oplus \Ocal(1)$
and $\Ocal\oplus \Ocal(-1)$. The map between the parameter spaces we have is realized by the former.
This is an easy special case of \cite[Section 4.2]{Losev_Procesi}, for example.

In particular, using the direct analog of \cite[Proposition 4.1]{Losev_Procesi} again, we see that the restriction of the line bundle $\Pro\epsilon_-$
to $X^{\wedge_y}$ is $\mathcal{O}(1)$ on that scheme. Since the line bundle $\mathcal{O}(\mathbb{P}^1)$ on $T^*\mathbb{P}^1$ is $\mathcal{O}(2)$, the claim of the lemma follows.
\end{proof}

Now we explain how to relate the rational Cherednik algebras to quantizations of $\tilde{\Pro}$
(instead of $\tilde{\Pro}^{reg}$). Let $\iota$ denote the embedding $\tilde{X}^{reg}\hookrightarrow \tilde{X}$.
Set $\tilde{\Pro}_{\hbar,\lambda}:=\iota_*(\tilde{\Pro}_{\hbar,\lambda}^{reg})$. Since $H^1(\tilde{X}^{reg},\tilde{\Pro}^{reg})=0$, we see that $\tilde{\Pro}_{\hbar,\lambda}$ is a quantization of $\tilde{\Pro}$. Similarly, $\tilde{\Ecal}_{\hbar,\lambda}:=\iota_* \tilde{\Ecal}_{\hbar,\lambda}^{reg}$ is the endomorphism sheaf of $\tilde{\Pro}_{\hbar,\lambda}$ (with opposite multiplication).

Let us proceed to quantizations of $\Pro$, the Procesi sheaf on $X$, and its endomorphism sheaf.
Similarly to the proof of Lemma \ref{Lem:Pro_X}, we see that $R^i\overline{\rho}_* \tilde{\Ecal}=0$ for $i>0$. So we get that  $\Ecal_{\hbar,\lambda}:=\overline{\rho}_*
\tilde{\Ecal}_{\hbar,\lambda}$ is a quantization of $\mathcal{E}nd(\Pro)$. Further, we set $\Pro_{\hbar,\lambda}:=\epsilon \Ecal_{\hbar,\lambda}$. This is a quantization of $\Pro$. %Moreover, since $\Pro=\epsilon \Ecal$ we get
%$\Pro_{\hbar,\lambda}=\epsilon \Ecal_{\hbar,\lambda}$.
Also $$\Ecal_{\hbar,\lambda}=
\mathcal{E}nd_{\Dcal_{\hbar,\lambda}}(\Pro_{\hbar,\lambda})^{opp},$$
where $\Dcal_{\hbar,\lambda}$ is the pushforward of $\tilde{\Dcal}_{\hbar,\lambda}$ to $X$.

%In the remainder of the section, we assume that $\g$ is simple. The general case will easily follow from here.  Recall the divisors $D$ (the simply laced case) and $D_1,D_2$ (the non-simply laced case)
%introduced in Section %\ref{SS_Y_partial_resol}.

In what follows we will write $\mathcal{O}^{reg}(1):=\Pro^{reg}\epsilon_-$.
This is a line bundle on $X^{reg}$.

Note that $\bar{\rho}^*$ induces an isomorphism $\operatorname{Pic}(X^{reg})
\xrightarrow{\sim}\operatorname{Pic}(\tilde{X}^{reg})$. This allows us to view $c_1(\mathcal{O}^{reg}(1))$  as an element of $H^2(\tilde{X}^{reg},\C)$.
If $\lambda\in H^2(\tilde{X}^{reg},\C)$ corresponds to a Cherednik parameter $c=c_\lambda$, then the Cherednik parameter, say $c'$, corresponding to $\lambda+c_1(\mathcal{O}^{reg}(1))$ satisfies $c'(s)=c(s)+1$ for all $s\in S$.

\subsection{Trigonometric Cherednik algebras}\label{SS_TCA} In this section we will discuss the trigonometric Cherednik algebras and their connection to rational Cherednik algebras.
Assume that $G$ is reductive. Recall that $T$ denotes a maximal torus.

Let $\Lambda$ denote the co-character lattice of $T$ and $\Lambda_0$ be the co-root lattice of $\g$, a sublattice of $\Lambda$. Consider the extended affine Weyl group $\widetilde{W}:=W\ltimes \Lambda$ that contains the affine Weyl group $W^a:=W\ltimes \Lambda_0$ as a normal subgroup. We have the length function
$\ell: \widetilde{W}\rightarrow \Z_{\geqslant 0}$.
The subgroup of length $0$ elements is identified with $\Lambda/\Lambda_0$ under the projection $\widetilde{W}\twoheadrightarrow \widetilde{W}/W^a\cong \Lambda/\Lambda_0$. We have the decomposition $\widetilde{W}=(\Lambda/\Lambda_0)\ltimes W^a$.

The group $\widetilde{W}$ acts on $\Lambda\times \mathbb{Z}$ by
\begin{equation}\label{eq:affine_Weyl_action_lattice}
w(\mu,a):=(w\mu,a), \chi(\mu,a):=(\mu+ a\chi, a), \quad\chi\in \Lambda\subset \widetilde{W},w\in W\subset \widetilde{W}, \mu\in \Lambda, a\in \Z.
\end{equation}
We consider the dual action of $\widetilde{W}$ on
$\h^*\oplus \C$. It is given by
$$w(y,z)=(wy,z), \chi(y,z)=(y,z+\langle \chi,y\rangle),
\quad\chi\in \Lambda,w\in W, y\in \h^*, z\in \C.$$

Let $s_1,\ldots,s_r$ denote the simple reflections in $W$ and $s_0$ denote the simple affine reflection. Let $\alpha^\vee_1,\ldots,\alpha^\vee_r$ denote the simple co-roots and $\alpha^\vee_0$ denote the minimal (negative) co-root. Pick a $W$-invariant function $c:S\rightarrow \C$.
Set $c(s_0):=c(s_{\alpha_0})$.

The trigonometric Cherednik algebra $H^\times_{\hbar,c}$ is defined as the algebra generated by two subalgebras:
$\C \widetilde{W}$ and $\C[\h,\hbar]$ subject to the following cross-relations
\begin{equation}\label{eq:trig_CA_relns}
\begin{split}
& s_i y-(s_i.y)s_i=\hbar c(s_i)\langle y,\alpha_i^\vee\rangle, \quad i=0,\ldots,r, y\in \h^*,\\
& \pi y=(\pi.y)\pi, \quad y\in \h^*, \pi\in \Lambda/\Lambda_0\subset \widetilde{W}, \\
& x \hbar=\hbar x,\quad x\in \widetilde{W}.
\end{split}
\end{equation}
Here we write $x.y$ for the image of $y\in \h^*$ under $x\in \widetilde{W}$ for  the action of $\widetilde{W}$ on $\h\oplus \C$ described above (with $\hbar$ corresponding to $1\in \C$).

The algebra $H^\times_{\hbar,c}$ admits an embedding into the algebra $D_\hbar(T^{\vee,reg})\# W$, where $T^{\vee,reg}$ denotes the complement to the union of root codimension $1$ subgroups in the Langlands dual torus $T^\vee$;  we write $D_\hbar$ for the algebra of homogenized differential operators. Namely, let us write $e^\lambda$ for the function on $T^\vee$ given by $\lambda$. The embedding maps $\lambda\in \Lambda\subset \widetilde{W}$ to $e^\lambda$, $w\in W$ to $w\in W$, $\hbar$ to $\hbar$ and $y\in \h$ to the trigonometric Dunkl operator (see \cite[(2.12.16)]{Cherednik}):

$$D^{trig}_y=\partial_y+\sum_{\alpha>0}\hbar c(s_\alpha)\frac{\langle\alpha,y\rangle}{1-e^{-\alpha^\vee}}(1-s_\alpha)-\langle\sum_{\alpha>0}\hbar c(s_\alpha)\alpha^\vee,y\rangle. $$

This embedding can be used to establish the following well-known result that plays an important role in our paper.

\begin{Lem}\label{Lem:trig_rat_iso}
We have an algebra isomorphism
$$H_{\hbar,c}\otimes_{\C[\h^*]}\C[\h^*]^{\wedge_0}\cong H^\times_{\hbar,c}\otimes_{\C[T^\vee]}\C[T^\vee]^{\wedge_1}$$
\end{Lem}
\begin{proof}
We can identify $\C[\h^*]^{\wedge_0}\cong \C[T^\vee]^{\wedge_1}$ by sending $x\in \h$ to $e^x$. This identification is $W$-equivariant. It remains to show that the subalgebra in $D_\hbar(\h^{*,\wedge_0,reg})\# W$
generated by $\C[\h^*]^{\wedge_0}\# W$
and the rational Dunkl operators coincides with the subalgebra generated by $\C[\h^*]^{\wedge_0}\# W$ and the trigonometric Dunkl operators. This is because the difference between the trigonometric and rational
Dunkl operators associated to $y\in \h^*$ lies in $\C[\h^*]^{\wedge_0}\# W$. The latter subalgebra
lies in both images.
\end{proof}

\subsection{Representation theory of rational Cherednik algebras}
In this section we will recall several known constructions and facts related to the representation theory of rational Cherednik algebras. Set $H_c:=H_{\hbar,c}/(\hbar-1)$, this is a filtered deformation of $\C[T^*\h^*]\# W$. Let $\epsilon,\epsilon_-$
denote the trivial and sign idempotents in
$\C W$.

We abuse the notation and write $c+1$ for the map $S\rightarrow W$ sending $s$ to
$c(s)+1$.
We start with the following classical result,
see, e.g., \cite{BEG_HC}, that will also be established below,
Lemma \ref{Lem:bimod_sph_isom}.

\begin{Lem}\label{Lem:translation}
We have a filtered algebra isomorphism
$\epsilon H_c \epsilon\cong \epsilon_- H_{c+1}\epsilon_-$ that is the identity on the associated
graded algebras.
\end{Lem}

We say that a parameter $c$ is {\it $\epsilon$-spherical} if $H_c=H_c\epsilon H_c$. In this case the categories $H_c\operatorname{-mod}$ and $\epsilon H_c \epsilon\operatorname{-mod}$ are equivalent via the bimodules $H_c\epsilon, \epsilon H_c$. The following result, due to Bezrukavnikov, is \cite[Theorem 5.5]{Etingof_affine}.

\begin{Prop}\label{Prop:spher_crit}
The parameter $c$ is $\epsilon$-spherical if and only if the algebra $\epsilon H_c \epsilon$ has finite homological dimension.
\end{Prop}

Similarly, we can talk about $\epsilon_-$-spherical parameters. A complete analog of Proposition \ref{Prop:spher_crit} holds. In particular, we can use Lemma \ref{Lem:translation} to prove the following result.

\begin{Cor}\label{Cor:spher_reln}
The parameter $c$ is $\epsilon$-spherical if and only if $c+1$ is $\epsilon_-$-spherical.
\end{Cor}

We will be interested in two classes of parameters $c$. The first class is the parameters $c$ with $c(s)\in \Z$ for all
$s$. The following result was obtained in
\cite[Corollary 4.2]{BEG_HC}.

\begin{Lem}\label{Lem:integr_spher}
If $c(s)\in \Z$ for all $s$, then the algebra $H_c$ is simple. In particular,
$c$ is both $\epsilon$- and $\epsilon_-$-spherical.
\end{Lem}

The second class of parameters is as follows. Assume that $\g$ is simple. Let $h$ denote the Coxeter number for $W$ and let $d\in \Z_{\geqslant 0}$.
We consider constant functions $c:S\rightarrow \C$ such that $c(s)=d+\frac{1}{h}$ for all $s\in S$.
%Recall that it makes sense to speak about the category $\mathcal{O}_c$ of $H_c$-modules: it consists of all modules that are finitely generated over $\C[\h]$ and have locally nilpotent action on $\h$,
%see \cite{GGOR}. The simple modules in
%$\Ocal_c$ are parameterized by $\operatorname{Irr}(W)$. For an irreducible representation $\tau$ of $W$ we write
%$L_c(\tau)$ for the corresponding simple %module in $\mathcal{O}_c$.

%Recall that we take $c$ with $c(s)=\frac{1}{h}+d$.
The  following result was obtained in \cite[Theorem 1.4, Proposition 1.7]{BEG}.

\begin{Prop}\label{Prop:BEG_comput}
There is a unique irreducible finite dimensional $H_c$-module. Its dimension is $(dh+1)^{\dim\h}$. Moreover, as a $W$-representation, it is isomorphic to the permutation module $\C(\Lambda_0/(dh+1)\Lambda_0)$, where, recall, $\Lambda_0$ is the co-root lattice.
\end{Prop}

The following is
\cite[Lemma 4.5]{Gordon}.

\begin{Prop}\label{Prop:spheric_param}
The parameter $c=d+\frac{1}{h}$ is $\epsilon$-spherical for $d\geqslant 0$
and $\epsilon_-$-spherical for $d>0$.
\end{Prop}

%Note that, thanks to Lemma \ref{Lem:translation}, we can view the space $\epsilon H_{c+1} \epsilon_-$ as an
%$\epsilon H_{c+1}\epsilon$-$\epsilon H_c \epsilon$ bimodule and the space
%$\epsilon_- H_{c+1} \epsilon$ as a
%$\epsilon H_c \epsilon$-$\epsilon H_{c+1}\epsilon$-bimodule. The following is a consequence of Proposition \ref{Prop:spheric_transl}.

%\begin{Cor}\label{Cor:translation_equiv}
%For $c=d=\frac{1}{h}$ with $d\geqslant 0$, the bimodules $\epsilon H_{c+1} \epsilon_-$
%and $\epsilon_- H_{c+1}\epsilon$ are mutually inverse Morita equivalences.
%\end{Cor}

%Let $\mathcal{H}_q$
%Using the KZ functor from \cite{GGOR} and %the fact the Hecke algebra of $W$ has
%$|\operatorname{Irr}(W)|-1$ irreducible %representations at parameter $q=\exp()$

%We will need to use a corollary of this lemma relating $H_c$-modules to certain microlocal sheaves on $X$, a ``localization theorem''.

\section{Deformation of $B_d$}\label{S_Procesi_bimodule_deformation}
Let $d$ be a positive integer.
The goal of this section is, for a Cherednik parameter $c$, to define an $H_{\hbar,c+d}$-$H_{\hbar,c}$-bimodule $B_{\hbar, c+d\leftarrow c}$ that is
a $\C[\hbar]$-flat deformation of $B_d$.  This is done
in Section \ref{SS_Cherednik_deform_constr}. This construction is based
on two algebro-geometric results of independent interest,
Propositions \ref{Prop:cohom_vanishing} and
\ref{Prop:CM}.

\subsection{Main geometric results}\label{SS_geom_results}
Recall the vector bundle $\Pro^{reg}$  and the line $\mathcal{O}^{reg}(1)$ on $X^{reg}$.
Let $\iota$ denote the inclusion $X^{reg}\hookrightarrow X$. We write $\mathcal{O}^{reg}(j)$
for the $j$th tensor power of $\mathcal{O}^{reg}(1)$.

Here is the first important result in this
section.

\begin{Prop}\label{Prop:cohom_vanishing}
The following claims hold:
\begin{enumerate}
\item
For all $j>0$,  the sheaf $\iota_*(\Pro^{reg,*}\otimes \mathcal{O}^{reg}(j)\otimes \Pro^{reg})$
on $X$ is maximal Cohen-Macaulay and its higher cohomology vanishes.
\item
In particular, we have $H^i(X^{reg}, \Pro^{reg,*}\otimes \mathcal{O}^{reg}(j)\otimes \Pro^{reg})=0$ for all $j>0$ and $i=1,2$.
\end{enumerate}
\end{Prop}

We will prove this proposition using another major result of this section:

\begin{Prop}\label{Prop:CM}
Let $\mathcal{L}$ be a line bundle on $X^{reg}$ and let $\iota$ denote the inclusion $X^{reg}\hookrightarrow X$.
Suppose there is $\ell>0$ such that
$\iota_*(\Lcal^{\otimes \ell})$ is a line bundle on $X$.
Then $\iota_* \mathcal{L}$ is a  Cohen-Macaulay sheaf.
\end{Prop}

Proposition \ref{Prop:CM} will be proved in Section \ref{SS_CM}.  Now we prove Proposition \ref{Prop:cohom_vanishing}
assuming Proposition \ref{Prop:CM}.

\begin{proof}[Proof of Proposition \ref{Prop:cohom_vanishing}]
We note that (1) implies (2): if $\mathcal{F}$ is a maximal Cohen-Macaulay sheaf on $X$
then $R^i\iota_*(\iota^*\mathcal{F})=0$ for $i=1,2$ because $\operatorname{codim}_X X^{sing}\geqslant 4$.
It follows that $H^i(X^{reg}, \Pro^{reg,*}\otimes \mathcal{O}^{reg}(j)\otimes \Pro^{reg})=
H^i(X, \iota_*(\Pro^{reg,*}\otimes \mathcal{O}^{reg}(j)\otimes \Pro^{reg}))$ for $i=1,2$.
The right hand side vanishes by (1).

The proof of (1) is in several steps.

{\it Step 1}. Recall that there is a positive integer $\ell$ such that $\Ocal^{reg}(\ell)$ is obtained by restricting an ample line bundle on $X$ that will be denoted by $\Ocal(\ell)$. So, for each coherent sheaf $\mathcal{F}$ on $X$, there is a positive integer $d(\mathcal{F})$ such that $H^i(X, \mathcal{F}\otimes \Ocal(d\ell))=0$ for all $i>0$ and all
$d>d(\mathcal{F})$. Set $d_0$ to be the maximum of $d(\mathcal{F})$, where $\mathcal{F}$ runs over $\iota_* \Ocal^{reg}(j)$
for $j=0,\ldots,\ell-1$. Now, by Proposition \ref{Prop:CM}, each of the sheaves $(\iota_* \Ocal^{reg}(j))\otimes \Ocal(d\ell)$ is a
maximal Cohen-Macaulay $\Ocal_X$-module.
%Since $\operatorname{codim}_X X^{sing}\geqslant 4$, we see that $$H^i_{X^{sing}}(X, (\iota_* %\Ocal^{reg}(j))\otimes \Ocal(d\ell))=0,
%\forall i\leqslant 3$$
%and hence
%$$H^i(X, (\iota_* \Ocal^{reg}(j))\otimes \Ocal(d\ell))\xrightarrow{\sim}
%H^i(X^{reg}, \Ocal(j+d\ell)), i=1,2.$$
We conclude that $\iota_*(\Ocal^{reg}(j))\otimes \Ocal(d\ell)$ is Cohen-Macaulay and has vanishing higher cohomology for all $d$ sufficiently large and all $j=0,\ldots,\ell-1$.

{\it Step 2}. Note that $X,\Ocal^{reg}(1)$ are defined over a finite localization of a  ring of algebraic integers, say $R$. After a further finite localization of $R$ we can achieve the following:
\begin{itemize}
    \item $(X^{reg})_R$ is smooth and $\Ocal^{reg}(1)$ is a base change of a line bundle, $\Ocal^{reg}_R(1)$, on $X_R^{reg}$,
    \item $\Ocal_R(\ell):=\iota_*\Ocal^{reg}_R(\ell)$ is an ample line bundle on $X_R$,
    \item $\iota_* \Ocal^{reg}_R(j)$ is a maximal Cohen-Macaulay sheaf on $X_R$ for all $j=0,\ldots,\ell-1$.
\end{itemize}
Using these properties we see that
$\iota_* \Ocal^{reg}_R(j)\otimes \Ocal_R(d\ell)$ is a maximal Cohen-Macaulay sheaf with vanishing higher cohomology for all
$d$ sufficiently large, say $d\geqslant d_1$, and all $j=0,\ldots,\ell-1$. We conclude
that for any field $\F$ that is an
$R$-algebra, we have
\begin{itemize}
\item[($\heartsuit$)]
$\iota_* \Ocal^{reg}_\F(j)\otimes \Ocal_\F(d\ell)$ is a maximal Cohen-Macaulay sheaf on $X_\F$
with vanishing higher cohomology for all $d\geqslant d_1$.
\end{itemize}

{\it Step 3}. We will use property ($\heartsuit$) to establish (1) in this and subsequent
step. Recall that we write $\Ecal$ for endomorphism sheaf of $\Pro$. Consider the scheme $X_\F^{\wedge}$ defined analogously to (\ref{eq:X_wedge}). It is enough to show
the direct analog of (1) over $\F$ for $p:=\operatorname{char}\F\gg 0$.
Let $\Ecal_\F^{\wedge}$ denote restriction of $\Ecal_\F$ to $X_\F^{\wedge}$
and, similarly, let $\Ocal_\F^{\wedge,reg}$ denote the restriction of $\Ocal_\F^{reg}$
to $X_\F^{\wedge,reg}$. Similarly to
Step 1 in the proof of \cite[Lemma 3.4]{HC}, we see that the restriction of
$\iota_*(\Ecal_\F^{reg}\otimes \mathcal{O}_\F^{reg}(j))$ to $X_\F^{\wedge}$
coincides with
\begin{equation}\label{eq:CM_sheaf1}
\iota^{\wedge}_*(\Ecal_\F^{\wedge,reg}\otimes \mathcal{O}_\F^{\wedge,reg}(j)),\end{equation}
where we write $\iota^\wedge$ for the inclusion $X_\F^{\wedge,reg}\hookrightarrow X_\F^{\wedge}$.
So we need to show that (\ref{eq:CM_sheaf1}) is maximal Cohen-Macaulay with vanishing
higher cohomology for all $j>0$. It is enough to do this after a Frobenius twist.
In the notation of Section \ref{SS_Procesi_sheaves}, (\ref{eq:CM_sheaf1}) becomes
\begin{equation}\label{eq:CM_sheaf2}
\iota^{\wedge_0}_*(\mathcal{A}_\F|_{X_\F^{(1),\wedge,reg}}\otimes \mathcal{O}_\F^{(1),\wedge,reg}(j)),\end{equation}
Recall that $\mathcal{A}_\F$ has the same direct summands as a quantization $\mathcal{D}_\F^{\wedge}$
of $X_\F^{\wedge}$. So we need to show that
\begin{equation}\label{eq:CM_sheaf3}
\iota^{\wedge}_*(\mathcal{D}^{\wedge}_\F|_{X_\F^{(1),\wedge,reg}}\otimes \mathcal{O}_\F^{(1),\wedge,reg}(j)),\end{equation}
is maximal Cohen-Macaulay with vanishing higher cohomology.
Again, arguing as in Step 1 of the proof of \cite[Lemma 3.4]{HC}, we see that
this sheaf is the restriction to $X_\F^{(1),\wedge}$ of
\begin{equation}\label{eq:CM_sheaf4}
\iota_*(\mathcal{D}_\F|_{X_\F^{(1),reg}}\otimes \mathcal{O}_\F^{(1),reg}(j)).\end{equation}
So it is enough to show that (\ref{eq:CM_sheaf4}) is maximal Cohen-Macaulay with vanishing higher cohomology. We will do this in
the next step.

{\it Step 4}. We note that $\mathcal{D}_\F$ is a filtered deformation of $\operatorname{Fr}_* \Ocal_{X_\F}$. It follows that  $$\mathcal{D}_\F|_{X_\F^{(1),reg}}\otimes \mathcal{O}_\F^{(1),reg}(j)$$
is a filtered deformation of
$$(\operatorname{Fr}_* \Ocal_{X^{reg}_\F})\otimes \mathcal{O}_\F^{(1),reg}(j)\cong
\operatorname{Fr}_*\mathcal{O}_\F^{reg}(pj).$$
Since $p$ is sufficiently large,  by ($\heartsuit$), $\iota_*\operatorname{Fr}_*\mathcal{O}_\F^{reg}(pj)$ is maximal Cohen-Macaulay
with vanishing higher cohomology. Similarly to the derivation of (1)$\Rightarrow$(2) in the beginning
of the proof, we see that  $R^1\iota_*\operatorname{Fr}_*\mathcal{O}_\F^{reg}(pj)=0$.
It follows that (\ref{eq:CM_sheaf4}) is a filtered deformation of $\iota_*\operatorname{Fr}_*\mathcal{O}_\F^{reg}(pj)$. Since the latter is maximal Cohen-Macaulay
with vanishing higher cohomology, so is (\ref{eq:CM_sheaf4}). This finishes the proof.
\end{proof}

\subsection{Cohen-Macaulay property}\label{SS_CM}
In this section we prove Proposition \ref{Prop:CM}. Let $\Lcal$ be as in Proposition
\ref{Prop:CM}. We need to prove that $\iota_* \Lcal$ is Cohen-Macaulay. We start with the following lemma.

\begin{Lem}\label{Lem:Poisson_structure} Every point $x\in X$ has a  Zariski open neighborhood, say $U$, such that $\mathcal{L}|_{U^{reg}}$ has a D-module structure.  \end{Lem}
\begin{proof}
Since $H^1(X^{reg},\Ocal)=0$, the line bundle $\mathcal{L}$ quantizes to a $\Dcal_{\hbar,\lambda+c_1(\Lcal)}$-$\Dcal_{\hbar,\lambda}$-bimodule for any $\lambda\in H^2(X^{reg},\C)$,
see, e.g., \cite[Section 5.1]{BPW}. Take a Zariski open neighborhood $U$ of $x$ in $X$ such that the line bundle $\iota_*(\mathcal{L}^{\otimes \ell})$ trivializes on $U$. We can assume that $U$ is affine. Then
$H^i(U^{reg},\Ocal)=0$ for $i=1,2$ (because $U$ is Cohen-Macaulay and $\operatorname{codim}_U U^{sing}\geqslant 4$), hence the formal quantizations of $U^{reg}$ are classified by their period, see \cite[Theorem 1.8]{BK}.

Since $\iota_*(\mathcal{L}^{\otimes \ell})$ trivializes on $U$, it follows that
$\mathcal{L}^{\otimes \ell}$ is trivial on $U^{reg}$, hence $c_1(\mathcal{L}|_{U^{reg}})=0$.
Therefore, we have
$$\Dcal_{\hbar,\lambda}|_{U^{reg}}\cong \Dcal_{\hbar,\lambda+c_1(\Lcal)}|_{U^{reg}}.$$
A vector bundle that quantizes to a bimodule over the same formal quantization on the left and on the right, gets a Poisson structure. But over a smooth symplectic variety, a coherent Poisson module is the same thing as a D-module, see, e.g., Step 3 of the proof of \cite[Lemma 3.9]{B_ineq}.
\end{proof}

Recall the morphism $\underline{\rho}:X\rightarrow Y$. Set $y:=\underline{\rho}(x)$.
Let $v\in \h\oplus \h^*$ be a point
in the preimage of $y$. Choose a $W_v$-stable small disc $Z$ around $v$.
Then $Z/W_v$ is a neighborhood of $y$ in the complex analytic topology. Set $\tilde{Z}:=\underline{\rho}^{-1}(Z/W_v)$.
%Consider the scheme $X^{\wedge_y}$.
%Inside we have the open subschemes
%$X^{\wedge_y,reg}\supset %Y^{\wedge_y, reg}$.

\begin{Lem}\label{Lem:alg_fun_grp_fin}
The group $\pi_1(\tilde{Z}^{reg})$
is a quotient of $W_v$.
%The algebraic fundamental group of %$X^{\wedge_y,reg}$ is finite.
\end{Lem}
\begin{proof}
Indeed, $(Z/W_v)^{reg}$ embeds into $\tilde{Z}^{reg}$ as the complement to a
closed complex analytic subspace
hence $\pi_1((Z/W_v)^{reg})\twoheadrightarrow \pi_1(\tilde{Z}^{reg})
$. But $\pi_1((Z/W_v)^{reg})$ is easily seen to coincide with $W_v$.
%It is enough to prove this for the %open subscheme $Y^{\wedge_y,reg}$. %Let $v\in \h\oplus \h^*$ lie in the %preimage of $y$. Let $(\h\oplus %\h^*)^0$ denote the locus in %$\h\oplus \h^*$, where the action %of $W$ is free.
\end{proof}

\begin{proof}[Proof of Proposition \ref{Prop:CM}]
What we need to show is that the completion $(\iota_* \Lcal)^{\wedge_x}$ is Cohen-Macaulay, this implies
that the stalk of $\iota_*\Lcal$ is Cohen-Macaulay.
The proof is in several steps.

{\it Step 1}.
Let $W^0$ denote the kernel of $\pi_1((Z/W_v)^{reg})\twoheadrightarrow \pi_1(\tilde{Z}^{reg})$. Set $Y^0:=(\h\oplus \h^*)/W^0$ and $X^0:=Y^0\times_Y X$.
Let $\eta$ denote the projection $X^0\twoheadrightarrow X$. The preimage $Z^0$ of $\tilde{Z}^{reg}$ in $X^0$ is smooth and is a  simply connected cover of $\tilde{Z}^{reg}$, by the choice of $W^0$. The morphism $\eta:Z^0\rightarrow \tilde{Z}^{reg}$ is etale. It follows that there is a Zariski open neighborhood $U_1$ of $x$ in $U$ such that $\eta$ is etale over $U_1^{reg}$. Define $U_1^0$ from the Stein decomposition for $\pi^{-1}(U_1^{reg})
\rightarrow U_1$ so that $\eta^{-1}(U_1^{reg})$ embeds into $U_1^0$ as an open subset and $U_1^0\rightarrow U_1$
is the quotient morphism for the group $\pi_1(\tilde{Z}^{reg})$. Similarly to the proof of
\cite[Lemma 2.5]{HC}, it follows that $U_1^0$ has symplectic singularties and hence Cohen-Macaulay.

{\it Step 2}.
Let $\Ocal^{an}_X, \Ocal^{an}_{X^{reg}}$ denote the sheaves of analytic functions on
$X$ and $X^{reg}$. Then we have the analytification functor $\bullet^{an}:=\Ocal^{an}_X\otimes_{\Ocal_X}\bullet$ from the category of coherent
$\Ocal_X$-modules to the category of coherent $\Ocal^{an}_{X}$-modules, and similarly for
$X^{reg}$.
We claim that  $(\iota_* \Lcal)^{an}$ coincides with the analytic push-forward of
$\Lcal^{an}$, to be denoted by $\iota^{an}_* \Lcal^{an}$.

Note that $\mathcal{O}_X^{an}$ is flat over $\Ocal_X$. So we have an isomorphism of
functors
\begin{equation}\label{eq:dual_an_iso}\mathcal{H}om_{\mathcal{O}_X}(\bullet, \mathcal{O}_X)^{an}\xrightarrow{\sim}
\mathcal{H}om_{\mathcal{O}_X^{an}}(\bullet^{an}, \mathcal{O}_X^{an})
\end{equation}
(for $\operatorname{Coh}(X)$
to $\operatorname{Coh}(X^{an})$)
Since $\operatorname{codim}_X X^{sing}\geqslant 2$, the pushforward $\iota_*\mathcal{L}$ is a reflexive $\mathcal{O}_X$-module, i.e., it coincides with its double dual.
It follows from (\ref{eq:dual_an_iso}) that $(\iota_*\mathcal{L})^{an}$ is a reflexive $\mathcal{O}_X^{an}$-module.
Note that $\mathcal{L}^{an}$ coincides with the pullback of $(\iota_*\Lcal)^{an}$.
Since  $(\iota_*\Lcal)^{an}$ is reflexive and $\operatorname{codim}_X X^{sing}\geqslant 2$,
we see that $(\iota_*\Lcal)^{an}$ coincides with $\iota_*^{an}\Lcal^{an}$.
%Indeed, it is easy to see that
%the natural morphism $(\iota_* \Lcal)^{an}\rightarrow \iota^{an}_* \Lcal^{an}$ is a monomorphism.
%This is an isomorphism for $\Lcal$ if and only if the analogous map for $\Lcal\otimes \iota^* \Lcal'$
%is an isomorphism for a line bundle $\Lcal'$ on $X$.
%Twisting $\Lcal$ with the pullback of a sufficiently ample line bundle on $X$, we may assume
%that $\iota^{an}_* \Lcal^{an}$ is generated by its global algebraic sections. In this situation
%$(\iota_* \Lcal)^{an}\rightarrow \iota^{an}_* \Lcal^{an}$ is an epimorphism.
%because $\operatorname{codim}_X X^{sing}\geqslant 4$ and we have a
%contracting $\C^\times$-action on $X$.

{\it Step 3}.
Recall, Lemma \ref{Lem:Poisson_structure}, that $\Lcal|_{U^{reg}}$ has a D-module structure.
In particular, $\Lcal^{an}|_{\tilde{Z}^{reg}}$ is a $D$-module,
i.e., a vector bundle with a flat connection. It follows that it is the direct sum of  $\pi_1(\tilde{Z}^{reg})$-isotypic
component in $\eta_*\Ocal^{an}_{\eta^{-1}(\tilde{Z}^{reg})}$.
Therefore $(\iota_* \Lcal)^{\wedge_x}$ is also the direct sum of isotypic components in the complete ring
$\C[\eta^{-1}(X^{\wedge_x})]$. The latter is ring is Cohen-Macaulay because $U_1^0$ has symplectic singularities.
 Hence
$(\iota_*\Lcal)^{\wedge_x}$ is a Cohen-Macaulay $\C[X]^{\wedge_x}$-module.
\end{proof}

\begin{Rem}\label{Rem:CM_generalization}
We expect that a direct analog of
Proposition \ref{Prop:CM} holds for the partial resolutions of general conical symplectic singularities. The proof should be similar to the one we gave above, modulo some technical issues.
\end{Rem}

%In fact, the proof of Proposition \ref{Prop:CM} implies the following result.
%
%\begin{Cor}\label{Cor:CM}
%The $\Ocal_X$-module $\iota_*(\Pro^{reg,*}\otimes \Ocal^{reg}(d)\otimes \Pro^{reg})$
%is maximal Cohen-Macaulay.
%\end{Cor}

%The goal of this section is to recall a %few basic things about
%Hilbert schemes, Procesi bundles and %(rational and trigonometric) Cherednik
%algebras.

\subsection{Construction and properties of $B_{\hbar, c+d\leftarrow c}$}\label{SS_Cherednik_deform_constr}
Let $d>0$.
Recall the space
$$B_d:=\Gamma(X^{reg}, \Pro^{reg,*}\otimes \Ocal^{reg}(d)\otimes \Pro^{reg}).$$
We view $\Pro$ as a right $H$-module, so $B_d$ becomes an $H$-bimodule. Moreover, recall that
$\Pro$  has a  $(\C^\times)^2$-equivariant structure. The line bundle $\Ocal(1)=\Pro\epsilon_-$
inherits the $(\C^\times)^2$-equivariant structure.
This equips $B_d$ with a $(\C^\times)^2$-equivariant structure. We will mostly consider a part of the action, an action of $\C^\times$ given by $t.(x,y)=(x,t^{-2}y)$ for $x\in \h, y\in \h^*$.
In this section we produce a deformation of $B_d$ to an $H_{\hbar,c+d}$-$H_{\hbar,c}$-bimodule and study
its properties.

The deformation of $B_d$ is constructed as follows.
Recall, Section \ref{SS_RCA}, that $\Pro^{reg}$ quantizes to a
$\Dcal^{reg}_{\hbar,\lambda}-\Ecal^{reg}_{\hbar,\lambda}$-bimodule $\Pro^{reg}_{\hbar,\lambda}$ for any
$\lambda\in H^2(X^{reg},\C)$. Set $\nu:=c_1(\Ocal(1))$.
Also $\Ocal^{reg}(d)$ quantizes to a
$\Dcal^{reg}_{\hbar,\lambda+d\nu}$-$\Dcal^{reg}_{\hbar,\lambda}$-bimodule to be denoted by $\Dcal^{reg}_{\hbar,\lambda+d\nu\leftarrow \lambda}$. So
$$\Gamma(X^{reg}, \Pro^{reg,*}_{\hbar,\lambda+d\nu}\otimes_{\Dcal^{reg}_{\hbar, \lambda+d\nu}}\Dcal_{\hbar,\lambda+d\nu\leftarrow \lambda}\otimes_{\Dcal^{reg}_{\hbar,\lambda}}\Pro^{reg}_{\hbar,\lambda})$$
becomes a $(\C^\times)^2$-equivariant $H^{\wedge_\hbar}_{\hbar,\lambda+d\nu}$-$H^{\wedge_\hbar}_{\hbar,\lambda}$-bimodule. We set
$$B_{\hbar,\lambda+d\nu\leftarrow \lambda}:=
\Gamma(X^{reg}, \Pro^{reg,*}_{\hbar,\lambda+d\nu}\otimes_{\Dcal^{reg}_{\hbar, \lambda+d\nu}}\Dcal^{reg}_{\hbar,\lambda+d\nu\leftarrow \lambda}\otimes_{\Dcal^{reg}_{\hbar,\lambda}}\Pro^{reg}_{\hbar,\lambda})^{fin},$$
where the superscript ``fin'' means that we take the finite part for our  $\C^\times$-action.
This is a bigraded $H_{\hbar,\lambda+d\nu}$-$H_{\hbar,\lambda}$-bimodule. Note that for $d=0$, we recover the regular $H_{\hbar,\lambda}$-bimodule.

\begin{Rem}\label{Rem:bigraded}
By the construction, $B_{\hbar,\lambda+d\nu\leftarrow \lambda}$ carries an action of $(\C^\times)^2$. So it is bigraded.
\end{Rem}

It follows from (2) of Proposition \ref{Prop:cohom_vanishing} that $B_{\hbar,\lambda+d\nu\leftarrow \lambda}$ is a free graded $\C[\hbar]$-module with
$$B_{\hbar, \lambda+d\nu\leftarrow \lambda}/(\hbar)=B_d.$$

We now explain the choice of $\lambda$ we mostly need: we want $c_\lambda$ to take the same value
(to be denoted by $c$) on all simple reflections. Then $c_{\lambda+d\nu}$
takes value $c+d$ on every reflection. So we will write $H_{\hbar,c+d}, H_{\hbar,c}$ for the algebras and $B_{\hbar,c+d\leftarrow c}$ for the bimodule.

Now we explain an important property of the bimodules $B_{\hbar, c+?\leftarrow c}$.

\begin{Lem}\label{Lem:bimod_sph_isom}
The following claims are true:
\begin{enumerate}
\item For all $c\in \C$, we have a bigraded algebra isomorphism
$\epsilon_-H_{\hbar,c+1}\epsilon_-\cong
\epsilon H_{\hbar,c}\epsilon$.
\item Thanks to (1) we can view  $\epsilon_- B_{\hbar,c+d+1\leftarrow c}$ as $\epsilon H_{\hbar,c+d}\epsilon$-$H_{\hbar,c}$-bimodule.
For all $c\in \C$ and all $d>0$, we have an
$\epsilon H_{\hbar,c+d}\epsilon$-$H_{\hbar,c}$-bilinear bigraded isomorphism
$$\epsilon_- B_{\hbar,c+d+1\leftarrow c}\cong
\epsilon B_{\hbar,c+d\leftarrow c}.$$
\end{enumerate}
\end{Lem}
\begin{proof}
%(1) is a classical fact, and, moreover, it %can be proved similarly to (2) so we just %prove (2).

Consider the $\Ecal^{reg}_{\hbar,c+d+1}$-$\Dcal^{reg}_{\hbar,c+d}$-bimodule $(\Pro^{reg}_{\hbar,c+d+1})^*\otimes_{\Dcal^{reg}_{\hbar,c+d+1}}\Dcal^{reg}_{\hbar,c+d+1\leftarrow c+d}$. We claim that $$\epsilon_-\left((\Pro^{reg}_{\hbar,c+d+1})^*\otimes_{\Dcal^{reg}_{\hbar,c+d+1}}\Dcal^{reg}_{\hbar,c+d+1\leftarrow c+d}\right)$$
is the regular $\Dcal^{reg}_{\hbar,c+d}$-bimodule.
 Indeed, $\Pro^{reg}_{\hbar,c+d+1}\epsilon_-$
 is the unique quantization of $\Ocal^{reg}(1)$ to a left $\Dcal^{reg}_{\hbar,c+d+1}$-module. The opposite endomorphism sheaf is $\Dcal^{reg}_{\hbar,c+d}$. Hence  $\Pro^{reg}_{\hbar,c+d+1}\epsilon_-\cong \Dcal^{reg}_{\hbar,c+d+1\leftarrow c+d}$ as a bimodule. Our claim follows.

To prove (1) we use the previous paragraph to see that
\begin{equation}\label{eq:spherical_iso_sheaf}\epsilon_- \Ecal^{reg}_{\hbar,c+1}\epsilon_-=\mathcal{E}nd_{\Dcal^{reg}_{\hbar,c+1}}(\Pro^{reg}_{\hbar,c+1}\epsilon_-)^{opp}=\Dcal_{\hbar,c}= \epsilon \Ecal^{reg}_{\hbar,c}\epsilon.
\end{equation}
Since $\Gamma(X, \Ecal^{reg}_{\hbar,c})^{fin}=H_{\hbar,c}$, (\ref{eq:spherical_iso_sheaf}) implies (1).

 We proceed to (2).
 We note that
 \begin{align*}
 &\epsilon_-\Pro^{reg*}_{\hbar,c+d+1}\otimes_{\Dcal^{reg}_{\hbar, c+d+1}}\Dcal^{reg}_{\hbar,c+d+1\leftarrow c}\otimes_{\Dcal^{reg}_{\hbar,c}}\Pro^{reg}_{\hbar,c}=\\&\left(\epsilon_-\Pro^{reg*}_{\hbar,c+d+1}
 \otimes_{\Dcal^{reg}_{\hbar, c+d+1}}\Dcal^{reg}_{\hbar,c+d+1\leftarrow c+d}\right)\otimes_{\Dcal^{reg}_{\hbar,c+d}}
 \left(\Dcal^{reg}_{\hbar,c+d\leftarrow c}\otimes_{\Dcal^{reg}_{\hbar,c+d}}\Pro^{reg}_{\hbar,c}\right)=\\
 &\Dcal^{reg}_{\hbar,c+d\leftarrow c}\otimes_{\Dcal^{reg}_{\hbar,c}}\Pro^{reg}_{\hbar,c}=\epsilon \Pro^{reg*}_{\hbar,c+d}\otimes_{\Dcal^{reg}_{\hbar,c+d}}
 \Dcal^{reg}_{\hbar,c+d\leftarrow c}\otimes_{\Dcal^{reg}_{\hbar,c}}\Pro^{reg}_{\hbar,c}.
 \end{align*}
 Passing to the global sections and taking the $\C^\times$-finite part, we arrive at the statement of (2).
\end{proof}

We will also need a description of $\operatorname{End}_{H_{\hbar,c}}(B_{\hbar,c+d\leftarrow c})$. Note that we have a graded $\C[\hbar]$-algebra homomorphism
$H_{\hbar,c+d}\rightarrow \operatorname{End}_{H_{\hbar,c}}(B_{\hbar,c+d\leftarrow c})$.

\begin{Lem}\label{Lem:endomorphism}
This homomorphism is an isomorphism.
\end{Lem}
\begin{proof}
    It is enough to prove that the homomorphism
    \begin{equation}\label{eq:endom_isom}
    H\rightarrow
    \operatorname{End}_H(B_{d})
    \end{equation}
    is an isomorphism.
    Since $B_d$ is the global sections of a vector bundle on $X^{reg}$,
    and $\C[Y]=\C[X^{reg}]$, we see that $B_d$ is a torsion-free $\C[Y]$-module. It follows that
    we have an algebra embedding
    \begin{equation}\label{eq:endom_embed}
        \operatorname{End}_H(B_{d})\hookrightarrow
        \operatorname{End}_{H^{reg}}(B_d^{reg}).
    \end{equation}
    Here we write $H^{reg}, B_d^{reg}$ for the restrictions of $H,B_d$ to $Y^{reg}$.
    So it is enough to show that the composition of (\ref{eq:endom_embed}) and (\ref{eq:endom_isom}) is an isomorphism. On the other hand,
    \begin{equation}\label{eq:endom_local}
    H^{reg}\xrightarrow{\sim}\mathcal{E}nd_{H^{reg}}(B_d^{reg}).\end{equation}
    The composition $H\rightarrow \operatorname{End}_{H^{reg}}(B_d^{reg})$ of (\ref{eq:endom_embed}) and (\ref{eq:endom_isom})
    is obtained from (\ref{eq:endom_local}) by passing to global sections. This finishes the proof.
\end{proof}

\section{Borel-Moore homology}\label{S_BM_background}
\subsection{General Properties of Borel-Moore homology}\label{SS_BM_general}
 In this section we recall the notion of equivariant Borel-Moore homology and the necessary properties needed to prove the isomorphism in Theorem \ref{Thm:iso}. The main references we use are \cite{Kivinen}, \cite{Brion} and \cite{GKM}.

Let $X$ be a projective variety. Then we can consider the dualizing sheaf $\omega_X\in D^b_c(X)$, the bounded derived category of constructible sheaves on $X$. Then we can define:
$$H_*^{BM}(X)=H^{-*}(\omega_X).$$

Now assume we have an algebraic action of a torus $T$ on $X$. To consider the equivariant Borel-Moore homology we need to define the Borel-Moore homology of the Borel construction $X\times^TET$, where $ET\rightarrow BT$ is the universal $T$ bundle. Since this is not a finite type variety we need to do this by approximating $ET$ using finite type varieties, which can be done along the lines of \cite{BL}.

Note that from the map $X\times^TET\rightarrow BT$, we get a map $\zeta:H^{BM}_T(X)\rightarrow H^{BM}_T(pt)=\C[\h]$. Also there is an action of the constant sheaf $\C\in D^b_c(X)$ on $\omega_X$, which equips $H^{BM}_T(X)$ with an $H^*_T(X)$-module structure. In particular, $H^{BM}_T(X)$
becomes an algebra over $H^*_T(pt)=\C[\h]$. The map $\zeta:H^{BM}_T(X)\rightarrow H^{BM}_T(pt)$ is $\C[\h]$-linear. We get a map
\begin{equation}\label{eq:duality_map}
H^*_T(X)\rightarrow \operatorname{Hom}_{H^*_T(pt)}(H^{BM}_T(X),H^*_T(pt))\end{equation}
by $\alpha\mapsto [\beta\rightarrow \zeta(\alpha\beta)]$.
This map is an isomorphism when $X$ is equivariantly formal, which follows from \cite[Proposition 1]{Brion}.
Also, when $X$ is equivariantly formal the dual map
\begin{equation}\label{eq:duality_map1}
H^{BM}_T(X)\rightarrow \operatorname{Hom}_{H^*_T(pt)}(H^*_T(X),H^*_T(pt))
\end{equation}
is also an isomorphism.

We further have the following two localization lemmas which follow from \cite[Lemma 1]{Brion}.
\begin{Lem}\label{Lem:localization}
	Suppose that $X$ has isolated $T$-fixed points. Consider the inclusion of the fixed points $X^T\hookrightarrow X$. This induces a map
	$$H^{BM}_T(X^T)\rightarrow H^{BM}_T(X)$$
	This map is an isomorphism after inverting finitely many characters of $T$.
\end{Lem}

A dual result also holds for the cohomology $H^*_T(X)$, i.e. we have a natural map
$$ H^*_T(X)\rightarrow H^*_T(X^T)$$
that is an isomorphism after inverting the same characters as in the above lemma.

\begin{Lem}\label{rel:localization}
Let $T'\subset T$ and $X$ be a variety with $T$-action, then we have the localization map
$$H^{BM}_T(X^{T'})\rightarrow H^{BM}_T(X)$$
becomes an isomorphism after inverting those characters of Lemma \ref{Lem:localization} that do not vanish on $T'$.
\end{Lem}

We also have that these two localization maps are compatible with the action of $H^*_T(X)$ on $H^{BM}_T(X)$ in the sense that we have a commuting diagram
\[\begin{tikzcd}
H^*_T(X)\otimes H^{BM}_T(X^T)\ar[d]\ar[rr]&& H^*_T(X)\otimes H^{BM}_T(X)\ar[d]\\
H^*_T(X^T)\otimes H^{BM}_T(X^T)\ar[r]& H^{BM}_T(X^T)\ar[r]& H^{BM}_T(X)
\end{tikzcd}.\]

Further, we can explicitly understand the equivariant Borel-Moore homology under certain conditions on the $T$-action on the space $X$, using the map in Lemma \ref{Lem:localization}.

We first introduce some notation, that we will need to state the result. Consider a $1$-dimensional orbit $E$ of $T$ in $X$. Then the action of $T$ on $E$ factors through some character $\chi:T\rightarrow\mathbb{G}_m$, such that the kernel of $\chi$ is precisely the stabilizer of a point in $E$. Note that there are two choices here by changing the sign, but this does not make a difference to the conditions in the following proposition. Taking the closure of $E$ we get two fixed points in the boundary, which we denote by $x_0$ and $x_\infty$. With this notation we get the following result \cite[Corollary 1]{Brion}.

\begin{Prop}\label{BMcoho.description}
	Let $X$ be a proper equivariantly formal variety with a $T$-action. Assume further that it only has finitely many $1$-dimensional orbits. Let $E_i, i=1,\ldots,k$ be these orbits and let $\chi_i, i=1,\ldots,k,$ denote the corresponding characters. Then $H^{BM}_T(X)\subset H^{BM}_T(X^T)\otimes_{H^*_T(pt)} Frac(H^*_T(pt))$ coincides with the subset of all tuples $(f_x)_{x\in X^T}$ (with $f_x\in Frac(H^*_T(pt))$, note that only finitely many $f_x$ are nonzero because we consider BM homology) satisfying the following conditions
	\begin{itemize}
		\item Let $x\in X^T$. Let $E_1,\ldots,E_k$ be all 1-dimensional orbits whose closure contains $x$, and let $\chi_1,\ldots,\chi_k$ be the corresponding characters. Then $f_x\prod_{i=1}^k \chi_i\in H^*_{T}(pt)$ for any $x\in X^T$.
		%$f_x$ has at most first order poles along the characters $\chi_i$ of the $1$-dimensional orbits $E_i$ such that $x\in\overline{E}_i$.
		\item
		Let $E$ be a 1-dimensional $T$-orbit and let $x_0,x_\infty$
		be the two points in the boundary of $E$. Let $\chi$ be the character corresponding to $E$.
		Then
		%For any character $\chi$ and any connected component $Y$ of $X^{Ker(\chi)}$ we have
		$$Res_{\chi=0}(f_{x_0}+f_{x_\infty})=0$$
		for all 1-dimensional orbits $E$.
		%where $Res_{\chi=0}(f)$ is the residue of the function $f$ along the hyperplane $\chi=0$.
	\end{itemize}
\end{Prop}

%\begin{Rem}\label{Rem:characters_localize}
%
%\end{Rem}

The above results are stated for varieties, but we will need them for ind-schemes. In this setting the corresponding functors $H^*_T$ and $H^{BM}_T$ can be defined respectively as the limit and colimit over the finite dimensional $T$-stable subvarieties and so we can use the above results for varieties to get similar results for ind-schemes.

\begin{Rem}\label{Rem:homology_duality}
Under $\operatorname{Hom}$ colimits are sent to limits. So we still have an isomorphism
$$H^*_T(X)\xrightarrow{\sim} \Hom_{H^*_T(pt)}(H^{BM}_T(X),H^*_T(pt))$$
Note that in the finite type scheme case we also have the dual map
$$H^{BM}_T(X) \rightarrow \Hom_{H^*_T(pt)}(H^*_T(X),H^*_T(pt))$$
being an isomorphim, but in the ind-scheme case this is only true when we consider continuous $\Hom$ with respect to the limit topology.
\end{Rem}

\begin{Rem}\label{Rem:BM_coho_ind}
In the case of ind-schemes, we have a direct analog of Proposition \ref{BMcoho.description} under the following conditions:
\begin{itemize}
    \item $X$ is an ind-proper equivariantly formal ind-scheme with a $T$-action.
    \item $X$ has isolated fixed points.
    \item For any two fixed points $x,x'$, there are finitely many one-dimensional orbits $E$ whose boundary is $\{x,x'\}$.
\end{itemize}
\end{Rem}

\subsection{Borel-Moore homology of equivalued unramified affine Springer fibers}\label{SS_BM_Springer}
In this section we will describe some properties of the Borel-Moore homology of our affine Springer fibers. We use the above results and the main reference for this section is \cite{GKM1} and \cite{GKM2}.

We use the notation $\mathcal{K}=\C((t))$ and $\mathcal{O}=\C[[t]]$.

We start by recalling the definition of the affine flag variety. For a reductive algebraic group $G$ with root data $(R,\mathbb{X}^*=\Lambda^*,R^\vee,\mathbb{X}_*=\Lambda)$, consider the Borel subgroup $B\subset G$ and a maximal torus $T\subset B$.  We also consider the arc and loop groups $G(\mathcal{O})\subset G(\mathcal{K})$ and the Iwahori subgroup $\mathfrak{B}\subset G(\mathcal{O})$.
Recall that the latter is defined as
the preimage of $B$ under the projection $G(\mathcal{O})\twoheadrightarrow G$.
%defined by the following Cartesian %square
%$$\begin{tikzcd}
%	\mathfrak{B}\arrow{d} \arrow{r}
%	& G(\mathcal{O}) \arrow{d}\\
%	B \arrow{r}
%	& G
%\end{tikzcd}$$

 Using these we can define the affine flag variety $\Fl=G(\mathcal{K})/\mathfrak{B}$, which is an ind-projective variety. This space has  actions by $T$ and $T(\mathcal{K})$ given by left multiplication. Further $\C^\times$ acts by field automorphisms on $\mathcal{K}$ scaling $t$ and so we get an induced action on $\Fl$, which is referred to as the loop rotation action.

 We write $\Lambda$ for the co-character lattice of $T$.
 The fixed points of the action of both $T$ and $T\times\C^\times$ are in bijection with the affine Weyl group $\widetilde{W}=W\ltimes\Lambda$ under the natural embedding $\widetilde{W}\hookrightarrow \Fl$. To get this embedding note that $W\hookrightarrow G/B\hookrightarrow\Fl$ and that $T(\mathcal{K})/T(\mathcal{O})\cong \Lambda$ and $T(\mathcal{O})$ acts trivially on the image of $W$ in $\Fl$.

 Further, we have an action of the affine Weyl group $\widetilde{W}$ on the extended torus $T\times\C^\times$. The finite Weyl group $W$ acts only on the $T$ factor with the usual action coming from $W=N(T)/T$. The co-character lattice $\Lambda$ acts via
 \begin{align*}
     t^\lambda:T\times\C^\times&\rightarrow T\times\C^\times\\
     (t,h)&\mapsto (t\lambda(h),h).
 \end{align*}
 Note that the co-character lattice of $T\times \C^\times$ is naturally identified with $\Lambda\times \Z$ and the induced action of $\widetilde{W}$ on $\Lambda\times \Z$ is given by (\ref{eq:affine_Weyl_action_lattice}).

 Now we can introduce the affine Springer fibers we will look at. Fix a nonnegative integer $d$. Consider a regular semisimple element $s\in \h\hookrightarrow\mathfrak{g}$. Then we can consider $e_d=t^ds\in\mathfrak{g}(\mathcal{O})$ and its associated affine Springer fiber, known as the {\it equivalued unramified affine Springer fiber}
 \begin{equation}\label{eq:Springer_fiber}
     \Fl_{e_d}:=\{g\mathfrak{B}\in\Fl|Ad(g)^{-1}e_d\in Lie(\mathfrak{B})\}
 \end{equation}

 Note that $e_d$ is fixed by $T$ and thus $\Fl_{e_d}\subset\Fl$ is $T$-stable and the loop rotation scales $e_d$ hence these Springer fibers $\Fl_{e_d}$ are also stable under the loop rotation action. The image of $\widetilde{W}$ is contained in all these affine Springer fibers, thus these give the $T$-fixed and $T\times\C^\times$-fixed points for all $\Fl_{e_d}$.

 We can further consider the $1$-dimensional orbits of $T\times\C^\times$. In order to do this, we need some notation. For a root $\alpha$ of $\g$, we write $s_\alpha$ for the corresponding reflection in $W$. For an integer $k$, we write $s_{\alpha,k}$ for $t^{k\alpha}s_\alpha$, this is a reflection in $\widetilde{W}$.
 A root $\alpha$ gives a character $\alpha:
 T\rightarrow\mathbb{C}^\times$ and so also gives a character of $T\times\mathbb{C}^\times$, by acting trivially on the loop rotation factor. Further, define $\hbar:T\times\mathbb{C}^\times\rightarrow\mathbb{C}^\times$ as the projection to the loop rotation factor. We can also act on the characters of $T\times \C^\times$ by $\widetilde{W}$, the action induced from that on $T\times\mathbb{C}^\times$. So we get the character $\alpha+k\hbar$ of $T\times \C^\times$. Let ${}^x(\alpha+k\hbar)$ denote the image of $\alpha+k\hbar$ under the action of $x\in\widetilde{W}$.

 The $1$-dimensional orbits in $\Fl$ can be seen to be given by $\mathbb{P}^1$s connecting the fixed points $x$ and $xs_{\alpha,k}$ for all $x\in \widetilde{W}$, roots $\alpha$ and integers $k$. The associated character is given by ${}^x(\alpha+k\hbar)$.

 Below we will use the following notation
 \begin{equation}\label{eq:ring_notation}
 \Ring:=H^*_{T\times\C^\times}(pt), \Field:=Frac(\Ring).
 \end{equation}

 \begin{Prop}\label{GKM:results}
 \begin{enumerate}
     \item For the affine Springer fibers $\Fl_{e_d}$, the $1$-dimensional orbits are given by the $1$ dimensional orbits of $\Fl$ connecting $x$ and $xs_{\alpha,k}$ if $-d\leq k\leq d-1$.
     \item The affine Springer fibers $\Fl_{e_d}$ and the affine flag variety $\Fl$ with the $T\times\C^\times$-action are equivariantly formal.
     \item $H^{BM}_{T\times\C^\times}(\Fl_{e_d})$ is flat as an $\Ring$-module and we have
     \begin{align*}
    H^{BM}_T(\Fl_{e_d})&\cong H^{BM}_{T\times\C^\times}(\Fl_{e_d})\otimes_{H^*_{\C^\times}(pt)}\C\\
	H^{BM}(\Fl_{e_d})&\cong H^{BM}_{T\times\C^\times}(\Fl_{e_d})\otimes_{\Ring}\C.
\end{align*}
The similar claim holds for $\Fl$.
 \end{enumerate}
 \end{Prop}
\begin{proof} The first result is worked out in \cite[Section 5.11]{GKM2}. The second result follows from the existence of an affine space paving as constructed in \cite[Theorem 0.2]{GKM1} for the affine Springer fibers, whilst for the affine flag variety it follows from the Bruhat decomposition. The last result follows immediately from the second.
\end{proof}

\begin{Ex}\label{ex:d_equal_0}
Let $d=0$. Then $e_0=s$, a regular semisimple element. The Springer fiber $\Fl_{e_0}$ is discrete and is identified with the $T$-fixed point locus, $\widetilde{W}$. Claim (1) of the proposition is manifestly true.
\end{Ex}

The following claim follows from combining Proposition \ref{BMcoho.description} (or, more precisely, its ind-scheme generalization, see Remark \ref{Rem:BM_coho_ind}) and
Proposition \ref{GKM:results}.

\begin{Cor}\label{Cor:GKM_Springer}
The localization homomorphism identifies $H^{BM}_{T\times\C^\times}(\Fl_{e_d})$ with the subset of all elements $(f_x)_{x\in \widetilde{W}}\in \bigoplus_{\widetilde{W}}\Field$ satisfying the following two conditions:
\begin{itemize}
    \item[(i)] For all $x$, the product $$f_x\prod_{\alpha\in R^+}\prod_{k=-d}^{d-1}(\,^x\!\alpha+k\hbar)$$
    is an element of $\Ring$. Here $R^+$ stands for the system of positive Dynkin roots.
    \item[(ii)] For all $x\in \widetilde{W}$, $\alpha\in R^+$ and $k$ with $-d\leqslant k\leqslant d-1$, we have $$Res_{\,^x\!(\alpha+k\hbar)}(f_x+f_{xs_{\alpha,k}})=0.$$
\end{itemize}
\end{Cor}

We will also need the following corollary of (1) of Proposition \ref{GKM:results}. Recall that $e_d=t^d s$, where $s\in \h^{reg}$.

\begin{Cor}\label{Cor:indep_s}
The image of $H^{BM}_{T\times \C^\times}(\Fl_{e_d})$ in $\bigoplus_{\widetilde{W}}\Field$ is independent of the choice of a regular semisimple element $s\in \h^{reg}$.
\end{Cor}

Using this corollary we identify the spaces $H^{BM}_{T\times\C^\times}(\Fl_{e_d})$ for different choices of $s$.

 \begin{Rem}\label{Chern:Class}
We now discuss line bundles on $\Fl$.
%The line bundles are given by $\Lambda^*$ as %follows.
For a weight $\lambda\in \Lambda^*\times\mathbb{Z}$ of $T\times \C^\times$ we can construct a $1$-dimensional $T(\mathcal{O})\times\C^\times$-representation $\C_\lambda$, which extends to a $\mathfrak{B}$-representation. The latter gives rise to a $G(\mathcal{K})\rtimes\C^\times$-equivariant line bundle  on $\Fl$ to be denoted by $\mathcal{L}_\lambda$.

 The proof of Proposition \ref{GKM:results} also implies that the conditions for Proposition \ref{BMcoho.description} are satisfied for $\Fl_{e_d}$ and $\Fl$. We can thus consider the localization homomorphism
 $$H^*_{T\times \C^\times}(\Fl)\hookrightarrow \prod_{\widetilde{W}}\Field.$$
 Now we  want to compute the images of the Chern classes of the line bundles $\mathcal{L}_\lambda$ under this localization map.
 To compute the localization to the fixed points of $c_1(\mathcal{L}_\lambda)$, we need to consider the $T\times\C^\times$-representations given by $\mathcal{L}_\lambda$ restricted to a fixed point, $x\in \widetilde{W}$. Note that this gives the $1$-dimensional representation $\C_{\,^x\!\lambda}$ and thus under the map
 $$H^*_{T\times\C^\times}(\Fl)\rightarrow \prod_{x\in\widetilde{W}} \Field$$
 the Chern class $c_1(\mathcal{L}_\lambda)$ is sent to $(\,^x\!\lambda)_{x\in\widetilde{W}}$.
\end{Rem}

\section{The actions on the Borel-Moore homology}\label{S_BM_actions}
In Section \ref{SS_TCA} we have recalled the trigonometric Cherednik algebras $H^\times_{\hbar,c}$.
The goal of this section is to equip $H^{BM}_{T\times \C^\times}(\Fl_{e_d})$ with a structure of an $\Hcal^\times_{\hbar,d}$-$\Hcal^\times_{\hbar,0}$-bimodule and establish some properties of this bimodule.
Recall that we write $\Ring$ for $H^*_{T\times\C^\times}(pt)$ and $\Field$
for $Frac(\Ring)$.

\subsection{Chern-Springer action}

 In this section we will establish a left action of $H^\times_{\hbar,d}$ on $H^{BM}_{T\times \C^\times}(\Fl_{e_d})$.
 Let $\iota$ denote the localization embedding
\begin{equation}\label{eq:localization_embedding} H^{BM}_{T\times \C^\times}(\Fl_{e_d})\hookrightarrow
\bigoplus_{\widetilde{W}}\Field.
\end{equation}
For $x\in \widetilde{W}$, let $\iota(?)_x$ denote the $x$-component of $\iota(?)$, this is an element of $\Field$. We note that the target of (\ref{eq:localization_embedding}) can be viewed as the space of functions $\widetilde{W}\rightarrow \Field$ that are zero outside of a finite set.

 We start by describing the action of the Chern classes. Note that $\prod_{\widetilde{W}}\Ring$ naturally acts on
 $\bigoplus_{\widetilde{W}}\Ring$. So, for a character $\lambda$ of $T\times \C^\times$, the element $c_1(\mathcal{L}_\lambda)$ acts on
$\operatorname{im}\iota$ as the multiplication with  $(\,^x\!\lambda)_{x\in\widetilde{W}}$. This is a consequence of Remark \ref{Chern:Class}. So we get an action of $\h^*\oplus \C\hbar$ on $H^{BM}_{T\times \C^\times}(\Fl_{e_d})$. Note that the operators of this action pairwise commute.
 %Note that $\mathfrak{B}\twoheadrightarrow T$ and so characters of $T$ give characters of $\mathfrak{B}$. %We can thus consider line bundles $\mathcal{L}(\chi)$ on $\Fl$ given by inducing characters $\chi$ of $T$ to $\Fl$. Consider the action of $T$ on the fiber of $\mathcal{L}(\chi)$ at a fixed point $x$. Note that since the action is induced from $\chi$, the action on this fiber is just given by the twisted character ${}^x(\chi)$. Further if we consider the line bundle given by $\chi$ on a fixed point it has Chern class given by $\chi\in H^*_T(pt)$. It thus follows that the Chern class of $\mathcal{L}(\chi)$ is given by $({}^x(\chi))_x\in \prod_{\widetilde{W}} H^*_T(pt)$.\\

 The group $\widetilde{W}$ acts on $H^*_{T\times \C^\times}(\Fl_{e_d})$ via the Springer action, see  \cite{Lusztig}, \cite{Yun} and \cite{OY}. We will recall the construction in Section \ref{SS_affine_Springer_reminder}.

 %Further by  \cite{Lusztig}, this action satisfies the relations of the affine Weyl group as required.\\

 %We then have the following result on the Chern-Springer action.

So we get two actions on $H^{BM}_{T\times \C^\times}(\Fl_{e_d})$: the action of $\h^*\oplus \C\hbar$ by the multplication with Chern classes and the Springer action of $\widetilde{W}$. The former gives rise to an action of the algebra $\C[\h,\hbar]$, while the latter gives an action of the algebra $\C\widetilde{W}$. Both actions are $\Ring$-linear and so extend to the localization $\bigoplus_{\widetilde{W}}\Field$.

\begin{Prop}\label{Springer-Chern-action}
These two actions equip $\bigoplus_{\widetilde{W}}\Field$  with an   $\Hcal^\times_{\hbar,d}$-module structure.
The subspace $H^{BM}_{T\times \C^\times}(\Fl_{e_d})$ embedded via $\iota$ is a submodule.
\end{Prop}

The key tool in the proof is as follows: we write formulas for the actions of simple affine reflections, the elements of $\Lambda/\Lambda_0\subset \widetilde{W}$ and also the elements of $\C[\h^*][\hbar]$ on the image of the embedding $\iota$. Let us state the corresponding result.

\begin{Lem}\label{Lem:CS_action_fixed_pts}
For all $\beta\in H^{BM}_{T\times\C^\times}(\Fl_{e_d})$,
$x\in \widetilde{W}$, simple affine reflections
$s=s_\alpha$, $\lambda\in \h^*\oplus \C\hbar$ and $\pi\in \Lambda/\Lambda_0\subset \widetilde{W}$ we have the following formulas:
\begin{equation}\label{eq:CS_formulas}
\begin{split}
   &\iota(s\beta)_x=\frac{d\hbar}{\,^x\!\alpha}\iota(\beta)_x+\frac{\,^{xs}\!\alpha-d\hbar}{\,^{xs}\!\alpha}\iota(\beta)_{xs},\\
   &\iota(\lambda\beta)_x=(\,^{x}\!\lambda)\iota(\beta)_x,\\
   &\iota(\pi\beta)_x=\iota(\beta)_{x\pi}.
\end{split}
\end{equation}
\end{Lem}
Note that the formulas make sense for an arbitrary element of $\bigoplus_{\widetilde{W}}\Field$ not just for $\iota(\beta)$. They define an action of $H^\times_{\hbar,d}$ on $\bigoplus_{\widetilde{W}}\Field$.

The second equality in (\ref{eq:CS_formulas}) has already been discussed in the beginning of the section.
The last equality easily follows from the construction of the $\Lambda/\Lambda_0$-action
to be discussed in Section \ref{SS_affine_Springer_reminder}. The first equality requires more work, it will be established in the appendix,
Section \ref{SS_CS_localiz}.

\begin{proof}[Proof of Proposition \ref{Springer-Chern-action}]
It is enough to check  the commutation relations of (\ref{eq:trig_CA_relns}).

The second and third equalities  in (\ref{eq:trig_CA_relns}) are immediate from Lemma \ref{Lem:CS_action_fixed_pts}. In the remainder of the proof we will check the first equality.
That is, for a simple reflection $s\coloneqq s_\alpha$ and $\lambda\in\h^*$, we should check the following relation:
\begin{equation}\label{eq:trig_interesting}	
	s\lambda-\,^s\!\lambda s=d\langle \lambda,\alpha^\vee\rangle\hbar
\end{equation}
To check this, we apply the summands of the left hand side to an element $\xi\in \bigoplus_{\widetilde{W}} \Field$.
\begin{align*}
    & (s\lambda\xi)_x=\frac{d\hbar}{\,^x\!\alpha}(\lambda\xi)_x+\frac{\,^{xs}\!\alpha-d\hbar}{\,^{xs}\!\alpha}(\lambda\xi)_{xs}=\frac{d\hbar}{\,^x\!\alpha}\,^x\!\lambda \xi_x+\frac{\,^{xs}\!\alpha-d\hbar}{\,^{xs}\!\alpha}\,^{xs}\!\lambda\xi_{xs},\\
    &(\,^s\!\lambda s\xi)_x=\,^{xs}\!\lambda\left(\frac{d\hbar}{\,^x\!\alpha}\xi_x +\frac{\,^{xs}\!\alpha-d\hbar}{\,^{xs}\!\alpha}\xi_{xs}\right).
\end{align*}
So
\begin{align*}
    &(s\lambda\xi-\,^s\!\lambda s\xi)_x=(\,^x\!\lambda-\,^{xs}\!\lambda)\frac{d\hbar}{\,^x\!\alpha}\xi_x=\langle\lambda,\alpha^\vee\rangle\,^x\!\alpha \frac{d\hbar}{\,^x\!\alpha}\xi_x=d\langle\lambda,\alpha^\vee\rangle \hbar\xi_x.
\end{align*}
This proves the first equality in (\ref{eq:trig_CA_relns}) and finishes the proof.
	%To do this apply this to an element $(g_x)\in $, to get
	%\begin{align*}
	 %   (sy-{}^s(y)s)g_w &=A_{w,w}{}^w(y)g_w+A_{w,ws}{}^{ws}(y)g_{ws}-{}^{ws}(y)(A_{w,w}g_w+A_{w,ws}g_{ws})\\
	 %   &=A_{w,w}\langle y,\check{\alpha}\rangle{}^w(\alpha)g_w=d\langle y,\check{\alpha}\rangle\hbar g_w.
	%\end{align*}
	%To check the relation
	%$$\pi y=(\pi.y)\pi,y\in\mathfrak{t}^*,\pi\in\Lambda/\Lambda_0\subset\widetilde{W}$$
	%act on $f=(f_w)\in  H^{BM}_{T\times\C^\times}(\Fl_{e_d})$.
	%\begin{align*}
	    %(\pi y)(f)_w &=({}^{w\pi}(y)f_{w\pi})=({}^w(\pi.y)f_{w\pi})=(\pi.y)\pi(f)_w
	%\end{align*}
	%The Chern class action of $\hbar$ by Equation \ref{eq:ChernFormula} coincides with the action by the equivariant parameter $\hbar$ and so commutes with the Springer action by Lemma \ref{Springer-action.basics}. We thus have the equation
	%$$x\hbar=\hbar x,x\in\widetilde{W}$$
	%is satisfied.
	%So we get as required a $\Hcal^\times_{d,\hbar}$-module structure on the above cohomology.
	\end{proof}

 \subsection{Equivariant-Centralizer-Monodromy action}
The goal of this section is to define  an action of
$\Hcal_{\hbar,0}^\times$ on $H^{BM}_{T\times\C^\times}(\Fl_{e_d})$.  We will view $H^{BM}_{T\times\C^\times}(\Fl_{e_d})\hookrightarrow
\bigoplus_{\widetilde{W}}\Field$ as right $\Ring$-modules, this structure on the former space was
discussed in the general situation in Section \ref{SS_BM_general}.

%By the construction of $e_d$, this element is centralized by $T(\mathcal{K})$, so the fiber $\Fl_{e_d}$ is stable under $T(\mathcal{K})$. So the component group $\Lambda=T(\mathcal{K})/T(\mathcal{O})$ acts on
%$H^{BM}_{T\times \C^\times}(\Fl_{e_d})$.

%Recall, Corollary \ref{Cor:indep_s}, that
%the spaces $H^{BM}_{T\times \C^\times}(\Fl_{e_d})$ are identified for all $s\in \h^{reg}$. On the other hand, every element $w\in W$ defines an isomorphism
%$\Fl_{e_d}\rightarrow \Fl_{we_d}$ and hence gives the pushforward map
%$H^{BM}_{T\times\C^\times}(\Fl_{e_d})\rightarrow H^{BM}_{T\times\C^\times}(\Fl_{we_d})$. This allows to define a {\it right} representation of $W$ in $H^{BM}_{T\times\C^\times}(\Fl_{e_d})$.
%It is easy to see that the representations of $W,\Lambda$ in $H^{BM}_{T\times\C^\times}(\Fl_{e_d})$ combine into a right representation of $\widetilde{W}$.
%The following lemma describes the representations of $\C[\h][\hbar],\widetilde{W}$
%after applying $\iota: H^{BM}_{T\times\C^\times}(\Fl_{e_d})\hookrightarrow \bigoplus_{\widetilde{W}}\Field$.

Define a right action of $\widetilde{W}$ on $\bigoplus_{\widetilde{W}}\Field$ by \begin{equation}\label{eq:CM_action}
    (fy)_x=\,^{y^{-1}}\!f_{yx}, \quad x,y\in \widetilde{W}, (f_x)\in \bigoplus_{\widetilde{W}}\Field.
\end{equation}

\begin{Lem}\label{Lem:ECM_action_formulas}
The right actions of $\Ring=\C[\h][\hbar]$ and $\widetilde{W}$ on $\bigoplus_{\widetilde{W}}\Field$ constitute a right action of $\Hcal^\times_{\hbar,0}$. Moreover, $\operatorname{im}\iota$ is a submodule.
%For $\beta\in H^{BM}_{T\times\C^\times}(\Fl_{e_d})\lambda\in \h^*\oplus \C\hbar$ and
%$x,y\in \widetilde{W}$ we have the following formulas:
%\begin{equation}\label{eq:ECM_formulas}
%\begin{split}
%    & \iota(\beta\lambda)_x=\iota(\beta)_x \lambda,\\
%    &\iota(\beta y)_x=\,^{y^{-1}}\!(\iota(\beta)_{yx}).
%    \end{split}
%\end{equation}
\end{Lem}

%This lemma will be proved in the appendix, Section \ref{SS_ECM_localiz}.

%\begin{Cor}\label{Cor:Springer_right_module_structure}
%The right actions of $\Ring$ and $\tilde{W}$ on $\bigoplus_{\widetilde{W}}\Field$ equip the latter with the structure of a right $\Hcal^\times_{\hbar,0}$-module. The subspace   $H^{BM}_{T\times\C^\times}(\Fl_{e_d})$ (embedded via $\iota$) becomes a submodule.
%\end{Cor}
\begin{proof}
We start by proving that we indeed get an action of $\Hcal^\times_{\hbar,0}$.
    The only missing relation is the commutation relations of the $\widetilde{W}$ action and the $\Ring$ action, i.e.,
	$$
	    y\lambda-\,^y\!\lambda y=0,\quad y\in\widetilde{W},\lambda\in\mathfrak{t}^*.
	$$
	For $(f_x)\in \bigoplus_{\widetilde{W}}\Field$ we get
	$$
	    (fy\lambda)_x=\lambda(f y)_x=\lambda\,^{y^{-1}}\!f_{yx}=\,^{y^{-1}}\!(\,^y\!\lambda f_{yx})
	    =(f\,^y\!\lambda y)_x.
	$$
	This completes the proof of the claim that the actions of $\Ring$ and $\widetilde{W}$ constitute an action of $\Hcal^\times_{\hbar,0}$.
	
	The claim that the image of $\iota$ is $\Hcal^\times_{\hbar,0}$-stable is immediate from the formulas defining the action and the description of the image in
	Corollary \ref{Cor:GKM_Springer}.
\end{proof}

\begin{Rem}\label{Rem:ECM_name}
The action of  $\Lambda\subset \widetilde{W}$ on $H^{BM}_{T\times\C^\times}(\Fl_{e_d})$ comes from the action of $T(\mathcal{K})$ on $\Fl_{e_d}$. The action of $W\subset \widetilde{W}$ is more tricky. Recall, Corollary \ref{Cor:indep_s} that the spaces $H^{BM}_{T\times \C^\times}(\Fl_{e_d})$ are identified for all choices of $s$ via $\iota$. So the action of $W$ can be interpreted as the monodromy action.
However, we do not know a way to identify the BM homology space for various $s$ without the GKM description. So it is easier just to define the action on the localized BM homology spaces.

The resulting action of $\Hcal_{\hbar,0}^\times$ will be called the ECM (equivariant-centralizer-monodromy) action.
\end{Rem}

\begin{Cor}\label{Cor:Springer_bimodule_structure}
The CS action of $\Hcal^\times_{\hbar,d}$ on $\bigoplus_{\widetilde{W}}\Field$  commutes with  the ECM action of $H^\times_{\hbar,0}$.  Hence these actions also commute on $H^{BM}_{T\times\C^\times}(\Fl_{e_d})$.
\end{Cor}
\begin{proof}
    The actions of generators are specified in Lemma \ref{Lem:CS_action_fixed_pts} for the CS action and in Lemma \ref{Lem:ECM_action_formulas} for the ECM action. One directly checks that the generators of $H^\times_{\hbar,d}$ commute with the generators of $H^\times_{\hbar,0}$.
    %Lemma \ref{Lem:CS_action_fixed_pts} specifies the action of generators of $\Hcal^\times_{\hbar,d}$ on $\operatorname{im}\iota$, while Lemma
    %\ref{Lem:ECM_action_formulas} specifies the action of generators of $\Hcal_{\hbar,0}^\times$. It is straightforward to verify that the operators in Lemma \ref{Lem:CS_action_fixed_pts} commute with those in Lemma \ref{Lem:ECM_action_formulas}. The claim of the corollary follows.
\end{proof}

So $H^{BM}_{T\times\C^\times}(\Fl_{e_d})$ becomes an $\Hcal_{\hbar,d}^\times$-$\Hcal_{\hbar,0}^\times$-bimodule.

\begin{Ex}\label{Ex:Springer_bimodule_0}
Consider the example of $d=0$, where $\Fl_{e_0}\xrightarrow{\sim} \widetilde{W}$ by Example \ref{ex:d_equal_0}. The image of $\iota$ is just $\bigoplus_{\widetilde{W}}\Ring$ that naturally identifies with $H^\times_{\hbar,0}$. The bimodule structure on $H^{BM}_{T\times\C^\times}(pt)$ is that of the regular bimodule, as seen directly from the formulas in Lemmas \ref{Lem:CS_action_fixed_pts} and
\ref{Lem:ECM_action_formulas}.
\end{Ex}
%The action combines the actions of
%\begin{itemize}
%\item
%The action of the equivariant parameters,
%$\Ring$, discussed in the %general situation in Section %\ref{SS_BM_general}.
%\item The action of %$\Lambda=T(\mathcal{K})/T(\mathcal{O})$ on %$H^{BM}_{T\times \C^\times}(\Fl_{e_d})$.
%\item The monodromy action of $W$ discussed %above.
%\end{itemize}

% The action of equivariant classes are very easy to understand under the localization map, as this is explicitely given by an equivariant restriction map.\\

% The action of the centralizer and monodromy can be combined into the action of the affine Weyl group. To see this let $w\in \widetilde{W}$. The action is given by a geometric action given by multiplying by an element of the affine Weyl group followed by an identification of the Borel-Moore homology for different regular semisimple elements $s$ via monodromy. This action sends the fixed point $y$ to $wy$. Further it twists the action of $T\times\C^\times$  by the automorphism given by $w$.\\
% We then have the following result.
% \begin{Lem}
%The ECM action on $H^{BM}_{T}(\Fl_{e_d})$ deforms to an action of $H^\times_{0,\hbar}$.
%\end{Lem}
%\begin{proof}
%    We leave the proof to the Appendix B \ref{ECM:action}.
%\end{proof}

\subsection{Properties of the bimodule}

The goal of this section is to prove some properties of the $\Hcal^\times_{\hbar,d}$-$\Hcal^\times_{\hbar,0}$-bimodule $H^{BM}_{T\times\C^\times}(\Fl_{e_d})$ that are analogous to those of the $\Hcal_{\hbar,d}$-$\Hcal_{\hbar,0}$-bimodule $B_{\hbar,d\leftarrow 0}$ in Lemma \ref{Lem:bimod_sph_isom}.

\begin{Lem}\label{Claim:iso_spherical_general}
For $d\geqslant 0$, we have a graded $\Hcal^\times_{\hbar,0}$-linear isomorphism
$\epsilon H^{BM}_{T\times \C^\times}(\Fl_{e_{d}})\cong \epsilon_- H^{BM}_{T\times \C^\times}(\Fl_{e_{d+1}})$ (where we shift the grading on one of the sides).
\end{Lem}
\begin{proof}
The proof is in several steps.
%	We again start considering the case of $SL_2$.\\

{\it Step 1}.
	Let $\beta\in H^{BM}_{T\times \C^\times}(\Fl_{e_{d+1}})$.
	Set $(f_x):=\iota(\beta)$.
	The condition that $\beta\in \epsilon_- H^{BM}_{T\times \C^\times}(\Fl_{e_{d+1}})$ is equivalent to $f_x=-f_{xs}$ for all simple Dynkin reflections $s$. This follows from Lemma \ref{Lem:CS_action_fixed_pts}.
	
%	Thus using the above notation, the basis of this summand for $H^{BM}_{T\times \C^\times}(\Fl_{e_{d+1}})$ is given by
%	\begin{align*}
%		a_1&=(\dots 0,\frac{1}{y},-\frac{1}{y},0\dots)\\
%		a_2&=(\dots 0,\frac{1}{y(y+\hbar)},-\frac{1}{y(y+\hbar)},-\frac{1}{(y+\hbar)(y+2\hbar)},\frac{1}{(y+\hbar)(y+2\hbar)},0\dots)\\
%		\vdots &\\
%		c_{2k}&=(\dots,0,\frac{1}{f_{2k}^{(0)}},-\frac{1}{f_{2k}^{(0)}},-\frac{{d\choose 1}}{f_{2k}^{(1)}},\dots, -\frac{(-1)^{d-1}{d\choose d-1}}{f_{2k}^{(d-1)}},\frac{(-1)^{d}{d\choose d}}{f_{2k}^{(d)}},-\frac{(-1)^{d}{d\choose d}}{f_{2k}^{(d)}},0\dots)
%	\end{align*}
%	Considering the action of the centralizer, we get the relations
%	$$(1-x)a_i=(y+i\hbar)a_{i+1}$$
%	where here $a_d=c_0$ and further $x(c_{2k})=c_{2k+2}$. To see this we compute the $2r$th entry of $(1-x)a_i$ is given by:
%	\begin{align*}
%		\frac{(-1)^r{i-1 \choose r}}{(y+r\hbar)\dots(y+(r+i-1)\hbar)}+\frac{(-1)^r{i-1 \choose r-1}}{(y+(r+1)\hbar)\dots(y+(r+i)\hbar)}=\\\frac{(-1)^r}{(y+r\hbar)\dots(y+(r+i)\hbar)}({i-1\choose r}(y+(r+i)\hbar)+{i-1\choose r-1}(y+r\hbar))=\\\frac{(-1)^r{i \choose r}(y+i\hbar)}{(y+r\hbar)\dots(y+(r+i-1)\hbar)}
%	\end{align*}
%	And similarly for the $2r+1$st entry. The action on $c_{2k}$ is clear.\\

	Now let $\beta'\in H^{BM}_{T\times\C^\times}(\Fl_{e_d})$.
	Set $(f'_x)=\iota(\beta')$.
	Thanks to Lemma \ref{Lem:CS_action_fixed_pts},
	we have $\beta'\in \epsilon H^{BM}_{T\times\C^\times}(\Fl_{e_d})$ if and only if
	 $(\,^x\!\alpha+d\hbar)f_{xs}=(\,^x\!\alpha-d\hbar)f_x$ for all simple Dynkin reflections $s=s_\alpha$.
	
{\it Step 2}.	
	We want to define mutually inverse maps between $\iota(\epsilon_- H^{BM}_{T\times \C^\times}(\Fl_{e_{d+1}}))$ and $\iota(\epsilon H^{BM}_{T\times \C^\times}(\Fl_{e_{d}}))$.
	Define the element $\upsilon\in \C[\h][\hbar]=\Ring$ by
	$$\upsilon:=\prod_{\alpha\in R^+}(\alpha+d\hbar),$$
	where we write $R^+$ for the system of positive Dynkin roots.
	Define an endomorphism of
	$\bigoplus_{\widetilde{W}}\Field$ by
	\begin{equation}\label{eq:map}
	    \Upsilon:(f_x)\mapsto (g_x)\coloneqq(\,^x\!\upsilon f_x).
	\end{equation}
	Note that $\Upsilon$ is invertible. Also note that $\upsilon$ can be viewed as an element of $H^\times_{\hbar,d}$, see Proposition
	\ref{Springer-Chern-action} and Lemma
	\ref{Lem:CS_action_fixed_pts}. From
	Corollary \ref{Cor:Springer_bimodule_structure} we deduce that $\Upsilon$ is $H^\times_{\hbar,0}$-linear. The element $\upsilon$ has degree $|R^+|$ so we can shift the grading and assume $\Upsilon$ is graded. It remains to show that
	\begin{align}\label{eq:inclusion1}
	& \Upsilon\left(\iota(\epsilon_- H^{BM}_{T\times \C^\times}(\Fl_{e_{d+1}}))\right)\subset
	\iota\left(\epsilon H^{BM}_{T\times \C^\times}(\Fl_{e_{d}})\right),\\\label{eq:inclusion2}
	& \Upsilon^{-1}\left(\iota(\epsilon H^{BM}_{T\times \C^\times}(\Fl_{e_{d}}))\right)\subset
	\iota\left(\epsilon_- H^{BM}_{T\times \C^\times}(\Fl_{e_{d+1}})\right).
	\end{align}
	%and its obvious inverse and then check that these maps send $\epsilon_- H^{BM}_{T\times \C^\times}(\Fl_{e_{d+1}})$ to $\epsilon H^{BM}_{T\times \C^\times}(\Fl_{e_{d}})$ and vice versa.
	
	{\it Step 3}. We start by proving (\ref{eq:inclusion1}) in this step and the next two.
	Assume $(f_x)\in \iota\left(\epsilon_- H^{BM}_{T\times \C^\times}(\Fl_{e_{d+1}})\right)$. We need to check that $(g_x)\in
	\iota\left(\epsilon H^{BM}_{T\times \C^\times}(\Fl_{e_{d}})\right)$. We begin by checking $(g_x)\in  \iota\left(H^{BM}_{T\times \C^\times}(\Fl_{e_{d}})\right)$. This will be done using Corollary \ref{Cor:GKM_Springer} (for both $d$
	and $d+1$).
	
	We  first check (i) for $d$, i.e., that $$g_x\prod_{\alpha\in R^+}\prod_{k=-d}^{d-1}(\,^x\!\alpha+k\hbar)[=\left(f_x\prod_{\alpha\in R^+}(\,^x\!\alpha+d\hbar)\right)
	\prod_{\alpha\in R^+}\prod_{k=-d}^{d-1}(\,^x\!\alpha+k\hbar)]
	\in \Ring$$
	%Note that $f_w$ only has at most simple poles along $({}^w(\alpha)+k\hbar)$ for $\alpha\in R^+$, $k=-d-1,\dots d$ by the description of the $1$-dimensional orbits for $\Fl_{e_{d+1}}$ in Proposition \ref{GKM:results} and Proposition \ref{BMcoho.description}.
	By (i) applied to $d+1$ and the point $x$ in Corollary \ref{Cor:GKM_Springer} we have
	$$f_x
	\prod_{\alpha\in R^+}\prod_{k=-d-1}^{d}(\,^x\!\alpha+k\hbar)
	\in \Ring$$
	It remains to show that $f_x$ (hence $g_x$) cannot have poles along $\,^x\!\alpha-(d+1)\hbar$ for any positive roots $\alpha$.
Note $f_x=-f_{xs}$ so it can only have poles along $(\,^{xs}\alpha+k\hbar)$
for $k=-d-1,\ldots,d$. But, for $s=s_\alpha$, $(\,^{xs}\alpha+k\hbar)=-(\,^x\!\alpha-k\hbar)$. So $f_x$ indeed has no pole along $(\,^x\!\alpha-(d+1)\hbar)$. This establishes (i) of Corollary \ref{Cor:GKM_Springer} for $d$.
	
	{\it Step 4}.
	Now we need to check that (ii) of Corollary \ref{Cor:GKM_Springer} holds for $(g_x)$:
	$$Res_{\,^x\!\beta+k\hbar}(g_x+g_{xs_{\beta,k}})=0$$
	for all $x\in \widetilde{W},\beta\in R^+$, and $k=-d,\dots d-1$.
	Note that $$\,^x\!F\equiv\,^{xs_{\beta,k}}\!F\mod \,^x\!\beta+k\hbar,\quad \forall F\in \Ring.$$
	In particular,
	\begin{equation}\label{eq:congruence}
	\prod_{\alpha\in R^+}(\,^x\!\alpha+d\hbar)\equiv \prod_{\alpha\in R^+}(\,^{xs_{\beta,k}}\!\alpha+d\hbar)\mod \,^x\!\beta+k\hbar.
	\end{equation}
	Recall that $f_x$ has at most simple pole at $\,^x\!\beta+k\hbar$.
	It follows that
	\begin{equation}\label{eq:residue_vanish1}
	Res_{\,^x\!\beta+k\hbar}
	(f_x\prod_{\alpha\in R^+}(\,^x\!\alpha+d\hbar)-f_x\prod_{\alpha\in R^+}(\,^{xs_{\beta,k}}\!\alpha+d\hbar))=0.\end{equation}
	  Since $$Res_{\,^x\!\beta+k\hbar}(f_x+f_{xs_{\beta,k}})=0,$$
	  for all $\beta\in R^+$ and all $k=-d,\ldots,d-1$ (this is a part of (ii) of Corollary \ref{Cor:GKM_Springer}) for $d+1$, we deduce from (\ref{eq:residue_vanish1}) that
	  $$Res_{\,^x\!\beta+k\hbar}(g_x+g_{xs_{\beta,k}})=0$$
	  for $\beta$ and $k$ in the same range. This is exactly (ii) of Corollary \ref{Cor:GKM_Springer}. This finishes the proof of
	  $\Upsilon\left(\iota(\epsilon_- H^{BM}_{T\times \C^\times}(\Fl_{e_{d+1}}))\right)\subset
	\iota\left(H^{BM}_{T\times \C^\times}(\Fl_{e_{d}})\right)$.
	
	 {\it Step 5}.
	We finally check that $\epsilon(g_x)=(g_x)$, equivalently $s_\beta(g_x)=(g_x)$ for each Dynkin simple root $\beta$. This will finish the proof of (\ref{eq:inclusion1}).
	
	Using the formula for the Springer action of $s_\beta$, Lemma \ref{Lem:CS_action_fixed_pts}, and the construction of $(g_x)$ we see that the equality $s_\beta(g_x)=(g_x)$ is equivalent to
	\begin{equation}\label{eq:d_symmetric}
	 (\,^x\!\beta-d\hbar)\left(\prod_{\alpha\in R^+}(\,^x\!\alpha+d\hbar)\right)f_x=(\,^x\!\beta+d\hbar)\prod_{\alpha\in R^+}(\,^{xs}\!\alpha+d\hbar)f_{xs}
	 \end{equation}
	 for all $x\in \widetilde{W}$.
	
	 Rearranging the factors, we get
	\begin{equation}\label{eq:rearranging}
	(\,^x\!\beta-d\hbar)\prod_{\alpha\in R^+}(\,^x\!\alpha+d\hbar)=-(\,^x\!\beta+d\hbar)\prod_{\alpha\in R^+}(\,^{xs}\!\alpha+d\hbar).
	\end{equation}
	Since $(f_x)\in \iota\left(\epsilon_- H^{BM}_{T\times \C^\times}(\Fl_{e_{d+1}})\right)$, we have $f_x=-f_{xs}$.
	 Combining this with (\ref{eq:rearranging}) we get (\ref{eq:d_symmetric}). This finishes the proof of (\ref{eq:inclusion1}).
	
	{\it Step 6}. Now we check (\ref{eq:inclusion2}). Let $(g_x)\in \iota\left(\epsilon H^{BM}_{T\times \C^\times}(\Fl_{e_{d}})\right)$.
	Set $f_x:=\left(g_x \prod_{\alpha\in R^+}(\,^x\!\alpha+d\hbar)^{-1}\right)$. We need to show that
	\begin{itemize}
	\item
	$f_x=-f_{xs}$
	for all $x\in \widetilde{W}$ and simple Dynkin reflection $s$;
	\item and the collection $(f_x)$ satisfies (i) and (ii) of Corollary \ref{Cor:GKM_Springer} for $d+1$.
	\end{itemize}
	The first bullet is checked by reversing the argument of Step 5. In the remainder of the proof we will check the second bullet.

	{\it Step 7}. We start by checking (i).  Note that, by condition (i) for $d$, $g_x$ has at most simple poles along $(\,^x\!\alpha+k\hbar)$ for $\alpha\in R^+$, $k=-d,\dots d-1$. Hence  $f_x$ has at most simple poles along $(\,^x\!\alpha+k\hbar)$ for $\alpha\in R^+$, $k=-d,\dots d$.
	This verifies condition (i) for $d+1$.
	
	{\it Step 8}.
	Now we just need to check condition (ii):
	\begin{equation}\label{eq:residue_d+1}Res_{\,^x\!\beta+k\hbar}(f_x+f_{xs_{\beta,k}})=0
	\end{equation}
	for $\beta\in R^+$ and $k=-(d+1),\dots, d$. Step 7 implies that $f_x$ has no pole along the roots $(\,^x\!\alpha-(d+1)\hbar)$. (\ref{eq:residue_d+1}) for $k=-(d+1)$ and all $\beta$ follows.
	
	Now we establish (\ref{eq:residue_d+1}) for $k=-d,\ldots,d-1$.
	The function  $\left(\prod_{\alpha\in R^+}(\,^x\!\alpha+d\hbar)\right)^{-1}$ has no poles along $(\,^x\!\beta+k\hbar)$ for $k\neq d$.
	Using this and (\ref{eq:congruence}), we easily deduce (\ref{eq:residue_d+1}) from condition (ii) (of Corollary \ref{Cor:GKM_Springer}) for the collection $(g_x)$.

	 It remains to establish (\ref{eq:residue_d+1}) for $k=d$. Note that, by Step 6, $f_x+f_{xs_{\beta,d}}=-f_{xs_\beta}-f_{xs_{\beta,-d}}$. So (\ref{eq:residue_d+1}) for $k=d$  follows from the equation for $k=-d$
	 (with $x$ replaced with $xs_\beta$). The latter has been established in the previous paragraph.
\end{proof}

\section{Proofs of the main theorems}\label{S_proof_1}
In this section we will prove Theorems \ref{Thm:iso} and Theorem \ref{Thm:dim}.
\subsection{Isomorphism of deformations}\label{SS_bimod_deform}
Now we state the main result of this section that implies
Theorem \ref{Thm:iso}. We write $H^\wedge_{\hbar,c}$ for the isomorphic algebras in Lemma \ref{Lem:trig_rat_iso}.
Set
$$B^\wedge_{\hbar,d\leftarrow 0}:=B_{\hbar,d\leftarrow 0}\otimes_{\C[\h^*]}\C[\h^*]^{\wedge_0},
H^{BM}_{T\times\C^\times}(\Fl_{e_d})^\wedge:=H^{BM}_{T\times\C^\times}(\Fl_{e_d})\otimes_{\C[T^\vee]}\C[T^\vee]^{\wedge_1}.$$
Both $B^\wedge_{\hbar,d\leftarrow 0}, H^{BM}_{T\times\C^\times}(\Fl_{e_d})^\wedge$
are graded $H^\wedge_{\hbar,d}$-$H^\wedge_{\hbar,0}$-bimodules that are flat over $\C[\hbar]$. This follows from Section
\ref{SS_Cherednik_deform_constr} for the former bimodule, and from Corollary \ref{Cor:Springer_bimodule_structure} and (2) of Proposition \ref{GKM:results}
for the latter bimodule.

\begin{Thm}\label{Thm:iso_deformed}
We have a graded $H^{\wedge}_{\hbar,d}$-$H^\wedge_{\hbar,0}$-bimodule isomorphism
$B_{\hbar,d\leftarrow 0}^{\wedge}\xrightarrow{\sim}
H^{T\times \C^\times}_{BM}(\Fl_{e_d})^\wedge$.
\end{Thm}

Let us explain key ideas of the proof. We use induction on $d$. Note that for $d=0$
both sides are isomorphic the regular $H^\wedge_{\hbar,0}$-bimodule: for the left hand side this follows from the construction in Section \ref{SS_Cherednik_deform_constr}.
For the right hand side the claim follows
from Example \ref{Ex:Springer_bimodule_0}.
This is our induction base. The induction step is based on Lemmas \ref{Lem:bimod_sph_isom}, \ref{Claim:iso_spherical_general} and the next proposition.

\begin{Prop}\label{Prop:isom_extension}
Let $d>0$. Any graded $\epsilon_- H^{\wedge}_{\hbar,d}\epsilon_-$-$H^{\wedge}_{\hbar,0}$-linear
isomorphism
\begin{equation}\label{eq:sign_spher_iso}
\epsilon_-H^{T\times \C^\times}_{BM}(\Fl_{e_d})^\wedge
\xrightarrow{\sim}
\epsilon_- B_{\hbar,d\leftarrow 0}^\wedge
\end{equation}
uniquely extends to a graded $H^\wedge_{\hbar,d}$-$H^\wedge_{\hbar,0}$-linear isomorphism
$$ H^{T\times \C^\times}_{BM}(\Fl_{e_d})^\wedge\xrightarrow{\sim}
B_{\hbar,d\leftarrow 0}^\wedge.$$
\end{Prop}

The proof will be given after a construction and a lemma.

We can  view $H^\wedge_{\hbar,c}$ as a filtered algebra
(with $\deg \hbar=\deg \C[\h]^{\wedge_0}=\deg W=0,
\deg \h=1$). Formally, the filtered algebra $H^\wedge_{\hbar,c}$ is obtained as
$\C[\hbar']\otimes_{\C[\hbar]}H^\wedge_{\hbar,c}$,
where the homomorphism $\C[\hbar]\rightarrow \C[\hbar']$ sends $\hbar$ to $\hbar'$, but
$\hbar'$ is treated as a degree  $0$ element.  In what follows we write $\hbar$ instead of $\hbar'$. Note that the resulting filtration on $H_{\hbar,c}^\wedge$ is $\C^\times$-stable.
We have $\gr H_{\hbar,c}=H^\wedge\otimes \C[\hbar]$.

Set $$\h^{*\wedge}:=\operatorname{Spec}(\C[\h^*]^{\wedge_0}), Y^\wedge:=\h^{*\wedge}/W\times_{\h^*/W}Y,
Y^{\wedge}_\hbar:=Y^\wedge\times \operatorname{Spec}(\C[\hbar]).$$
The  scheme $Y^\wedge_\hbar$ is the spectrum of the center of
of  $\gr H^\wedge_{\hbar,c}$.
So the algebra $H^\wedge_{\hbar,c}$ can be {\it microlocalized} to $Y^\wedge_\hbar$.

Now we recall some basics on the microlocalization.
The result of microlocalization of $H^\wedge_{\hbar,c}$ is a sheaf of algebras on $Y^\wedge_\hbar$ whose sections are defined on $\C^\times$-stable open subsets of $Y^\wedge_\hbar$ (for the $\C^\times$-action that is the original action on $Y^\wedge$ and is trivial on $\operatorname{Spec}(\C[\hbar])$).
Namely, pick a homogeneous element $f\in \C[Y_\hbar^\wedge]$. Consider the Rees algebra $R_{\mathsf{h}}(H^\wedge_{\hbar,c})$, where $\mathsf{h}$ is a variable of degree $1$.
Lift $f$ to a homogeneous element $\tilde{f}\in R_{\mathsf{h}}(H^\wedge_{\hbar,c})$. Then $\{\tilde{f}^k| k\geqslant 0\}$ is an Ore subset in each quotient $R_{\mathsf{h}}(H^\wedge_{\hbar,c})/(\mathsf{h}^n)$. The localization is easily seen to be independent of the lift $\tilde{f}$, denote it by $R_{\mathsf{h}}(H^\wedge_{\hbar,c})/(\mathsf{h}^n)[f^{-1}]$. The localizations inherit a grading from the grading on
$R_{\mathsf{h}}(H^\wedge_{\hbar,c})/(\mathsf{h}^n)$ that comes from the Rees construction. The graded algebras
$R_{\mathsf{h}}(H^\wedge_{\hbar,c})/(\mathsf{h}^n)[f^{-1}]$ form a projective system with respect to $n$. So we can consider the inverse limit in the category of graded algebras. Denote this inverse limit by
$R_{\mathsf{h}}(H^\wedge_{\hbar,c})[f^{-1}]$.
Set
$$H^\wedge_{\hbar,c}[f]^{-1}:=
R_{\mathsf{h}}(H^\wedge_{\hbar,c})[f^{-1}]/(\mathsf{h}-1).$$

The algebras $H^\wedge_{\hbar,c}[f^{-1}]$ glue to a sheaf with sections on $\C^\times$-stable open subsets. We denote this sheaf by $H_{\hbar,c}^{\wedge,loc}$.
This sheaf is complete and separated with respect to the filtration induced from $H^\wedge_{\hbar,c}$. Its algebra of global sections coincides with $H^\wedge_{\hbar,c}$. We note that if $f$ is homogeneous with respect to the initial $\C^\times$-action on $H^\wedge_{\hbar,c}$, then $H^\wedge_{\hbar,c}[f]^{-1}$
inherits this action. So  $H^{\wedge,loc}_{\hbar,c}$ is a $\C^\times$-equivariant sheaf
of filtered algebras.

Now consider a graded $H^\wedge_{\hbar,d}$-$H^\wedge_{\hbar,0}$-bimodule $\mathcal{B}$. We can view it as a filtered $H^{\wedge}_{\hbar,d}$-$H^\wedge_{\hbar,0}$-bimodule by doing the same base change as with the algebra.  Consider the microlocalization $\mathcal{B}^{loc}$ of $\mathcal{B}$, a microlocal filtered sheaf
on $Y^{\wedge}_\hbar$, defined similarly to the $H^{\wedge,loc}_{\hbar,c}$. The sections are defined on $\C^\times$-stable Zariski open
subsets, while the filtration is complete and separated. In particular, the space of sections
on any open subset inherits the filtration, and this filtration is complete and separated. Note that $\mathcal{B}^{loc}$ is a sheaf of $H^{\wedge,loc}_{\hbar,d}$-$H^{\wedge,loc}_{\hbar,0}$-bimodules. We have an isomorphism
$\mathcal{B}\xrightarrow{\sim} \Gamma(\mathcal{B}^{loc})$
because $Y^\wedge_\hbar$ is an affine scheme, compare to \cite[Lemma 2.10]{Bez_Losev}.

We note that, similarly to $H_{\hbar,c}^{\wedge,loc}$, the sheaf $\mathcal{B}^{loc}$ still carries a
natural $\C^\times$-action that turns it into a $\C^\times$-equivariant $H^{\wedge,loc}_{\hbar,d}$-$H^{\wedge,loc}_{\hbar,0}$.

%The locus
%in $T^*(\h^\wedge)/W\times \mathbb{A}^1$, where the sheaves of algebras $\epsilon_- %H^\wedge_{\hbar,1}\epsilon_-$
%and $H^\wedge_{\hbar,1}$ fail to be Morita equivalent  coincides with
%$[T^*(\h^\wedge)/W]^{sing}\times \{0\}$.
%This follows from Lemma \ref{Lem:integr_spher}.
%Let $[T^*(\h^\wedge)/W\times \mathbb{A}^1]^{0}$ denote the complement.
%The restriction of thee bimodule $\mathcal{B}^{loc}$ to this locus will also be denoted by %$\mathcal{B}^{0}$. We get a natural bimodule homomorphism $\mathcal{B}\rightarrow %\Gamma(\mathcal{B}^{0})$.
Set
\begin{equation}\label{eq:gen_locus}
Y_\hbar^{\wedge,0}:=
Y^\wedge_\hbar\setminus [Y^{\wedge,sing}\times \{0\}].\end{equation}
Let $\mathcal{B}^0$ denote the restriction of $\mathcal{B}^{loc}$ to (\ref{eq:gen_locus}).
We get a natural homomorphism $\mathcal{B}\rightarrow \Gamma(\mathcal{B}^0)$.

\begin{Lem}\label{Lem:glob_sec_iso}
We have the following properties:
\begin{enumerate}
\item For any $\epsilon_-$-spherical parameter $c$, the microlocal sheaves of algebras $H_{\hbar,c}^{\wedge,0}$
and $\epsilon_-H_{\hbar,c}^{\wedge,0}\epsilon_-$ are Morita equivalent via the bimodule
$H_{\hbar,c}^{\wedge,0}\epsilon_-$.
\item For $\mathcal{B}=B^\wedge_{\hbar,d\leftarrow 0}$, the homomorphism $\mathcal{B}\rightarrow \Gamma(\mathcal{B}^{loc})$ is an isomorphism.
\end{enumerate}
\end{Lem}
\begin{proof}
Let us prove (1). The claim is equivalent to $H_{\hbar,c}^{\wedge,0}\epsilon_-H_{\hbar,c}^0=H_{\hbar,c}^{\wedge,0}$,
which, in its turn, is equivalent to the claim that $H^\wedge_{\hbar,c}/H^\wedge_{\hbar,c}\epsilon_-H^\wedge_{\hbar,c}$
is supported on $Y^{\wedge,sing}\times \{0\}$. First, the condition that
$c$ is $\epsilon_-$-spherical is equivalent to the claim that $H_{\hbar,c}/H_{\hbar,c}\epsilon_-H_{\hbar,c}$ is
$\hbar$-torsion. So the support of $H^\wedge_{\hbar,c}/H^\wedge_{\hbar,c}\epsilon_-H^\wedge_{\hbar,c}$
is contained in $Y^{\wedge}\times \{0\}$.
(1) follows because $H/H\epsilon_-H$ is supported on $Y^{sing}$.

Let us prove (2). Both $\mathcal{B},\Gamma(\mathcal{B}^0)$ come with complete and separated
filtrations.
The homomorphism  $\mathcal{B}\rightarrow
\Gamma(\mathcal{B}^0)$ is that of filtered bimodules. To show that it is an isomorphism
it is enough to check that the associated graded homomorphism
\begin{equation}\label{eq:assoc_gr_homom1} \gr\mathcal{B}\rightarrow
\gr\Gamma(\mathcal{B}^0)\end{equation}
is an isomorphism. We have
$\gr\mathcal{B}=B_{d\leftarrow 0}^\wedge\otimes \C[\hbar]$. Also
we have a natural inclusion
\begin{equation}\label{eq:assoc_gr_homom2}\gr\Gamma(\mathcal{B}^0)\hookrightarrow \Gamma(\gr \mathcal{B}^0),\end{equation}
and the composition of (\ref{eq:assoc_gr_homom1}) and (\ref{eq:assoc_gr_homom2})
is the natural homomorphism
\begin{equation}\label{eq:assoc_gr_homom3} \gr\mathcal{B}\rightarrow
\Gamma(\gr\mathcal{B}^0).\end{equation}
So, (2) will follow if we show that (\ref{eq:assoc_gr_homom3}) is an isomorphism.

Set
$$X^{\wedge}:=Y^\wedge\times_{Y}X, X^\wedge_\hbar:=X^\wedge\times \operatorname{Spec}(\C[\hbar]), X_\hbar^{\wedge,0}:=Y_\hbar^{\wedge,0}\times_{Y_\hbar^{\wedge}}X_\hbar^{\wedge}.$$
%Note that $\operatorname{codim}_{X^\wedge}X^{\wedge,sing}\geqslant 4$.
We have
\begin{itemize}
\item[($\heartsuit$)]
the complement to  $X_\hbar^{\wedge,0}$ in $X_{\hbar}^{\wedge}$ has codimension $2$.
\end{itemize}

Note that $\gr\mathcal{B}$ is the global sections of the vector bundle
$$\left(\Pro^{reg,*}\otimes \Ocal^{reg}(d)\otimes \Pro^{reg}\right)\boxtimes \Ocal_{\operatorname{Spec}(\C[\hbar])}$$
on $X^{\wedge,reg}_\hbar$, while $\Gamma(\gr \mathcal{B}^0)$ is the global sections of the same vector bundle restricted to
$X_\hbar^{\wedge,reg}\cap X_{\hbar}^{\wedge,0}$. Because of the codimension condition ($\heartsuit$),
(\ref{eq:assoc_gr_homom3}) is indeed an isomorphism.
%Consider the sheaf \begin{equation}\label{eq:sheaf}
%\iota_*(\Pro^{reg,*}_{\hbar,1}
%\otimes_{\Dcal^{reg}_{\hbar,1}}\Dcal^{reg}_{\hbar,1\leftarrow 0}\otimes_{\Dcal^{reg}_{\hbar,0}}
%\Pro^{reg}_{0,\hbar})^{y-fin},
%\end{equation}
%where now we take the locally finite section for
%the action of $\C^\times$ that is the identity on $\h$. This is a sheaf of graded $\C[\hbar]$-modules whose sections are defined on $\C^\times$-stable open affine subsets in $X$.
\end{proof}

%For $\mathcal{B}=B^\times_{\hbar,1\leftarrow %0}$, this homomorphism is an isomorphism. This %is   because
%$\Ecal_{\hbar,1\leftarrow 0}$ is a quantization %of a vector
%bundle, hence is  reflexive.

\begin{proof}[Proof of Proposition \ref{Prop:isom_extension}]
The proof is in several steps.

{\it Step 1}. We are going to produce a homomorphism $H^{BM}_{T\times \C^\times}(\Fl_{e_d})^\wedge
\rightarrow B^\wedge_{\hbar,d\leftarrow 0}$.
Consider the isomorphism
$$ \epsilon_-H^{T\times \C^\times}_{BM}(\Fl_{e_d})^{\wedge,0}\xrightarrow{\sim}
\epsilon_- B_{\hbar,d\leftarrow 0}^{\wedge,0} $$
induced by (\ref{eq:sign_spher_iso}). Thanks to (1) of Lemma
\ref{Lem:glob_sec_iso}, this isomorphism gives rise to
\begin{equation}\label{eq:local_iso}
H^{T\times \C^\times}_{BM}(\Fl_{e_d})^{\wedge,0}\xrightarrow{\sim}
B_{\hbar,d\leftarrow 0}^{\wedge,0}.
\end{equation}
Note that this isomorphism is $\C^\times$-equivariant, by the construction.
So we have homomorphisms
$$H^{BM}_{T\times \C^\times}(\Fl_{e_d})^\wedge\rightarrow
\Gamma(H^{BM}_{T\times \C^\times}(\Fl_{e_d})^{\wedge,0})\xrightarrow{\sim}\Gamma(B^{\wedge,0}_{\hbar,d\leftarrow 0})
\xrightarrow{\sim} B^\wedge_{\hbar,d\leftarrow 0}.$$
The first homomorphism is the natural one, see the discussion before
Lemma \ref{Lem:glob_sec_iso}, the second is obtained from (\ref{eq:local_iso})
by passing to the global sections, while the third is the inverse of the isomorphism
in (2) of Lemma \ref{Lem:glob_sec_iso}.
The composed homomorphism is graded and $H^\wedge_{\hbar,d}$-$H^\wedge_0$-bilinear by the construction.
We need to show that it is an isomorphism.

{\it Step 2}. Our proof of this  is based on the following easy general fact:
let $M,N$ by two $\Z_{\geqslant 0}$-filtered vector spaces. Let
$\varphi:M\rightarrow N$ be an isomorphism mapping $M^{\leqslant i}
\rightarrow N^{\leqslant i}$ for all $i$. If $\gr\varphi:
\gr M\rightarrow \gr N$ is injective, then it is an isomorphism
(and hence $\varphi$ intertwines the filtrations).

{\it Step 3}. We apply the observation of Step 2
to the homomorphism
$$H_{BM}^{T\times \C^\times}(\Fl_{e_d})^\wedge\rightarrow B_{\hbar,d\leftarrow 0}^\wedge $$
specialized at $\hbar=1$. Denote this specialization by $\varphi$. It is an isomorphism
by (1) of Lemma \ref{Lem:glob_sec_iso} and is filtered by the construction. To show that $\gr\varphi$ is injective
we observe that $H^{BM}_{T}(\Fl_{e_d})$ is free over $\C[\h^*]$, this follows from (2) of Proposition \ref{GKM:results}. It follows that
$H^{BM}_{T}(\Fl_{e_d})^\wedge$ is flat over $\C[\h^*]^{\wedge}$.
%The homomorphism $\gr \varphi$ is intertwined %with a natural map
%$H^{BM}_{T}(\Fl_e)^\wedge\rightarrow
%\Gamma(H^{BM}_{T}(\Fl_e)^{\wedge,0})$ by an %isomorphism.
By the assumption of the proposition,
$\gr\varphi$ gives  an isomorphism between the sign-invariant parts.
It follows from (1) of Lemma \ref{Lem:glob_sec_iso} that $\gr\varphi$ is an isomorphism  over $(\h^{*\wedge}/W)^{reg}$. Since $H^{BM}_{T}(\Fl_{e_d})^\wedge$ is flat over $\h^{*\wedge}/W$, we see that $\gr\varphi$ is injective.

This completes the proof.
\end{proof}

%We consider $H^\wedge_{\hbar,d}$-$H^\wedge_{\hbar,0}$-bimodule
%$B^\wedge_{\hbar,d\leftarrow 0}$. It follows from Lemma \ref{Lem:bimod_sph_isom} that
% $$\epsilon_-B^\wedge_{\hbar,d+1\leftarrow 0}\cong \epsilon B^\wedge_{\hbar,d\leftarrow 0},$$
%a $(\C^\times)^2$-equivariant isomorphism of $\epsilon %H^\wedge_{\hbar,d}\epsilon$-$H^\wedge_{\hbar,0}$-bimodules.

%\begin{Claim}
%$H^{BM}_{T\times \C^\times}(\Fl_{e_d})$ is a $\Hcal^\times_{d,\hbar}$-$\Hcal^\times_{0,\hbar}$-bimodule.\textcolor{red}{[I think we should delete this claim as it is proven above]}
%\end{Claim}

\begin{proof}[Proof of Theorem \ref{Thm:iso_deformed}]
We prove the theorem by induction on $d$.
We have an isomorphism
$$B_{\hbar,0\leftarrow 0}^\wedge\xrightarrow{\sim} H^{T\times \C^\times}_{BM}(\Fl_{e_0})^\wedge$$
by the remark after the theorem. The proof of the theorem is now in several steps.

{\it Step 1}. Suppose we already have a graded bimodule isomorphism
$$B_{\hbar,d\leftarrow 0}^\wedge\xrightarrow{\sim} H^{T\times \C^\times}_{BM}(\Fl_{e_d})^\wedge$$
for some $d\geqslant 0$. Multiply by $\epsilon$ on the left.  Thanks to Lemma
\ref{Lem:bimod_sph_isom},
we have a graded algebra isomorphism $\epsilon_-H_{\hbar,d+1}^\wedge\epsilon_-\cong \epsilon H_{\hbar,d}^\wedge\epsilon$
and a graded $\epsilon H^\wedge_{\hbar,d}\epsilon$-$H^\wedge_{0,\hbar}$-bimodule isomorphism
$\epsilon_- B^\wedge_{\hbar,d+1\leftarrow 0}
\xrightarrow{\sim} \epsilon B^\wedge_{\hbar,d\leftarrow 0}$.
On the other hand, by Lemma \ref{Claim:iso_spherical_general}, we have a graded $H^\wedge_{\hbar,0}$-linear isomorphism
\begin{equation}\label{eq:spher_isom_11}\epsilon_- H^{BM}_{T\times\C^\times}(\Fl_{e_{d+1}})^\wedge\xrightarrow{\sim} \epsilon H^{BM}_{T\times\C^\times}(\Fl_{e_d})^\wedge.
\end{equation}
We are not going to check that this isomorphism is also $\epsilon H^\wedge_{\hbar,d}\epsilon$-linear. Instead, we will see that it is semilinear with respect to an automorphism of $\epsilon H^\wedge_{\hbar,d}\epsilon$ given by conjugation with an invertible element of $\C[\h^*]^{\wedge_0}$.
%Then we can twist with this element and achieve that (\ref{eq:spher_isom_11}) is a graded bimodule isomorphism.

{\it Step 2}. We claim that the homomorphism
\begin{equation}\label{eq:endom_isom1}\epsilon H^{\wedge}_{\hbar,d}\epsilon \rightarrow
\operatorname{End}_{H^\wedge_{\hbar,0}}(\epsilon B^\wedge_{\hbar,d\leftarrow 0})
\end{equation}
is an isomorphism. From Lemma \ref{Lem:endomorphism} we deduce that
\begin{equation}\label{eq:endom_isom2}
 \epsilon H_{\hbar,d}\epsilon \xrightarrow{\sim}
\operatorname{End}_{H_{\hbar,0}}(\epsilon B_{\hbar,d\leftarrow 0}).\end{equation}
Note that $B_{\hbar,d\leftarrow 0}$ is a finitely generated right $H_{\hbar,0}$-module. Using this and the fact that $\C[\h^*]^{\wedge_0}$ is a flat $\C[\h^*]$-module, we see that (\ref{eq:endom_isom2})
implies (\ref{eq:endom_isom1}).

Recall that
$$\epsilon B_{\hbar,d\leftarrow 0}^\wedge\xrightarrow{\sim} \epsilon H^{T\times \C^\times}_{BM}(\Fl_{e_{d}})^\wedge.$$
It follows that  isomorphism  (\ref{eq:spher_isom_11}) becomes $\epsilon H^{\wedge}_{\hbar,d\leftarrow 0}\epsilon$-linear  after we twist one of the actions by a  uniquely determined graded $\C[\hbar]$-linear automorphism of the algebra $\epsilon H^{\wedge}_{\hbar,d\leftarrow 0}\epsilon$. We denote this automorphism by $\zeta$. We claim that there is an invertible element $F\in \C[\h^{*\wedge}]^W$ such that $\zeta$ is the conjugation with $F$.

{\it Step 3}. The formula for $\Upsilon$ in
the proof of Lemma \ref{Claim:iso_spherical_general} implies that $\Upsilon$ modulo $\hbar$ is $\C[T^*T^\vee]^W$-linear. It follows that $\zeta$ is the identity modulo $\hbar$. So $\zeta=\exp(\hbar \partial)$, where $\partial$ is a derivation of
$\epsilon H^\wedge_{\hbar,d} \epsilon$ that has degree $-1$
with respect to the grading.  We have $\partial=\frac{1}{\hbar}[f,\cdot]$ for some $f\in \epsilon H^\wedge_{\hbar,d} \epsilon$. This follows because every Poisson derivation of $\C[Y^\wedge]$
is restricted from a $W$-equivariant Poisson derivation of $\C[T^*\h^{*,\wedge}]$ and hence is inner. Then $f\in \C[\h^{*\wedge}]^ W$ because $f$
has degree $0$. We set $F:=\exp(f)$.
Then we can compose (\ref{eq:spher_isom_11}) with the multiplication by $F$ and achieve that (\ref{eq:spher_isom_11}) is a graded bimodule isomorphism.

{\it Step 4}.  We now have  graded $\epsilon_-H^\wedge_{\hbar,d+1}\epsilon_-$-$H^\wedge_{\hbar,0}$-bimodule isomorphisms
$$\epsilon_-B_{\hbar,d+1\leftarrow 0}^\wedge\xrightarrow{\sim}
\epsilon B_{\hbar,d\leftarrow 0}^\wedge\xrightarrow{\sim}
\epsilon H^{T\times \C^\times}_{BM}(\Fl_{e_{d}})^\wedge
\xrightarrow{\sim}
\epsilon_-H^{T\times \C^\times}_{BM}(\Fl_{e_{d+1}})^\wedge.$$
Applying Proposition \ref{Prop:isom_extension}, we extend the composed isomorphism to a graded $H^\wedge_{\hbar,d+1}$-$H^\wedge_{\hbar,0}$-bimodule isomorphism
$$B_{\hbar,d+1\leftarrow 0}^\wedge\xrightarrow{\sim} H^{T\times \C^\times}_{BM}(\Fl_{e_{d+1}})^\wedge.$$
This finishes the proof of the induction step and hence of the theorem.
%Now we argue by
%induction on $d$. Suppose that we have (\ref{eq:sign_spher_iso}). Then we get
%a $\C^\times$-equivariant bilinear isomorphism
%$$B_{\hbar,d\leftarrow 0}^\wedge\xrightarrow{\sim} H^{T\times \C^\times}_{BM}(\Fl_{e_d})^\wedge.$$
%It restricts to
%$$\epsilon B_{\hbar,d\leftarrow 0}^\wedge\xrightarrow{\sim} \epsilon H^{T\times \C^\times}_{BM}(\Fl_{e_d})^\wedge.$$
%Thanks to Lemma \ref{Lem:bimod_sph_isom}, we have
%$$\epsilon_- B_{\hbar,d+1\leftarrow 0}^\wedge\cong \epsilon B_{\hbar,d\leftarrow 0}^\wedge$$
%and, thanks to Claim \ref{Claim:iso_spherical_general}, we get
%\begin{equation}\label{eq:Springer_iso}\epsilon_-H^{T\times \C^\times}_{BM}(\Fl_{e_{d+1}})^\wedge\cong
%\epsilon H^{T\times \C^\times}_{BM}(\Fl_{e_d})^\wedge.
%\end{equation}
%This isomorphism is graded and $H^\wedge_{\hbar,0}$-linear.
%Observe that the centralizer of $H^{\wedge}_{\hbar,0}$ in
%$\epsilon H^{T\times \C^\times}_{BM}(\Fl_{e_d})^\wedge\cong
%\epsilon B_{\hbar,d\leftarrow 0}^\wedge$ coincides with
%$\epsilon H^\wedge_{d,\hbar}\epsilon\cong \epsilon_-H^{\wedge}_{d+1,\hbar}\epsilon_-$.
%Now we can argue just like after Lemma \ref{Lem:automorphisms} and twist
%(\ref{eq:Springer_iso}) to make it bilinear (and still graded).
%We conclude that we have a graded $\epsilon_-H^\wedge_{d+1,\hbar}\epsilon_-$-$H^\wedge_{0,\hbar}$-linear isomorphism
%$$\epsilon_- B_{\hbar,d+1\leftarrow 0}^\wedge\xrightarrow{\sim} \epsilon_- H^{T\times \C^\times}_{BM}(\Fl_{e_{d+1}})^\wedge.$$
%This finishes the induction step.
\end{proof}

\begin{Cor}\label{Cor:Procesi_flatness}
 $B_{\hbar,d}$ is flat over
$\C[\h^*][\hbar]$.
\end{Cor}
\begin{proof}
     The bimodule $B_{\hbar,d}$ is bigraded, Remark \ref{Rem:bigraded}, and, thanks to Theorem \ref{Thm:iso_deformed} combined with Proposition \ref{GKM:results}, $B_{\hbar,d}^{\wedge_0}$ is flat over
$\C[\h^{*,\wedge}][\hbar]$. The claim
of the corollary follows.
\end{proof}

\begin{Rem}\label{Rem:bimod_isom}
In fact, the proof of Theorem
\ref{Thm:iso_deformed} gives us a characterization of the family of bimodules $B_{\hbar,d}$ for $d\geqslant 0$. Suppose we have another family of finitely generated graded $H_{\hbar,d}$-$H_{\hbar,0}$-bimodules $B'_{\hbar,d}$ satisfying the following conditions:
\begin{itemize}
    \item[(i)] $B'_{\hbar,d}$ is flat over $\C[\h^*][\hbar]$ for all $d$.
    \item[(ii)] $B'_{\hbar,0}$ is isomorphic to $H_{\hbar,0}$ as a graded $H_{\hbar,0}$-bimodule.
    \item[(iii)] We have an isomorphism of graded right $H_{\hbar,0}$-modules $\epsilon B'_{\hbar,d}\cong \epsilon_-B'_{\hbar,d+1}$.
\end{itemize}
Then the argument of the proof of Theorem \ref{Thm:iso_deformed} shows that for all $d$ we have a graded $H_{\hbar,d}$-$H_{\hbar,0}$-bimodule isomorphism $B'_{\hbar,d}\xrightarrow{\sim} B_{\hbar,d}$. Moreover, if we require the isomorphisms in (ii) and (iii) to be bigraded, then we get a bigraded isomorphism $B'_{\hbar,d}\xrightarrow{\sim} B_{\hbar,d}$. In fact, the proof simplifies: $\zeta$ from Step 2 of the proof of Theorem \ref{Thm:iso_deformed} is automatically the identity.

This observation will be applied in a subsequent paper  to establish, for $G=\operatorname{GL}_n$, an isomorphism between $B_{\hbar,d\leftarrow 0}$ and $H^{BM}_{T\times\C^\times}(\Fl_{e_d}^+)$, where  $\Fl_{e_d}^+$ is the ``polynomial part'' of $\Fl_{e_d}$.
\end{Rem}

%Now we can prove
%$$B^\wedge_{\hbar,d\leftarrow 0}
%\cong H^{BM}_{T\times \C^\times}(\Fl_{e_d})^\wedge,$$
%a graded $H^\wedge_{\hbar,d}$-$H_{\hbar,0}$-bilinear isomorphism,
%by induction on $d$. The base is $d=1$ considered in the previous section. Now suppose that we know this %isomorphism for
%a given $d$ and want to prove it for $d+1$.
%e get
%$$\epsilon_-B^{\wedge}_{\hbar,d+1\leftarrow 0}
%\cong \epsilon_-H^{BM}_{T\times \C^\times}(\Fl_{e_{d+1}})^\wedge$$
%and then argue exactly as in the previous section getting our isomorphism for $d+1$.

\subsection{Proof of Theorem \ref{Thm:dim}}
Recall that we are going to prove that
$$B_d\otimes_{H}\C_{triv}\cong
H_T^{BM}(\Fl_{e_d})\otimes_{H^\times}\C_{triv}=\C(\Lambda_0/(dh+1)\Lambda_0),$$
where $h$ denotes the Coxeter number of
$W$ and $\Lambda_0$ is the root lattice of $\g$. We write $\C_{triv}$ for the one-dimensional trivial $W$-module, and we assume that $\C[\h^*\oplus \h]\subset H$ acts on $\C_{triv}$ via the specialization
to $0$, while $\C[T^\vee\times \h]\subset H^\times$ acts on $\C_{triv}$ via the specialization
to $(1,0)$.

We already know that the dimensions are the same
thanks to Theorem \ref{Thm:iso}. We will prove that
\begin{align}\label{eq:lower_bound}& B_d\otimes_{H}\C_{triv}\twoheadrightarrow \C(\Lambda_0/(dh+1)\Lambda_0),\\\label{eq:upper_bound}
&\dim H_T^{BM}(\Fl_{e_d})\otimes_{H^\times}\C_{triv}
\leqslant (dh+1)^{\dim \h}.
\end{align}
This will prove Theorem \ref{Thm:dim}.

We first establish (\ref{eq:lower_bound}).

\begin{Prop}\label{Prop:lower_bound}
We have $ B_d\otimes_{H}\C_{triv}\twoheadrightarrow \C(\Lambda_0/(dh+1)\Lambda_0)$,
an epimorphism of $W$-modules.
\end{Prop}
\begin{proof}
Consider the $H_{d+1/h}$-$H_{1/h}$-bimodule
$B_{\hbar,d+1/h\leftarrow 1/h}$. By the construction in
Section \ref{SS_Cherednik_deform_constr}, this is a $\C[\hbar]$-flat bimodule with
\begin{equation}\label{eq:quotient}
    B_{\hbar,d+1/h\leftarrow 1/h}/\hbar B_{\hbar,d+1/h\leftarrow 1/h}\xrightarrow{\sim} B_d.
\end{equation}
Set
\begin{equation}\label{eq:specialization}
    B_{d+1/h\leftarrow 1/h}:=B_{\hbar,d+1/h\leftarrow 1/h}/(\hbar-1)B_{\hbar,d+1/h\leftarrow 1/h}.
\end{equation}
Since $B_{\hbar,d+1/h\leftarrow 1/h}$ is flat over $\C[\hbar]$, (\ref{eq:quotient}) is equivalent to $\gr B_{d+1/h\leftarrow 1/h}=B_d$.
Recall, Proposition \ref{Prop:BEG_comput}, that $H_{d+1/h}$ has a unique finite dimensional representation to be denoted by $L_{d+1/h}$.
By the proposition, this representation is isomorphic to $\C(\Lambda_0/(dh+1)\Lambda_0)$ as a $W$-representation.
In particular,  $L_{1/h}$ is the trivial one-dimensional representation of $W$.
The subspaces $\h,\h^*\subset H_{1/h}$ act
by $0$ on $L_{1/h}$.
Equip $B_{d+1/h\leftarrow 1/h}\otimes_{H_{1/h}}L_{1/h}$ with the tensor product filtration. Then we have
$$B_d\otimes_{H}\C_{triv}\twoheadrightarrow
\gr(B_{d+1/h\leftarrow 1/h}\otimes_{H_{1/h}}L_{1/h}).$$
To show that $\dim B_d\otimes_{H}\C_{triv}
\twoheadrightarrow \C(\Lambda_0/(dh+1)\Lambda_0)$ it is therefore sufficient to show that
\begin{equation}\label{eq:rep_iso}
B_{d+1/h\leftarrow 1/h}\otimes_{H_{1/h}}L_{1/h} \cong L_{d+1/h}.
\end{equation}
Thanks to Proposition \ref{Prop:BEG_comput},
(\ref{eq:rep_iso}) will follow once we show that $B_{d+1/h\leftarrow 1/h}$ is a Morita equivalence bimodule. We will prove this
by induction on $d$ starting with $d=0$, where
$B_{1/h\leftarrow 1/h}=H_{1/h}$ and the claim is vacuous.

Suppose we already know that $B_{d+1/h\leftarrow 1/h}$ is a Morita equivalence bimodule.
Since $d+1/h$ is $\epsilon$-spherical,
see Proposition \ref{Prop:spheric_param},
we see that $\epsilon B_{d+1/h\leftarrow 1/h}$
is a Morita equivalence bimodule between
$H_{1/h}$ and $\epsilon H_{d+1/h\leftarrow 1/h}
\epsilon$.
It follows from (2) of Lemma \ref{Lem:bimod_sph_isom} that
we have  a bimodule isomorphism $$\epsilon_- B_{d+1+1/h\leftarrow 1/h}\cong \epsilon B_{d+1/h\leftarrow 1/h}.$$
So $\epsilon_- B_{d+1+1/h\leftarrow 1/h}$ is a Morita equivalence bimodule between $\epsilon_- H_{d+1+1/h}\epsilon_-$ and $H_{1/h}$.
But, according to Proposition \ref{Prop:spheric_param}, $d+1+1/h$ is $\epsilon_-$-spherical so $B_{d+1+1/h\leftarrow 1/h}$ is also a Morita equivalence bimodule between $H_{d+1+1/h}$ and $H_{1/h}$. This finishes the proof.
\end{proof}

Now we proceed to the upper bound. The proof here is an easy generalization of a proof due to the first named author joint with Bezrukavnikov, Shan and Vasserot in \cite{BBASV}, but we include it here for completeness.

\begin{Prop}\label{Prop:upper_bound}
We have $\dim H_T^{BM}(\Fl_{e_d})\otimes_{H^\times}\C_{triv}\leqslant (dh+1)^{\dim \h}$.
\end{Prop}
\begin{proof}
The proof is in several steps. Note that it is enough to assume that $G$ is simply connected and hence $\widetilde{W}=W^a$. For example, this follows from the isomorphism $H^{BM}_T(\Fl_{e_d})\otimes_{H^\times}\C_{triv}\cong B_d\otimes_{H}\C_{triv}$ as the right hand side manifestly depends only on $W$.

{\it Step 1}.
We note that $H^{BM}_T(\Fl_{e_d})\otimes_{H^\times}\C_{triv}$ is nothing else as the space of coinvariants $H^{BM}(\Fl_{e_d})_{\widetilde{W}}$ for the action of the affine Weyl group $\widetilde{W}$ on $H^{BM}(\Fl_{e_d})$.
Recall, Proposition \ref{GKM:results}, that the affine Springer fiber $\Fl_{e_d}$ has a paving by affine cells. Each cell is the intersection of $\Fl_{e_d}$ with Schubert cells by \cite[Theorem 0.2]{GKM1}. This gives a basis in $H^{BM}(\Fl_{e_d})$ consisting of the fundamental classes of cells.

We will study the action of $\widetilde{W}$ on this basis to get a spanning set of $H^{BM}(\Fl_{e_d})_{\widetilde{W}}$ with $(dh+1)^{\dim \h}$ elements.

{\it Step 2}.
Let us introduce some notation.
In this proof $\mathfrak{b}$
will denote the Lie algebra of the Iwahori subgroup $\mathfrak{B}\subset G(\mathcal{K})$. For the Schubert cell   $\mathfrak{B}x\mathfrak{B}/\mathfrak{B}$, we denote the corresponding basis element in $H^{BM}(\Fl_{e_d})$ (or $H^{BM}_T(\Fl_{e_d})$) by $\varphi_x$.
For $x\in \widetilde{W}$ we will write $\,^x\mathfrak{b}$ for $\operatorname{Ad}(\dot{x})\mathfrak{b}$ for a lift $\dot{x}$ of $x$ to the normalizer of $T(\mathcal{K})$. Also for a $T(\mathcal{O})$-stable subset $Z\subset \Fl$ we use the notation $\,^x\!Z$ for $\dot{x}Z$, this is well-defined. Finally, we set $e_d^x:=\operatorname{Ad}(\dot{x})^{-1}(e_d)$. We note that $\,^x\!\Fl_{e_d}=\Fl_{e^x_d}$ for all $x\in \widetilde{W}$. Finally, for $x\in\widetilde{W}$ we will write $A_x$ for the corresponding (closed) alcove in $\h_{\mathbb{R}}$.

{\it Step 3}. For $w\in \widetilde{W}$, consider the subvariety $\Fl^{\leq w}_{e_d}=\Fl_{e_d}\cap \sqcup_{x\leq w}\mathfrak{B}x\mathfrak{B}/\mathfrak{B}$ of $\Fl_{e_d}$. It is $T\times\C^\times$-stable. The Borel-Moore homology $H^{BM}_T(\Fl^{\leq w}_{e_d})\subset H^{BM}_T(\Fl_{e_d})$ is spanned by the classes $\varphi_x$ for $x\leq w$ as a $H^*_T(pt)$-module. The image $\iota(H^{BM}_T(\Fl^{\leq w}_{e_d}))$ is precisely the subset of $\iota(H^{BM}_T(\Fl_{e_d}))$ consisting of all elements $(g_y)_{y\in \widetilde{W}}$ that satisfy that $g_y\neq 0\Rightarrow y\leq w$. This follows by applying Proposition \ref{BMcoho.description} to the space $\Fl^{\leq w}_{e_d}$. Further, note that we have the long exact sequence
$$\dots\rightarrow H^{BM,i}_{T}(\Fl^{< w}_{e_d})\rightarrow H^{BM,i}_T(\Fl^{\leq w}_{e_d})\rightarrow H^{BM,i}_T(\Fl_{e_d}\cap \mathfrak{B}w\mathfrak{B})\rightarrow\dots,$$
where the superscript $i$ indicates the cohomological grading.

Note that odd homology vanishes as all spaces involved have affine pavings and so the long exact sequence breaks up into short exact sequences. Assembling these exact sequences for all degrees we get
$$0\rightarrow H^{BM}_{T}(\Fl^{< w}_{e_d})\rightarrow H^{BM}_T(\Fl^{\leq w}_{e_d})\rightarrow H^{BM}_T(\Fl_{e_d}\cap \mathfrak{B}w\mathfrak{B})\rightarrow 0.$$
Further, by construction $\varphi_w$ is mapped to the basis element spanning $H^{BM}_T(\Fl_{e_d}\cap \mathfrak{B}w\mathfrak{B})$. Using compatibility with the localization map
noted above in this step and
the description of the Borel-Moore homology for an affine space with a linear $T$-action, we see that that $\iota(\varphi_w)_w=\frac{1}{\prod \chi}$, where the product is over all characters $\chi$ appearing in the $T$-representation $\Fl_{e_d}\cap \mathfrak{B}w\mathfrak{B}$.

{\it Step 4}.
Pick a simple affine reflection $s\coloneqq s_\alpha$ at a root $\alpha$. We want to get a necessary and sufficient condition on $x$ for $\varphi_{sx}=s\varphi_x+l.o.t.$  when $sx> x$ in the Bruhat order. Here ``$l.o.t.$'' indicates an $H^*_T(pt)$-linear combination of the elements $\varphi_y$ with $y< sx$
in the Bruhat order.  We claim that this equality holds  if  the cells $\,^s\!(\Fl_{e_d}\cap\mathfrak{B}x\mathfrak{B}/\mathfrak{B})$ and $\Fl_{e^s_d}\cap\mathfrak{B}sx\mathfrak{B}/\mathfrak{B}$ are equal. Indeed, if $\,^s\!(\Fl_{e_d}\cap\mathfrak{B}x\mathfrak{B}/\mathfrak{B})=\Fl_{e^s_d}\cap\mathfrak{B}sx\mathfrak{B}/\mathfrak{B}$, then $\iota(s\varphi_x)_{sx}=\iota(\varphi_{sx})_{sx}$ and so $\varphi_{sx}-s\varphi_x$ is a class in $H^{BM}_T(\Fl_{e_d}^{< sx})$ and thus a combination of  $\varphi_y$ with $y<sx$.
We conclude that the equality $\varphi_{sx}=s\varphi_x+l.o.t.$ also holds in $H^{BM}(\Fl_{e_d})$.

Note that, for all $x$, one of $\,^s\!(\mathfrak{B}x\mathfrak{B}/\mathfrak{B})$ and $\mathfrak{B}sx\mathfrak{B}/\mathfrak{B}$ contains the other. Therefore one of the two cells $\,^s\!(\Fl_{e_d}\cap\mathfrak{B}x\mathfrak{B}/\mathfrak{B})$ and $\Fl_{e^s_d}\cap\mathfrak{B}sx\mathfrak{B}/\mathfrak{B}$ contains the other.
Note that both cells are contracting loci for suitable tori actions. So they coincide if and only if their tangent spaces at their common $T\times\C^\times$-fixed point $sx$ are the same, equivalently, have the same dimension.

Note that the tangent space of $\Fl_{e_d}\cap\mathfrak{B}x\mathfrak{B}/\mathfrak{B}$ at the fixed point $x$  is $T\times\C^\times$-equivariantly isomorphic to
\begin{equation}\label{eq:tangent}
\frac{\mathfrak{b}\cap t^{-d}(\,^x\mathfrak{b})}{\mathfrak{b}\cap \,^x\mathfrak{b}}.
\end{equation}
So the tangent spaces of interest are
\begin{equation}\label{eq:tangent_spaces}
\frac{\mathfrak{b}\cap t^{-d}(\,^{sx}\mathfrak{b})}{\mathfrak{b}\cap \,^{sx}\mathfrak{b}},
\,^s\!\left(\frac{\mathfrak{b}\cap t^{-d}(\,^{x}\mathfrak{b})}{\,\mathfrak{b}\cap \,^x\mathfrak{b}}\right).
\end{equation}

The roots that appear as weights of (\ref{eq:tangent}) are exactly from
\begin{equation}\label{eq:root_tangent}
    R^+_{aff}\cap x\left(\bigsqcup_{1\leq r\leq d} (R^+-r\delta)\sqcup\bigsqcup_{0\leq r\leq d-1} (R^--r\delta)\right),
\end{equation}
where we write $R^+_{aff}$ for the set of positive affine roots, $R^+,R^-$ for the sets of positive and negative Dynkin roots, and $\delta$ for the indecomposable imaginary root. Note that every element in $R^+_{aff}\setminus \{\alpha\}$ appears as a weight in one of the spaces in (\ref{eq:tangent_spaces}) if and only if it appears in the other. On the other hand, $-\alpha$ does not appear as a weight in the first space and $\alpha$ does not appear as a weight of the second space. It thus follows that
$${}^s(\Fl_{e_d}\cap\mathfrak{B}x\mathfrak{B}/\mathfrak{B})=\Fl_{e^s_d}\cap\mathfrak{B}sx\mathfrak{B}/\mathfrak{B}$$
if and only if
\begin{equation}\label{equation:root}
\alpha\not\in  x(\bigsqcup_{1\leq r\leq d} (R^+-r\delta)\sqcup\bigsqcup_{0\leq r\leq d-1} (R^--r\delta)).
\end{equation}

%where $\delta$ is the character of $\mathbb{G}_m$, $R^+_{aff}$ the positive affine roots, $R^+$ the finite positive roots and $R^-$ the finite negative roots. This follows by checking when the root space for $\alpha_i$ is in the $T\times\C^\times$ representation
%$$\frac{\mathfrak{b}\cap t^{-d}({}^w\mathfrak{b})}{\mathfrak{b}\cap {}^w\mathfrak{b}}$$
%as the roots appearing in $t^{-d}({}^w\mathfrak{b})/{}^w\mathfrak{b}$ are precisely $w(\bigsqcup_{1\leq r\leq d} (R^+-r\delta)\sqcup\bigsqcup_{0\leq r\leq d-1} (R^--r\delta))$ and the roots for $R^+_{aff}$ is precisely the roots of $\mathfrak{b}$.\\

{\it Step 5}.
 In particular, if (\ref{equation:root}) holds, the projection of $\varphi_{sx}$ (for $s=s_\alpha$) to $H^{BM}(\Fl_{e_d})_{\widetilde{W}}$
 coincides with a linear combination of projections of $\varphi_y$ with $y<sx$.
Consider the equivalence relation on $\widetilde{W}$ generated by the relation $x\rightarrow s_\alpha x$ for
$\alpha$ satisfying (\ref{equation:root}).

In the next step we will prove that
\begin{itemize}
\item[(*)]
each equivalence class has a representative $x$ satisfying $\langle\alpha_i,A_x\rangle\geq -d$ and $\langle\alpha_0,A_w\rangle\leq d+1$,
\end{itemize}
where we write $\alpha_i$ for the simple Dynkin roots and  $\alpha_0$ for the longest root.

Showing (*) will finish the proof of the proposition because the set of alcoves $A$ satisfying  $\langle\alpha_i,A\rangle\geqslant -d$ and $\langle\alpha_0,A\rangle\leqslant d+1$ forms a poset ideal in the Bruhat order and has exactly $(dh+1)^{\dim \h}$ elements. To see the latter we argue as follows. Shifting by $d\rho^\vee$
we can instead consider the set of alcoves $A'$ satisfying $\langle\alpha_i,A'\rangle\geqslant 0$ and $\langle\alpha_0,A\rangle\leqslant d+1+d(h-1)=dh+1$. There are exactly $(dh+1)^{\dim \h}$ such alcoves.

{\it Step 6}. Fix an equivalence class for the equivalence relation
specified in Step 5 and pick a representative $x$ that is minimal with respect to the Bruhat order.
To show (*) it is enough to check that if $\langle\alpha_i,A_x\rangle\leqslant -d$, then  (\ref{equation:root}) holds for $x$ and $\alpha_i$ (and that the similar claim holds for the affine simple reflection). Indeed, since $\langle\alpha_i,A_x\rangle\leqslant  -d\leqslant 0$, we see that $s_ix$ is less than $x$ in the Bruhat order, while $s_i x$ is equivalent to $x$. This will give a contradiction with the choice of $x$.

We will only consider the case of simple Dynkin roots, the remaining case is similar.

%It follows that we can use the action of $s_i\coloneqq s_{\alpha_i}$ to send the Schubert cell basis element of $w$ to the one of $s_iw$ and $s_iw$ is shorter in the Bruhat order. Using induction on the Bruhat order, we can keep applying the action of $\widetilde{W}$ to a Schubert cell basis element until we get one for a Schubert cell labeled by $w$ satisfying the conditions $\langle\alpha_i,A_w\rangle\geq -d$ and $\langle\alpha_0,A_w\rangle\leq d+1$.\\

%Note that the condition that $s_iw<w$ in the Bruhat order is equivalent to $\langle\alpha_i,A_w\rangle\leq 0$. It is thus enough to check that if $\langle\alpha_i,A_w\rangle\leq -d$ then $\alpha_i\notin R^+_{aff}\cap w(\bigsqcup_{1\leq r\leq d} (R^+-r\delta)\sqcup\bigsqcup_{0\leq r\leq d-1} (R^--r\delta)$.\\
Assume $x=wt^\beta$ for $w\in W$ and $\beta\in \Lambda$. Then $\langle\alpha_i,A_x\rangle=\langle w^{-1}(\alpha_i),A_{1}+\beta\rangle$, thus  $\langle\alpha_i,A_x\rangle\leqslant -d$ holds  if and only if one of the following conditions hold
\begin{itemize}
    \item $w^{-1}(\alpha_i)\in R^+$ and $\langle w^{-1}(\alpha_i),\beta\rangle\leqslant -d-1$
    \item $w^{-1}(\alpha_i)\in R^-$ and $\langle w^{-1}(\alpha_i),\beta\rangle\leqslant -d$.
\end{itemize}
(\ref{equation:root}) follows from
$x^{-1}(\alpha_i)=w^{-1}(\alpha_i)+ \langle x^{-1}(\alpha_i),\beta\rangle\delta$.
%A similar analysis works for the affine simple reflection to check that condition (*) for the affine simple reflection $\hbar-\alpha_0$ is satisfied if $\langle\alpha_0,A_w\rangle\leq d+1$.\\
%It thus follows that under the action of $\widetilde{W}$ the Schubert cell basis elements labeled by $w$ satisfying $\langle\alpha_i,A_w\rangle\geq -d$ and $\langle\alpha_0,A_w\rangle\leq d+1$ generate $H^{BM}(\Fl)$, and thus these Schubert cell basis elements span $H^T_{BM}(\Fl_{e_d})\otimes_{H^\times}\C_{triv}$.\\
\end{proof}

\section{Applications}\label{S_small_quantum}
The goal of this section is to obtain some corollaries of Theorems \ref{Thm:iso}
and \ref{Thm:dim} for $d=1$, mostly in type $A$. We will write $e$ for $e_1$.

\subsection{Statements of the results}\label{SS_appl_statements}
Until the further notice, $\mathfrak{g}=\mathfrak{sl}_n$. Then  $W=S_n$, $X$ is the normalized version of the
Hilbert scheme (of dimension $2n-2$), and $\Pro$ is the restriction of Haiman's Procesi
bundle to $X\subset \operatorname{Hilb}_n(\C^2)=X\times \C^2$. Let $\tilde{X}$ denote
the preimage of $X$ in the isospectral Hilbert scheme, in other words, $\tilde{X}$
is $(\h\oplus \h^*)\times_Y X$ with its reduced scheme structure. Let $\zeta$
denote the natural finite morphism $\widetilde{X}\rightarrow X$. Haiman's $n!$ theorem
says that $\tilde{X}$ is a Cohen-Macaulay scheme, equivalently, $\zeta$ is flat (of degree $n!$).
The bundle $\Pro$ can be obtained as $\zeta_* \Ocal_{\widetilde{X}}$. That the bundle we consider
coincide with Haiman's follows, for example, from the main result of \cite{Losev_Procesi}.

Set $B^{sgn}:=\Gamma(\Pro\otimes \Pro)$, this is an $H^{\otimes 2}$-module
(equivalently, an $H$-bimodule).
It follows from Haiman's construction -- or the main result of \cite{Losev_Procesi} --
that $\Pro\cong \Pro^*\otimes \Ocal(1)$, a  $(\C^\times)^2\times S_n$-equivariant isomorphism,
where the action of $(\C^\times)^2\times S_n$ on $\Ocal(1)$ comes from the isomorphism
$\Ocal(1)\cong \Pro\epsilon_-$. It follows that $B^{sgn}$ is obtained from $B(:=B_1)$ by twisting
the left $S_n$-action with the sign.

In particular, $B^{sgn}$
has an algebra structure, in fact, this is the algebra $\C[\tilde{X}\times_X\tilde{X}]$.
Our first goal is to describe this algebra structure.

Consider the algebra
$$\tilde{B}:=\C[\h\oplus \h^*]
\otimes_{\C[Y]}
\C[\h\oplus \h^*].$$
Note that both $B^{sgn}$ and $\tilde{B}$ are graded $\C[\h\oplus\h^*]^{\otimes 2}$-algebras.

\begin{Thm}\label{Thm:B_structure}
We have a graded $\C[\h\oplus \h^*]^{\otimes 2}$-algebra isomorphism $B^{sgn}\cong \tilde{B}/\operatorname{rad} \tilde{B}$.
\end{Thm}

We can also describe the $\C[\h\oplus \h^*]$-bimodule structure on $B^{sgn}$.

\begin{Thm}\label{Thm:B_bimod}
We have a graded $\C[\h\oplus \h^*]$-bimodule isomorphism
$B^{sgn}\cong H \epsilon H$, where the latter is viewed as
a subbimodule in $H$.
\end{Thm}

\begin{Rem}\label{Rem:CM_conj}
We note that Theorems \ref{Thm:B_bimod}, \ref{Thm:iso} and Proposition \ref{Prop:cohom_vanishing}
imply \cite[Conjecture 3.7]{CM}. Namely, their $M$ is $H\epsilon H$, the higher cohomology
of $\Pro\otimes \Pro$ vanish thanks to Proposition \ref{Prop:cohom_vanishing}, and
the claim that $B$ is flat over $\C[\h^*]$ follows from Corollary
\ref{Cor:Procesi_flatness}.
\end{Rem}

%We deduce this from several known results due to Haiman and conjectural results of the first named author %(joint with Bezrukavnikov, Shan and Vasserot)\cite{BBASV}.

Now we proceed to prospective applications to the center of the principal block of  of the small quantum group. We assume that $\g$ is an arbitrary simple Lie algebra -- but we still get more complete results in type $A$.

Recall the notations $Z,G^\vee,T^\vee$ from Introduction. Also recall that $H^*_T(\Fl_e)\xrightarrow{\sim} \Hom_{\Ring}(H^{BM}_T(\Fl_e),\Ring)$, see Remark \ref{Rem:homology_duality}. This gives a $\widetilde{W}$-action 
on $H^*_T(\Fl_e)$ corresponding to the centralizer-monodromy action on $H^{BM}_T(\Fl_e)$. The $\widetilde{W}$-action on
$H^*_T(\Fl_e)$ gives rise to a $W$-action on $H^*(\Fl_e)^\Lambda$. 

The following conjecture is due to the first named author joint with Bezrukavnikov,
Shan and Vasserot, \cite{BBASV}.

\begin{Conj}\label{Thm:BBASV}
There is an algebra isomorphism
$H^*(\Fl_e)^\Lambda\xrightarrow{\sim} Z^{T^\vee}$. This isomorphism is $W$-equivariant, where on the left hand side  we have the action described above and on the right hand side the action comes
from the identification $W=N_{G^\vee}(T^\vee)/T^\vee$.
\end{Conj}

In fact, \cite{BBASV} establishes the existence of a $W$-equivariant algebra monomorphism. The conjectural part is that this monomorphism is surjective.  

%To do so we introduce notation $Z$ for the center of the regular block of the small quantum group. This %has a an action of the Langlands dual group $G^L$ of $G$, with a maximal torus $T^L$.
%\begin{Thm}
%There is a commutative diagram
%\[\begin{tikzcd}
%H^*(\mathcal{F}l)\ar[d]\ar[r]& Z^{G^\vee}\ar[d,hook]\\
% 	H^*(\mathcal{F}l_e)^{\Lambda}\ar[r]& Z^{T^\vee}
%\end{tikzcd}\]
%\end{Thm}
%\begin{Conj}\label{conj:center}
%There is a commutative diagram
%\[\begin{tikzcd},
%H^*(\mathcal{F}l_e)^{\widetilde{W}}\ar[d,%hook]\ar[r,"\sim"]& Z^{G^\vee}\ar[d,hook]\\
% 	H^*(\mathcal{F}l_e)^{\Lambda}\ar[r,"\sim"]& Z^{T^\vee}
%\end{tikzcd}\]
%\end{Conj}

Here is our result on the structure of $Z$.
\begin{Thm}\label{center:thm}
Assume Conjecture \ref{Thm:BBASV} holds. Then the following claims are true.
\begin{enumerate}
    \item The dimension of the subalgebra of $N_{G^\vee}(T^\vee)$-invariants in $Z$ is $(h+1)^{\dim \h}$.
    \item If $\g=\mathfrak{sl}_n$, then the $G^\vee$-action on $Z$ is trivial. In particular, $\dim Z=(n+1)^{n-1}$.
\end{enumerate}
\end{Thm}
Note that (2) confirms a conjecture from \cite{LY}.
%\begin{proof}
%    The first result follows from Theorem \ref{Thm:dim} and the following Lemma \ref{coho:duality}.\\
%    The second result follows from the Section \ref{n!connection} in the Appendix.
%\end{proof}

The following result is used to prove Theorems \ref{Thm:B_structure},\ref{Thm:B_bimod}
as well as (2) of Theorem \ref{center:thm}. Consider the 1-dimensional representation
$\C_0$ of $\C[\h\oplus \h^*]$ corresponding to the point $0\in \h\oplus \h^*$.

\begin{Prop}\label{Prop:trivial_action}
For $\g=\mathfrak{sl}_n$,
we have $B\otimes_{\C[\h\oplus \h^*]}\C_0=(B\otimes_{\C[\h\oplus \h^*]}\C_0)\epsilon$.
\end{Prop}

%\section{Appendix A (by I. Losev): Structure of $B$}
%\label{S_B_structure}

%Note that the theorem implies that the action of
%$S_n$ on $B/B\mathfrak{m}$ from the right is via the trivial representation. In particular,
%$\dim B/B\mathfrak{m}=(n+1)^{n-1}$. But we will first prove that the action is trivial and use this to deduce the theorem.

\subsection{Proposition \ref{Prop:trivial_action} and $n!$ theorem}\label{n!connection}
In this section we prove Proposition \ref{Prop:trivial_action}. In fact, we will show that
Proposition \ref{Prop:trivial_action} is equivalent to the $n!$ theorem of Haiman, \cite{Haiman}. We need some preparation
for the proof.

For a partition $\mu$ on $n$, let
$x_\mu$ denote the fixed point in $X$
labelled by $\mu$  and  $\Pro_\mu$ denote the fiber  of  $\Pro$ at $x_\mu$. This is a $(\C^\times)^2$-equivariant $H$-module
of dimension $n!$.
The following is  a consequence of the $n!$ theorem.
\begin{itemize}
    \item[(A)] For each $\mu$, the head of the $H$-module $\Pro_\mu$ is a trivial $S_n$-module.
\end{itemize}
In fact, more is true. If we use the Bezrukavnikov-Kaledin construction of $\Pro$ as a definition, then (A) is equivalent to the $n!$ theorem. Indeed, (A) implies the similar claim for all fibers of $\Pro$.
So $\Pro$ acquires a sheaf of algebras structure. The relative spectrum of $\Pro$ is easily seen to
coincide with $\tilde{X}$.

We will give several equivalent formulations of (A). Consider
the adjoint pair  $$\Loc:= \Pro\otimes_H \bullet: H\operatorname{-mod}
\rightleftarrows
\operatorname{Coh}(X):\tilde{\Gamma}:=\Hom_{\Ocal_X}(\Pro,\bullet).$$
Note that the derived functors $L\Loc,
R\tilde{\Gamma}$ are mutually quasi-inverse
equivalences, see, e.g., \cite[Proposition 2.2]{BK}.

Note that we can view every irreducible
representation $\tau$ of $S_n$ as an
irreducible $H$-module by making $\h\oplus \h^*$ act by $0$.

\begin{Lem}
(A) is equivalent to the following claim:
\begin{itemize}
    \item[(B)] For a
    nontrivial irreducible representation $\tau$ of $S_n$, we have
    $\Loc(\tau)=0$.
\end{itemize}
\end{Lem}
\begin{proof}
Let us write $\C_\mu$ for the skyscraper sheaf
at $x_\mu$. Then $\Pro_\mu^*=\tilde{\Gamma}(\C_\mu)$. Therefore
$$\Hom_H(\tau,\Pro_\mu^*)=\Hom_{\Ocal_X}( \Loc(\tau), \C_\mu).$$
So (A) is equivalent to the claim that
$\Hom_{\Ocal_X}(\Loc(\tau),\C_\mu)=0$
for all $\mu$ as long as $\tau\neq \operatorname{triv}$. Hence (B)$\Rightarrow$(A).
To show the implication in the opposite direction, we must show that for a nonzero $(\C^\times)^2$-equivariant coherent sheaf
$\mathcal{F}$ on $X$ there is a
partition $\mu$ such that $\Hom_{\Ocal_X}(\mathcal{F},\C_\mu)\neq 0$.  The action of $(\C^\times)^2$ contains a contracting one-dimensional subtorus whose fixed points are precisely the points $x_\mu$
for all $\mu$. So if
the fiber $\mathcal{F}_{x_\mu}$ is zero for all $\mu$, then every fiber of $\mathcal{F}$ is zero.
(A)$\Rightarrow$(B) follows.
\end{proof}

To prove Proposition \ref{Prop:trivial_action} we now need to show that (A)$\Leftrightarrow$(B)
is equivalent to the following condition:
\begin{itemize}
    \item[(C)] For an irreducible representation $\tau$ of $S_n$, $\Loc (\tau)\neq 0\Leftrightarrow
    B^{sgn}\otimes_H \tau\neq \{0\}$.
\end{itemize}

In the proof  we will need to following lemma.

\begin{Lem}\label{Lem:functor_formula}
We have an isomorphism of
endofunctors of $D^b(H\operatorname{-mod})$, $$R\tilde{\Gamma}(L\Loc(\bullet)(1))\cong R\Gamma(\Pro\otimes \Pro)\otimes^L_H\bullet.
    $$
\end{Lem}
\begin{proof}
This is standard: the left hand side is the derived tensor product with $$R\tilde{\Gamma}(L\Loc(H)(1))=R\Hom_{X}(\Pro, \Pro(1)).$$
The right hand side in the last equation is $B^{sgn}$.
\end{proof}

\begin{proof}[Proof of Proposition \ref{Prop:trivial_action}]
Now we show that (C) holds. Assume first that $B\otimes_H \tau=\{0\}$. Note that since the algebra $H$
has finite homological dimension, only finitely many of homologies of $L\Loc(\tau)$
are nonzero. Pick $m$ large enough for so that the sheaves $H_i(L\Loc(\tau))(m)$ are generated by their global sections and their higher cohomology groups vanish. By Lemma
\ref{Lem:functor_formula},
$$R\tilde{\Gamma}(L\Loc(\tau)(m))=
(B^{sgn})^{\otimes_H^L m}\tau.$$
The zeroth homology group of the right hand side is zero. By our choice of $m$
this implies that $\Loc(\tau)=0$.

Now assume that $\Loc(\tau)=0$. By the previous paragraph, for some $m$ we have
$(B^{sgn})^{\otimes_H m}\tau=0$. Let $S$ denote the set of all
irreducible $S_n$-representations $\tau$ such that $B^{sgn}\otimes_H \tau\neq \{0\}$.
Note that $\tau\in S$ if and only if
\begin{itemize}
\item[(*)]
$\tau$ appears in the $S_n$-module
$B^{sgn}/B^{sgn}(\h\oplus \h^*)$ (where $S_n$ acts from the right).
\end{itemize} But the $H$-actions on $B^{sgn}=\Gamma(\Pro\otimes\Pro)$ from the left and from the right are completely symmetric. So (*) is equivalent to the condition that $\tau$ appears in  $B^{sgn}/(\h\oplus \h^*)B^{sgn}$ (where $S_n$ acts from the left). The latter condition in its turn is equivalent to $\Hom_H(B^{sgn},\tau)\neq 0$. So we see that $\tau\in S$ if and only if $\tau$ appears in the head of some $H$-module of the form $B^{sgn}\otimes_H\tau'$ (where $\tau'$ is automatically in $S$).
This shows that $\tau\in S$ if and only if $(B^{sgn})^{\otimes_H m}\tau\neq 0$
for all $m$. This finishes the proof of (C)
and hence shows that the proposition is equivalent to the $n!$ theorem.
\end{proof}

\subsection{Proofs of Theorems \ref{Thm:B_structure},\ref{Thm:B_bimod}}\label{SS_proofs_type_A}
\begin{proof}[Proof of Theorem \ref{Thm:B_structure}]
{\it Step 1.} Here we prove that $B^{sgn}$ is reduced.
First of all, note that $B^{sgn}=\Gamma(\Pro\otimes \Pro)$ is nothing else but the algebra $\C[\tilde{X}\times_{X}\tilde{X}]$. The scheme $\tilde{X}\times_{X}\tilde{X}$ is flat and finite over the Cohen-Macaulay scheme $\tilde{X}$, hence is Cohen-Macaulay. It is generically reduced and therefore reduced. The algebra of regular functions on a reduced scheme is always reduced.

{\it Step 2.} Here we produce an algebra homomorphism
$\varphi:\tilde{B}\rightarrow B^{sgn}$. This comes  as the pullback of the morphism
$$\tilde{X}\times_{X}\tilde{X}\rightarrow
(\h\oplus \h^*)\times_{Y}(\h\oplus \h^*)$$
induced by the morphisms $\tilde{X}\rightarrow \h\oplus \h^*, X\rightarrow Y$.
Note that $\varphi$ is the unique $\C[\h\oplus \h^*]^{\otimes 2}$-algebra
homomorphism $\tilde{B}\rightarrow B^{sgn}$.

{\it Step 3}. We show that the homomorphism
$\varphi:\tilde{B}\rightarrow B^{sgn}$ is surjective.
This is a crucial step in the proof that
uses Proposition \ref{Prop:trivial_action}. Namely,
note that both $\tilde{B},B^{sgn}$ are
finitely generated graded  $\C[\h\oplus \h^*]^{\otimes 2}$-modules.
Let $\tilde{B}_0,B^{sgn}_0$ denote the specializations of $\tilde{B},B^{sgn}$
to $(0,0)\in (\h\oplus \h^*)^2$. We need to show that the induced algebra
homomorphism $\tilde{B}_0\rightarrow B^{sgn}_0$ is surjective. Clearly,
$\tilde{B}_0$ is one-dimensional. Now consider $B^{sgn}_0$. This space
is acted by $S_n$ on the left and on the right. Proposition
\ref{Prop:trivial_action} implies that the action from the right
is trivial. By symmetry, the action on the left is trivial as well.
By Theorem \ref{Thm:dim}, we have $B^{sgn}\otimes_H \C_{triv}\cong
\operatorname{sgn}\otimes \C(\Lambda_0/(n+1)\Lambda_0)$. The space
of $S_n$-invariants in the latter module is one-dimensional. So
$\dim B^{sgn}_0=1$ and our claim follows.

{\it Step 4}. It is easy to see that $\varphi:\tilde{B}\rightarrow B^{sgn}$ is an isomorphism over
$Y^{reg}$. Since $\varphi$ is surjective and $B^{sgn}$ is reduced, we conclude that
$\varphi$ induces an isomorphism $\tilde{B}/\operatorname{rad}\tilde{B}\xrightarrow{\sim} B^{sgn}$.
This completes the proof of Theorem \ref{Thm:B_structure}.
\end{proof}

\begin{proof}[Proof of Theorem \ref{Thm:B_bimod}]
We need to prove that
$B\cong H \epsilon_-H$.

According to \cite[Proposition 6.1.5]{Haiman_CDM}, the $\C[\h\oplus \h^*]$-module
$\Gamma(\Pro\otimes \Ocal(1))$
is identified with the ideal $J$ in
$\C[\h\oplus \h^*]$ generated
by the $\operatorname{sgn}$-invariant polynomials.
Therefore we get a graded bimodule homomorphism
$$B\rightarrow \Hom_{\C[Y]}
(\C[\h\oplus \h^*], J)$$
-- from the global sections of the sheaf Hom to the Hom between the global sections.
Composing this with the inclusion $J\hookrightarrow \C[\h\oplus \h^*]$ we get  a bimodule homomorphism
\begin{equation}\label{eq:bimodule_homom}
B\rightarrow \End_{\C[\underline{x},\underline{y}]^{S_n}}
(\C[\underline{x},\underline{y}])=H.
\end{equation}
For the latter equality, see, e.g., \cite[Theorem 1.5]{EG}.
By the construction,
$B$ is torsion free as a module over $\C[Y]$. Also over the localization
$\C[\h^*]^{reg}$ of $\C[\h^*]$ at the Vandermond determinant,
(\ref{eq:bimodule_homom}) becomes an isomorphism. We conclude that (\ref{eq:bimodule_homom}) is injective.
So $B$ is a two-sided ideal in $H$.

It follows from Theorem \ref{Thm:B_structure}, that the $\C[\h\oplus \h^*]$-bimodule $B$ is generated by
a single element in degree $0$ that is sign invariant.
The corresponding element in $\Gamma(\Pro\otimes \Pro)$ is the image of
the identity under the inclusion $\C[Y]$ arising from the direct summand
$\Ocal$ of $\Pro\otimes\Pro$. So the element in
 $B=\Hom_{\Ocal_X}(\Pro,\Pro(1))$ we need  is described as the decomposition
$\Pro\twoheadrightarrow \mathcal{O}(1)
\hookrightarrow \Pro(1)$, where the first map
is $\epsilon_-$ and the second is the inclusion of $\Ocal(1)$ into $\Pro$. The image of this
element in $H$
is $\epsilon_-$. We conclude that
$B\cong H\epsilon_-H$.
\end{proof}

\begin{Rem}\label{Rem:d_generalization}
By \cite[Proposition 6.1.5]{Haiman_CDM}, we have $\Gamma(\Pro\otimes \mathcal{O}(d))=J^d$. For the same reason as in the proof of the proposition, we get $B_d\hookrightarrow \Hom_{\C[Y]}(\C[\h\oplus \h^*],J^d)$.
\end{Rem}

\subsection{Proof of Theorem \ref{center:thm}}\label{SS_proof_center}
In the proof we will need the following three lemmas.
\begin{Lem}\label{coho:duality}
For any $\g$,  we have a $W$-equivariant identification (for the right action)
$$\left(H^*(\Fl_e)^{\Lambda}\right)^*\cong
B\otimes_{\C[\h\oplus \h^*]}\C_0.$$
\end{Lem}
\begin{proof}
Recall, Theorem \ref{Thm:iso}, that we have an $H^\wedge$-bimodule isomorphism $H^{BM}_T(\Fl_e)^\wedge\cong
B^\wedge$.
Also $H^{BM}_T(\Fl_e)/ H^{BM}_T(\Fl_e)\h^* \cong H^{BM}(\Fl_e)$. Next, we have an identification
$H^*(\Fl_e)\cong H^{BM}(\Fl_e)^*$, this was discussed in Section \ref{SS_BM_general} (in the equivariant setting). This identification is $\widetilde{W}$-equivariant. It follows that $H^*(\Fl_e)^{\Lambda}
\cong \left(H^{BM}(\Fl_e)_\Lambda\right)^*$, where the subscript $\Lambda$ indicates taking the  coinvariants. Note that
$$H^{BM}(\Fl_e)_\Lambda\xrightarrow{\sim} H^{BM}(\Fl_e)^\wedge/H^{BM}(\Fl_e)^\wedge\h.$$
The claim of the lemma follows.
\end{proof}

Part (1) of Theorem \ref{center:thm} follows from Lemma \ref{coho:duality} combined with Theorem \ref{Thm:dim}.

In the remainder of this section we will prove (2) of Theorem \ref{center:thm}.
Recall that $G=\operatorname{SL}_n$ and hence
$G^\vee=\operatorname{PGL}_n$.

\begin{Lem}\label{Lem:center_bound}
Let $\mu$ be a highest weight of $G^\vee$ in the center of $\mathfrak{u}_\epsilon$. Let $\ell$ denote the order of $\epsilon$, recall that it is an odd number. Then we have $\ell \mu\leqslant 2(\ell-1)\rho$ in the usual order on the dominant weights.
\end{Lem}
\begin{proof}
    We note that $2(\ell-1)\rho=(\ell-1)\sum_{\alpha>0}\alpha$ is the maximal weight in $\mathfrak{u}_\epsilon$. The action of the Lusztig form $\dot{U}_\epsilon$ on $Z$ factors through the quantum Frobenius epimorphism to give an action of $G^\vee$. The pullback inflates the weights $\ell$ times. This gives the required inequality.
\end{proof}

\begin{Lem}\label{Lem:wt_triv_action}
Let $V$ be an irreducible $\operatorname{PGL}_n$-module with the following property: the action of $S_n$ on the weight zero subspace, $V_0$, is trivial. Then $V\cong S^{2kn}(\C^n)$ or
$S^{2kn}(\C^n)^*$ for some $k\in \Z_{\geqslant 0}$.
\end{Lem}
\begin{proof}
In what follows it will be convenient to view $V$ as a representation of $\operatorname{GL}_n$. Our proof of the lemma is by induction on $n$.

The base is $n=2$, where our claim is easy. Now suppose it is proved for $n-1$, we are going to prove it for $n$. Let $\mu=(\mu_1,\ldots,\mu_n)$ be the highest weight of $V$.  The $\operatorname{GL}_{n-1}$-module with highest weight $\lambda=(\lambda_1,\ldots,\lambda_{n-1})$ occurs in the restriction of $V$ if and only if
\begin{equation}\label{eq:wt_ineq}
    \mu_1\geqslant \lambda_1\geqslant \mu_2\geqslant\ldots \geqslant \lambda_{n-1}\geqslant \mu_n.
\end{equation}
And this $\operatorname{GL}_{n-1}$-module intersects the zero weight space for $\operatorname{PGL}_n$ if and only if
\begin{equation}\label{eq:mean}
    \frac{\lambda_1+\ldots+\lambda_{n-1}}{n-1}=
    \frac{\mu_1+\ldots+\mu_n}{n}.
\end{equation}
Clearly, at least one $\lambda$ satisfying (\ref{eq:wt_ineq})
and (\ref{eq:mean}) exists.

Let $I$ be the set of indices $i\in \{1,\ldots,n-1\}$ such that $\mu_i>\mu_{i+1}$. Assume that $|I|>1$. The claim that a solution $\lambda$ to (\ref{eq:wt_ineq}) and (\ref{eq:mean}) satisfies the induction assumption easily implies  that one of the following possibilities holds:
\begin{enumerate}
    \item $\lambda_i=\mu_i$ for all $i\in I$
    \item $\lambda_i=\mu_{i+1}$ for all $i\in I$.
\end{enumerate}
Indeed, otherwise we can increase one component of $\lambda$ by $1$
and decrease another by $1$ so that (\ref{eq:wt_ineq}) continues to hold. But if $\lambda$ is the highest weight of $S^{2k(n-1)}(\C^{n-1})$ or its dual (up to a twist with a power of the determinant), then the modification is not of that form.

Replacing $V$ with $V^*$ if necessary we can assume that (1) holds. Also if $i\not\in I$, then $\lambda_i=\mu_i(=\mu_{i+1})$. So we have $\lambda_i=\mu_i$
for all $i=1,\ldots,n-1$. From (\ref{eq:mean}) we deduce $$\mu_n=\frac{\mu_1+\ldots+\mu_{n-1}}{n-1}.$$
Together with $\mu_1\geqslant\ldots\geqslant \mu_n$, this implies $\mu_1=\ldots=\mu_n$, a contradiction with $|I|>1$.

So $|I|=1$ meaning that $\mu$ has two different entries. Since $\lambda$ is the highest weight of $S^{2k(n-1)}(\C^{n-1})$ or its dual, this implies that $I=\{1\}$ or $I=\{n-1\}$, which, in turn, easily implies the claim of the lemma.
%while (\ref{eq:wt_ineq}) implies that
%\begin{equation}\label{eq:wt_ineq1}
%    \lambda_1\geqslant 2k(n-1), %\lambda_2+\ldots+\lambda_n\leqslant 2k.
%\end{equation}
\end{proof}

\begin{proof}[Proof of (2) of Theorem \ref{center:thm}]
    Recall that we have a $W$-equivariant isomorphism $H^*(\Fl_e)^\Lambda\xrightarrow{\sim} Z^{T^\vee}$ by Corollary \ref{Thm:BBASV}.
    Using Lemma \ref{coho:duality} combined with
    Proposition \ref{Prop:trivial_action}, we see that $S_n$ acts trivially on $Z^{T^\vee}$. By Lemma \ref{Lem:wt_triv_action}, all irreducible summands of the $\operatorname{PGL}_n$-module $Z$ are of the form
    $S^{2kn}(\C^n)$ or $S^{2kn}(\C^n)^*$. But for $k>0$,
    the highest weights $\mu$ of these modules do not satisfy the inequality of Lemma \ref{Lem:center_bound}. It follows that $Z$ is a trivial $\operatorname{PGL}_n$-module, implying the claim of the theorem.
\end{proof}

\section{Appendix (by P. Boixeda Alvarez, I. Losev, and O. Kivinen): Springer action on $H^{BM}_{T\times\C^\times}(\Fl_{e_d})$}
In this Appendix we include some of constructions and  proofs for Section \ref{S_BM_actions}.

\subsection{Reminder on the affine Springer action}\label{SS_affine_Springer_reminder}
In this section we recall the generalities on the affine Springer action. We use the notation from Section \ref{SS_BM_Springer}.

The action of $W^a$ was constructed in \cite[Section 5.4]{Lusztig}. 
To construct the operators corresponding to simple affine reflections we introduce certain auxiliary  spaces. For a parahoric subgroup $\mathfrak{P}$ of $G(\mathcal{K})$ containing $\mathfrak{B}$, we can consider the partial affine flag variety
 $$\Fl^\mathfrak{P}=G(\mathcal{K})/\mathfrak{P}.$$
 Using this space, we can introduce the affine Springer fibers in the partial flag variety $\Fl^\mathfrak{P}$:
 $$\Fl^\mathfrak{P}_{e_d}\coloneqq \{g\mathfrak{P}\in\Fl|Ad(g)^{-1}e_d\in \text{Lie}(\mathfrak{P})\}.$$

 Now we introduce certain stacks. To do this we need some notation. Let $L$ be the standard Levi subgroup of $\mathfrak{P}$. Let  $B_L$ denote the image of $\mathfrak{B}$ in $L$, this is a Borel subgroup of $L$.
 We write $\mathfrak{l},\mathfrak{b}_L$ for the Lie algebras of these groups.

 With this notation, we have a Cartesian diagram
\begin{equation}\label{eq:Springer_diagram}
\begin{tikzcd}
 	\Fl_{e_d} \arrow[r,"q_1"] \arrow[d,"\pi_2"]
 	& \mathfrak{b}_L/B_L \arrow[d,"\pi_1"] \\
 	\Fl_{e_d}^\mathfrak{P} \arrow[r,"q_2"]
 	& \mathfrak{l}/L
 \end{tikzcd}
 \end{equation}

The map $\Fl^\mathfrak{P}_{e_d}\rightarrow \mathfrak{l}/L$ sends $g\mathfrak{P}$ to the image of $Ad(g)^{-1}e_d$ and  $\operatorname{Lie}(\mathfrak{P})\rightarrow \mathfrak{l}$. The map  $\Fl_{e_d}\rightarrow \mathfrak{b}_L/B_L$
is defined in  a similar way.

Note that we have the following canonical isomorphisms of objects in the $T\times \C^\times$-equivariant derived category:
$$(\pi_2)_*(\omega_{\Fl_{e_d}})\xrightarrow{\sim}(\pi_2)_*(q_1^!(\C_{\mathfrak{b}_L/B_L}))\xrightarrow{\sim}q_2^!(\pi_1)_*(\C_{\mathfrak{b}_L/B_L}). $$

Using these isomorphisms we can define the action of $W^a$ on $H^{BM}_{T\times\C^\times}(\Fl_{e_d})$, see \cite{Lusztig} and \cite[Construction 7.1.3]{OY}
Namely, fix a simple affine reflection $s\in W^a$.
If $s$ is a reflection in the Weyl group $W_L$ of $L$, then we can define an action of $s$ on $(\pi_1)_*(\C_{\mathfrak{b}_L/B_L})$ via the usual finite dimensional Springer correspondence.
This gives rise to an action of $s$ on  \begin{equation}\label{eq:sheaf_equality}(\pi_2)_*(\omega_{\Fl_{e_d}})=q_2^!(\pi_1)_*(\C_{\mathfrak{b}_L/B_L}).
\end{equation}
Since $q_1,q_2$ are $T\times \C^\times$-equivariant, we get an action of $s$ on $H^{BM}_{T\times \C^\times}(\Fl_{e_d})$
(via pushforward of the left hand side of (\ref{eq:sheaf_equality}) to the point).
This action of $s$ is independent on the choice of $L$.
To check that the actions of the simple affine reflections satisfy the braid relations, it is enough to consider two simple reflections at a time, which reduces to the finite case, because any two simple reflections lie in $W_L$ for some choice of $\mathfrak{P}$.
 %and thus we get an action from any finite parabolic subgroup of $W^a$, by the existence of the usual Springer action as was considered in . Note that all the generators and relations can be checked in finite parabolics of $W^a$ and hence we can get a $W^a$ action.\\

%The above maps $q_1$ and $q_2$ are compatible with the torus $T\times\C^\times$ action and thus we can get an action  of $W^a$ on the $T\times\C^\times$-equivariant Borel-Moore homology.\\

 %In particular, consider  Cartesian diagram
 %(\ref{eq:Springer_diagram}) in the case of a minimal parahoric subgroup $\mathfrak{P}_i$, given by adding a single negative root, say $-\alpha$, to the Iwahori $\mathcal{B}$. By the construction of the Springer action of $W^a$, the action by the simple reflection $s_\alpha$ on
 %$H^{BM}_{T\times \C^\times}(\Fl_{e_d})$ comes from the action on $(\pi_1)_*(\C_{\mathfrak{b}_L/B_L})$ and it follows that the action of the simple reflection, already acts on the sheaf $(\pi_2)_*(\omega_{\Fl_{e_d}})$.

 To extend the $W^a$-action on $H^{BM}_{T\times \C^\times}(\Fl_{e_d})$  to an action of $\widetilde{W}$, recall that $\widetilde{W}=(\Lambda/\Lambda_0)\ltimes W^a$. We note that $\Lambda/\Lambda_0$ acts on $\Fl$.
 This action is constructed as follows. Take a lift of $\pi\in \Lambda/\Lambda_0\subset \widetilde{W}$ to $\dot{\pi}$ in the normalizer of $ T(\mathcal{K})$ and define the map $\Fl\rightarrow\Fl$ by $\g\mathfrak{B}\rightarrow g\dot{\pi}\mathfrak{B}$. This is well defined as the lift of any element of $(\Lambda/\Lambda_0)\subset \widetilde{W}$ normalizes $\mathfrak{B}$ and the map is independent of the chosen lift.
 From the definition of $\Fl_{e_d}$, see, e.g., (\ref{eq:Springer_fiber}),  it follows that this action preserves $\Fl_{e_d}$.  So we get an action of $\Lambda/\Lambda_0$ on $H^{BM}_{T\times\C^\times}(\Fl_{e_d})$.

 Recall that we write $\Ring$ for $H^*_{T\times \C^\times}(pt)$ and $\Field$ for $Frac(\Ring)$.
% These two actions combine to give an action of $\widetilde{W}$ by \cite[Construction 7.1.4]{OY}.

% The first result of the following Lemma follows immediately from this action on the $T\times\C^\times$-equivariant sheaf. The rest is stating the facts above.
 \begin{Lem}\label{Springer-action.basics}
 \begin{enumerate}
     %Under equivariant localization the Springer action by a simple reflection $s$ for $(f_w)\in H^{BM}_{T\times\C^\times}(\Fl_{e_d})$, $s(f)_w$ only depends on the preimage of $\pi_2(w)$ under $\pi_2$. Thus it only depends on $f_w$ and $f_{ws}$.
     \item The actions of $W^a,\Lambda/\Lambda_0$ give an action of the affine Weyl group $\widetilde{W}$ on $H^{BM}_{T\times\C^\times}(\Fl_{e_d})$.
     \item This $\widetilde{W}$-action is by $\Ring$-linear automorphisms.
     \item The action of $\widetilde{W}$ preserves the homological grading on $\Ring$.
 \end{enumerate}
 \end{Lem}
 \begin{proof}
 (1) follows from \cite[Theorem 3.3.3]{Yun} or \cite[Theorem 7.1.5]{OY}.

     (2) is a direct consequence of the construction.

     (3) follows from the construction of the action in \cite[Construction 7.13, 7.14]{OY}.

      %This isomorphism is compatible with the decomposition into connected components. So, as a  $W_L$-representation $H^{BM}_{T\times \C^\times}(\Fl_{e_d})$ breaks into the direct sum over $W_L$-orbits of the fixed points $\widetilde{W}$ after localization. It thus follows that in the particular case of a simple reflection $s$ the action of the group generated by $s$ breaks up into $s$-orbits after localization and thus the result of the Lemma follows.
 \end{proof}

 \begin{Rem}\label{connected:components}
In this remark we recall a classical description of the connected components
of $\Fl, \Fl_{e_d}$.

The connected components of the affine flag variety $\Fl$ are  in a natural  bijection with $\pi_1(G)$. Namely, recall the decomposition $\widetilde{W}=(\Lambda/\Lambda_0)\ltimes W^a$. The union of Schubert cells corresponding to the left $W^a$-orbits in $\widetilde{W}$ give the connected components. The group $\Lambda/\Lambda_0$ acts on $\Fl$ as  recalled above  in this section.  This action induces a simply transitive action on the set of components.

Let $\widetilde{G}$ be the simply connected cover of the derived subgroup $(G,G)\subset G$. Its extended affine Weyl group is $W^a$. In fact, $\Fl_{\widetilde{G}}$ is isomorphic to any of the connected components of $\Fl_{G}$. To see this note that we have a natural map $\Fl_{\widetilde{G}}\rightarrow\Fl_G$. This map is injective because the kernel of $\widetilde{G}\rightarrow G$ is contained in the center and thus contained in any Iwahori subgroup. The image contains precisely the $T$-fixed points given by $W^a$. The $\mathfrak{B}$-orbits coincide with the orbits of the pro-unipotent radical of $\mathfrak{B}$. Thus we see the image is precisely one connected component of $\Fl_G$. Since all connected components are isomorphic via the $\Lambda/\Lambda_0$-action the result follows.

 Moreover, the action of $\Lambda/\Lambda_0$ preserves $\Fl_{e_d}$. The embedding $\Fl_{\widetilde{G}}\hookrightarrow \Fl_{G}$ restricts to an embedding of the Springer fibers associated to $e_d$. This embedding realizes the Springer fiber for $\widetilde{G}$ as a connected component of the Springer fiber for $G$. It follows that every connected component of $\Fl_{e_d}$ for $G$ is identified with the Springer fiber of $e_d$ for $\widetilde{G}$.
\end{Rem}

\subsection{Springer action vs localization}\label{SS_CS_localiz}

%Recall from Remark \ref{Chern:Class} that we saw when applying the localization map of the Chern classes we get
%$({}^yg)\in H^*_{T\times\C^\times}(\Fl)\subset\bigoplus \Ring$ for $g\in \Ring=\C[\mathfrak{t}][\hbar]$. It follows that multiplication by Chern classes on $H^{BM}_{T\times\C^\times}(\Fl)\subset\prod_{w\in\widetilde{W}}\Field$ is given by
%\begin{equation}\label{eq:ChernFormula}
% g:(f_y)\mapsto({}^y(g)f_y).
%\end{equation}
%We now need to consider the Springer action. We begin considering the action of $(\Lambda/\Lambda_0)\subset \widetilde{W}$. This comes from a geometric action, so it is enough to consider what this map does on the fixed points, to understand this action under the localization map. The action on the fixed points is clearly $x\maps to x\pi$ for $x\in\widetilde{W}$ and $\pi\in(\Lambda/\Lambda_0)\subset \widetilde{W}$. Thus the action on the Borel-Moore homology for $(f_w)\in H^{BM}_{T\times\C^\times}(\Fl_{e_d})$ is given by

The goal of this section is to prove
Lemma \ref{Lem:Springer_fixed_pts}, which is the hard part of Lemma \ref{Lem:CS_action_fixed_pts}.

Recall that we write
$\iota$ for the localization homomorphism
$$H^{BM}_{T\times\C^\times}(\Fl_{e_d})\rightarrow
\bigoplus_{\widetilde{W}}\Field.$$

In the proof we will need an explicit description of $\operatorname{im}\iota$ for $SL_2$.
We identify $\mathbb{Z}$ with $\widetilde{W}_{SL_2}$ via $2m\mapsto t^{m\alpha}$, $2m+1\mapsto t^{m\alpha}s$, where $\alpha$ is the finite simple root of $SL_2$ and $s$ the corresponding simple reflection. Let $y$ be the basis element in $\h^*$
corresponding to the simple root. So $\Ring=\C[y,\hbar]$.

Pick an element $r\in \{0,\ldots,d\}$. For $k,m\in \Z$ set
\begin{equation}\label{eq:f_element}
    f_k^{r,(m)}=\prod_{i=1}^{r}(y+(k+m+i-1)\hbar).
\end{equation}
We then define elements $$b^r_k=(b^r_{k,\ell})_{\ell\in \Z}\in \bigoplus_{\Z} \C(y,\hbar)$$ for $r,k$ as above as follows. For $r=0$, we set $b^0_{k,\ell}:=\delta_{k,\ell}$, the Kroneker delta.  For $r\in \{1,\ldots,d\}$, define $m\in \Z, \epsilon\in \{0,1\}$ by $\ell=k+2m+\epsilon$
and set
\begin{equation}\label{eq:b_equation}
    b^r_{k,\ell}:=(-1)^{m+\epsilon}{r\choose m}\left(f_{k}^{r,(m)}\right)^{-1}.
\end{equation}

%Set $f_k^{r,(l)}=\prod_{i=1}^{r}(y+(k+l+i-1)\hbar)$ and $b_k^r\coloneqq (b_{k,l}^r)$, where
%\begin{equation*}
%  b_{k,k+2l+\epsilon}^r = \left \{
%  \begin{aligned}
%    &\frac{(-1)^{l}{r\choose l}}{f_k^{r,(l)}}, && %\text{if}\ \epsilon=0, r\neq 0 \\
%    &-\frac{(-1)^{l}{r\choose l}}{f_k^{r,(l)}}, && %\text{if}\ \epsilon=1, r\neq 0 \\
%    & 1, && \text{if}\ \epsilon=l=r=0\\
%    & 0, && \text{otherwise}
%  \end{aligned} \right.
%\end{equation*}
\begin{Lem}\label{Lem:SL_2_basis}
Let $G=SL_2$. Then
		 $\iota(H^{BM}_{T\times \C^\times}(\Fl_{e_d}))\subset \bigoplus_{\Z}\C(y,\hbar)$ has a basis over $\C[y,\hbar]$ given by $b_k^r$ where
		 \begin{itemize}
		 \item
		 either  $k=0,1$ and $r=0,...d-1$
		 \item or $r=d$ and $k\in \mathbb{Z}$.
		 \end{itemize}
	\end{Lem}
	%\footnote{From what I can see this claim formally implies that the affine Springer fibers for $\operatorname{SL}_2$ and $\operatorname{PGL}_2$ are the same. But this is false... Say what modifications are required.}.
\begin{proof}
The elements $b_k^r$ are indeed in  $\operatorname{im}\iota$: condition (i) of Corollary \ref{Cor:GKM_Springer} is immediate, while condition (ii)  is straightforward.
    %We first check these are indeed elements of the Borel-Moore homology. Note that every root appears with at most order $1$ singularities.\\
		%Now we just need to check the residue conditions, which reduces to the following:
		%$$ Res_{y=-(k+l+i-1)\hbar}(\frac{(-1)^{l}{r\choose l}}{f_k^{r,(l)}}-\frac{(-1)^{i}{r\choose i}}{f_k^{r,(l)}})=\frac{(-1)^{l}{r\choose l}}{(-1)^i i!(r-i)!\hbar^{r+1}}-\frac{(-1)^{i}{r\choose i}}{(-1)^l l!(r-l)!\hbar^{r+1}}=0$$
		
		Now we check that the elements $b^r_k$ for $r,k$ as in the statement of the lemma span that $\C[y,\hbar]$-module $\operatorname{im}\iota$. Pick $(g_k)\in \operatorname{im}\iota$.
		
		 Replacing $(g_k)$ with its sum with a linear combination of the elements $b^d_k$ we can assume that $(g_k)$ is supported between $0$ and $2d-2$. To see this assume that $g_r\neq 0$ for some $r<0$ and let $k$ be the minimal such number $r$.
		 Then the entry $g_k$ can have at most the same singularities as $1/f_k^{d,(0)}$ by Corollary \ref{Cor:GKM_Springer} and so is a multiple of this. Hence we can subtract a multiple of $b_k^d$ from $(g_k)$, such that the index of the minimal non-zero entry is bigger than $k$. Thus by induction we can assume that for all negative $k$ we have $g_k=0$.
		
		 A similar argument works for non-zero entries of $(g_k)$ for $k>2d-2$. Here the inequality $k>2d-2$ comes from the fact that $b_k^d$ has support exactly between $k$ and $k+2d-1$. So, subtracting the elements $b_k^d$ for $k\geq 0$  from $(g_r)$ doesn't change the condition that $g_r=0$ for $r<0$. So we can assume that $g_k\neq 0\Rightarrow 0\leqslant k\leqslant 2d-2$.
		
		Now using $b_k^r$ for $k=0,1$ and $r=0,...d-1$ we can continue reducing the support and using conditions (i) and (ii) of Corollary \ref{Cor:GKM_Springer} to ensure the maximal entries are indeed multiples of those of the $b_k^r$ we are considering. Indeed if $(g_k)$ is supported between $0$ and $2r-\epsilon$, $\epsilon\in\{0,1\}$ and $0\leq r\leq d-1$, then $g_{2r-\epsilon}$ has at most the singularities of $1/f_{1-\epsilon}^{r,(r)}$ by the conditions of Corollary \ref{Cor:GKM_Springer}.
		
		It follows that the elements $b_k^r$ for $k,r$ as described in the statement of the lemma span the $\Ring$-module $\operatorname{im}\iota$.

		To check that our elements are linearly independent (hence form a basis)  we use a partial order on $\Z$. Consider the partial order given by
		\begin{itemize}
		\item
		$k\preceq r$ if $0>k\geq r$,
		\item $k\preceq r$ if $2d-1\leq k\leq r$,
		\item and $0\preceq 1\preceq\dots\preceq 2d-2\preceq k$ for all $k\notin\{0,\dots 2d-2\}$.
		\end{itemize}
		
		For each element $b^r_k$ with $r,k$ as in the statement of the lemma, there is a unique maximal $\ell(=\ell(b_{k}^r))$ in the poset order such that $b^r_{k,l}\neq 0$, namely
		$$\ell(b^r_k)=\left\{\begin{array}{lr}
		     k &  \text{if }k<0\\
		     k+2r-1 & \text{else }.
		\end{array}\right.
		$$
		 It is clear that $(k,r)\mapsto \ell(b_k^r)$ identifies the set $r,k$ in the statement of the lemma with $\Z$. Now we use induction on the above partial order to show that the elements $b_k^r$ are linearly independent.
		 %Assume for contradiction that we have a linear combination $\sum a_k^r b_k^r=0$ for some $a_k^r\in \Ring$ not all $0$. Then there is some $k,r$ with $l(b_k^r)$ maximal among the $k',r'$ such that $a_{k'}^{r'}\neq 0$. It follows that the only $b_{k'}^{r'}$ with $a_{k'}^{r'}\neq 0$ that has a non-zero entry in $l(b_k^r)$ is this maximal $b_k^r$. It follows that the entry $l(b_k^r)$ of $\sum a_k^r b_k^r$ is non-zero, thus we get a contradiction and linear independence follows. \footnote{Overall, this looks like a collection of hints for a proof, rather than a proof. Also it needs to be written more carefully}
\end{proof}	

\begin{Rem}\label{Rem:Springer_basis_ssrk1}
For a general semisimple rank 1 group $G$, we have a similar basis for each connected component of $\Fl$ as described in Remark \ref{connected:components}. In that basis we use the polynomials $f^{r,(m)}_k$ replacing $y$ with the unique root $\alpha\in \mathfrak{t}^*\subset \Ring$ of $G$.
Indeed, the $1$-dimensional $T\times\C^\times$-orbit all have characters $\alpha+k\hbar$, where  $k\in\Z$ and $\alpha$ is the positive root of $G$. Each connected component of the affine Springer fiber for $G$ is isomorphic to the one for $\operatorname{SL}_2$ by Remark \ref{connected:components}. Then the same proof as for the $\operatorname{SL}_2$ case gives a similar basis.
\end{Rem}

\begin{Lem}\label{Lem:Springer_fixed_pts}
We have
\begin{equation}\label{eq:Springer_fixed_pts}
    \iota(s\beta)_x=\frac{d\hbar}{\,^x\!\alpha}\iota(\beta)_x+\frac{\,^{xs}\!\alpha-d\hbar}{\,^{xs}\!\alpha}\iota(\beta)_{xs}.
\end{equation}
\end{Lem}

%\begin{equation}\label{eq:SpringerFormula1}
% \pi:(f_w)\mapsto(f_{w\pi}).
%\end{equation}
	%Now we need to understand the Springer action of the simple reflection $s\coloneqq s_\alpha$ corresponding to a affine simple root $\alpha$.\\
\begin{proof}	
Our proof is in several steps.

{\it Step 1}.
Recall that	the Springer action is by $\Ring$-linear automorphisms and preserves the degrees by Lemma \ref{Springer-action.basics}.

{\it Step 2}. Let $\beta\in H^{BM}_{T\times \C^\times}(\Fl_{e_d})$ and let $s$ be a simple affine reflection. In this step we will prove that, for every $x\in \widetilde{W}$, the element $\iota(s\beta)_x$ only depends on $\iota(\beta)_x$ and $\iota(\beta)_{xs}$.
To do this we describe the localization morphism in terms of maps of sheaves. We use the notation
from the construction of the Springer action in Section \ref{SS_affine_Springer_reminder}.

For an arbitrary parahoric $\mathfrak{P}$ (including $\mathfrak{B}$) let $i^\mathfrak{P}$ denote the inclusion $$\pi_2^{-1}((\Fl^\mathfrak{P}_{e_d})^{T\times\C^\times})\rightarrow\Fl_{e_d}.$$
     By adjunction applied to $$\omega_{\Fl_{e_d}^{T\times\C^\times}}\rightarrow i^\mathfrak{B!}\omega_{\Fl_{e_d}},$$ we get a morphism of sheaves
     \begin{equation}\label{eq:adjunction_morphism}
     i^\mathfrak{B}_*(\omega_{\Fl_{e_d}^{T\times\C^\times}})\rightarrow\omega_{\Fl_{e_d}}
     \end{equation}
     in the $T\times\C^\times$-equivariant derived category. The localization map $$H^{BM}_{T\times\C^\times}(\Fl_{e_d}^{T\times\C^\times})\rightarrow H^{BM}_{T\times\C^\times}(\Fl_{e_d})$$
     is obtained from (\ref{eq:adjunction_morphism})
     by passing to cohomology.

      The same construction as  in Section \ref{SS_affine_Springer_reminder} establishes an action of the Weyl group $W_L$ of $L$ on
      \begin{equation}\label{eq:interm_sheaf}
      (\pi_2)_*i^\mathfrak{P}_*\left(\omega_{\pi_2^{-1}\left((\Fl^\mathfrak{P}_{e_d})^{T\times\C^\times}\right)}\right).
      \end{equation}
      Note that the space $\pi_2^{-1}\left((\Fl^\mathfrak{P}_{e_d})^{T\times\C^\times}\right)$ decomposes as the disjoint union of subspaces indexed by  $\widetilde{W}/W_L$ so that the subspace indexed by $xW_L$ contains exactly the fixed points labelled by the elements from $xW_L$. This decomposition is compatible with  Cartesian diagram (\ref{eq:Springer_diagram}). Hence this decomposition yields the decomposition of (\ref{eq:interm_sheaf}) into the direct sum with summands indexed by $\widetilde{W}/W_L$. Each summand is $W_L$-stable.

      Note that (\ref{eq:adjunction_morphism}) factors as  $$i^\mathfrak{B}_*(\omega_{\Fl_{e_d}^{T\times\C^\times}})\rightarrow i^\mathfrak{P}_*(\omega_{\pi_2^{-1}((\Fl^\mathfrak{P}_{e_d})^{T\times\C^\times})})
      \rightarrow \omega_{\Fl_{e_d}}.$$ The induced maps in cohomology,
      $$H^{BM}_{T\times\C^\times}(\Fl_{e_d}^{T\times\C^\times})\rightarrow H^{BM}_{T\times\C^\times}(\pi_2^{-1}((\Fl^\mathfrak{P}_{e_d})^{T\times\C^\times}))\rightarrow
      H^{BM}_{T\times\C^\times}(\Fl_{e_d})$$
      becomes an isomorphism after inverting some characters by a direct analog of Lemma \ref{Lem:localization}
      for ind-varieties.

      Apply the last observation to the minimal Levi subgroup corresponding to the reflection $s$. Consider the classes of the points $x$ and $xs$ in
      $H^{BM}_{T\times\C^\times}(\Fl_{e_d}^{T\times\C^\times})$. The subspace (over $\Field$) in the localization of
      $H^{BM}_{T\times\C^\times}(\pi_2^{-1}((\Fl^\mathfrak{P}_{e_d})^{T\times\C^\times}))$ spanned by their images is $s$-stable. Therefore the same conclusion is true if we consider the images in the localization of $H^{BM}_{T\times\C^\times}(\Fl_{e_d})$. This claim is equivalent to the claim of that $\iota(s\beta)_x$ only depends on $\iota(\beta)_x$ and $\iota(\beta)_{xs}$.

%	, which also gives that for $f=(f_w)\in H^{BM}_{T\times\C^\times}(\Fl)\subset\bigoplus_{w\in\widetilde{W}}\Field$, $s(f)_w$ only depends on $f_w$ and $f_{ws}$ and that the action preserves degree.\\
	
%Hence there are some constants $A_{w,w},A_{w,ws}\in \Field$ such that $s(f)_w=A_{w,w}f_w+A_{w,ws}f_{ws}$. \\
%\begin{Lem}
%For a simple affine reflection $s$, we have the Springer action %on $(f_w)\in H^{BM}_{T\times\C^\times}(\Fl_{e_d})$ is given by
%\begin{equation}\label{eq:SpringerFormula}
%s: (f_w)\mapsto (\frac{d\hbar}{{}^w(\alpha)}f_w+\frac{({}^{ws}(\%alpha)-d\hbar)}{{}^{ws}(\alpha)}f_{ws}),
%\end{equation}
%\begin{align*}
%    A_{w,w}&=\frac{d\hbar}{{}^w(\alpha)} \\
%    A_{w,ws}&=\frac{({}^{ws}(\alpha)-d\hbar)}{{}^{ws}(\alpha)}
%\end{align*}
%
%\end{Lem}
%\begin{proof}
{\it Step 3}.
For $x\in \widetilde{W}$ consider the element
     $a^x:=(\delta_{x,y})_{y\in\widetilde{W}}\in \bigoplus_{\widetilde{W}}\Field$, where, recall, $\delta_{x,y}$ is the Kroneker delta.
     Note that $a^x\in\operatorname{im}\iota$, one can see this, for example, from Corollary \ref{Cor:GKM_Springer}.
     Then the element $s(a^x)$ has the following properties:
     \begin{itemize}
         \item[(I)] $s(a^x)=A_{x,x}a^x+A_{x,xs}a^{xs}$ for some $A_{x,x},A_{x,xs}\in \Field$.
         This follows from Step 2.
         \item[(II)] $A_{x,x},A_{x,xs}$ are homogeneous of degree $0$. This is because $a^y$ is of degree $0$ for all $y$ and the Springer action preserves the grading, Step 1.
         \item[(III)]  $\,^x\!\alpha A_{x,x},\,^x\!\alpha A_{x,xs}$
         are linear functions. Indeed, from (i) of Corollary \ref{Cor:GKM_Springer} it follows that
         $\,^x\!\alpha A_{x,x},\,^x\!\alpha A_{x,xs}\in \Ring$. Now our claim follows from (II).
         \item[(IV)] $A_{x,x}+A_{x,xs}$ has no pole so is an element of $\C$. This follows from (III) and condition (ii) of Corollary \ref{Cor:GKM_Springer}.
         \item[(V)] The elements $\,^x\!\alpha A_{x,x}, \,^x\!\alpha (A_{x,xs}-1)$
         are divisible by $\hbar$. This follows from $sa^x=a^{xs}$ modulo $\hbar$, which is a consequence of \cite[Section 14.4]{GKM2}.
     \end{itemize}
     Combining (III),(IV) and (V), we see that
     \begin{equation}\label{eq:A_general}
     A_{x,x}=\frac{z\hbar}{\,^x\!\alpha}, A_{x,xs}=\frac{\,^x\!\alpha-z\hbar}{\,^x\!\alpha}
     \end{equation}
     for some $z\in \C$. Note that we have $\iota(s\beta)_x=A_{x,x}\iota(\beta)_x+A_{xs,x}\iota(\beta)_{xs}$. So the lemma amounts to showing that $z=d$.

%     Then we must have $s(a_w)$ is $0$ everywhere except $s(a_w)_w=A_{w,w}$ and $s(a_w)_{ws}=A_{w,ws}$ and it satisfies the conditions of Proposition \ref{BMcoho.description} for $H^{BM}_{T\times\C^\times}(\Fl_{e_d})$. It follows immediately from this that $A_{w,w}$ and $A_{w,ws}$ can only have a simple pole along ${}^w(\alpha)$ and thus ${}^w(\alpha)A_{w,w},{}^w(\alpha)A_{w,ws}\in H^*_{T\times\C}(pt)$.\\
%Further we know that the action of $s$ preserves degree and it thus follows that ${}^w(\alpha)A_{w,w},{}^w(\alpha)A_{w,ws}$ are degree $1$ polynomials.\\

{\it Step 4}.
We will use the case of $SL_2$ for the computation of the elements $A_{x,x},A_{x,xs}$. For an affine root $\beta$ let $\bar{\beta}$ be the projection of $\beta$ to $\h^*$. Equivalently, $\bar{\beta}$ is the unique root such that $\beta=\bar{\beta}+k\delta$ for some integer $k$, where, recall $\delta$ is the indecomposable imaginary root. Set $\beta:=\,^x\!\alpha$.
Let  $T_{\bar{\beta}}\subset T$ denote the kernel of $\bar{\beta}$ viewed as a homomorphism   $T\rightarrow\C^\times$. We write $\widetilde{W}_{\bar{\beta}}$ for the subgroup of $\widetilde{W}$ generated by the reflection $s_{\beta}$ and $t^{\bar{\beta}^\vee}$. Note that $\Fl^{T_{\bar{\beta}}}$ is a union of copies of the affine flag variety of the semisimple rank 1 subgroup $G_\beta:=Z_G(T_{\bar{\beta}})$ given by considering orbits of the loop group of $G_\beta$ at points $\widetilde{W}\subset \Fl$.  The connected  components of  $\Fl^{T_{\bar{\beta}}}$ are labelled by the cosets $\widetilde{W}_{\bar{\beta}}\setminus\widetilde{W}$. Each component is isomorphic to the affine flag variety of $SL_2$ and contains the $T$-fixed points labelled by points in the corresponding coset. A similar decomposition holds for $\Fl_{e_d}^{T_{\bar{\beta}}}$: it is the union of connected components labelled by $\widetilde{W}_{\bar{\beta}}\setminus\widetilde{W}$.

Now we can localize $H^{BM}_{T\times\C^\times}(\Fl_{e_d})$ at all characters that do not vanish on $T_{\bar{\beta}}$, which includes $\gamma+k\hbar$ for $\gamma\in R^+\setminus\{\bar{\beta}\}$ and $k\in\mathbb{Z}$, but not $\,^x\!\alpha+k\hbar$ for any $k$. By the ind-variety analog of  Lemma \ref{rel:localization}, this localized BM homology  is naturally isomorphic to the same localization  of $$H^{BM}_{T\times\C^\times}(\Fl^{T_{\bar{\beta}}}_{e_d}).$$ So the localization $H^{BM}_{T\times\C^\times}(\Fl_{e_d})$ breaks up as the direct sum of copies of the localized Borel-Moore homology of $\Fl_{e_d}(G_\beta)$, the equivalued unramified affine Springer fiber for the semisimple rank $1$ group $G_\beta$.
%\footnote{Please elaborate/ fix issues in this paragraph: the reference   does not have the result that you use. The summands are not really BM homology for $SL_2$ you need to enlarge the base ring. In fact, we cannot forget about the remainder of the torus -- we need to determine $\beta A_{x,x}, \beta A_{x,xs}$ that are elements of $\h^*\oplus \C\hbar$.}
%\footnote{An update: I think what should be said -- and proved -- is the following. We have an isomorphism of the localized BM homology for the initial Springer fiber and its smaller torus fixed points. The claim is that the result of applying $s$ to the class of our initial fixed point under this isomorphism coincides with applying the simple Dynkin reflection for $SL_2$ to a fixed point in the smaller torus fixed locus. This reduces the compoutation to a rank one group and then it should be argued that the computation further reduces to $SL_2$.}

{\it Step 5}.
Recall that the Springer action of $\widetilde{W}$ on $H^{BM}_{T\times\C^\times}(\Fl_{e_d})$ is $\Ring$-linear, Step 1. So it lifts to the localization considered in Step 4. Consider the connected component  in  $\Fl_{e_d}^{T_{\bar{\beta}}}$ whose $T$-fixed points are the coset $\widetilde{W}_{\bar{\beta}}x$.
Note that $\widetilde{W}_{\bar{\beta}}x=
\widetilde{W}_{\bar{\beta}}xs$ because $\beta=\,^x\!\alpha$. It follows that the summand in the localization of $H^{BM}_{T\times\C^\times}(\Fl_{e_d})$
corresponding to this component is
fixed by the Springer action of $s$.
% , since $\widetilde{W}_{\bar{\beta}}x$ is stable under right multiplication by $s$, as $\beta={}^x(\alpha)$.
%We now give a description of the basis of the Borel-Moore homology in this case.
Lemma \ref{Lem:SL_2_basis}, or, more precisely, its generalization discussed in Remark \ref{Rem:Springer_basis_ssrk1} give a basis in the summand we consider.
Note that there is a unique element $k\in \Z$ such that the basis element $b_k^d$ of this summand is $0$ at $xs$ and non-zero at $x$. Then $s(b_k^d)$ will only have singularities along the affine root hyperplanes $\,^x\!\alpha+p\hbar$ with $p\in \Z$. Further, by the description of the basis in Lemma \ref{Lem:SL_2_basis}, we have
$(b_k^d)_x=\frac{1}{f}$
for $f\coloneqq (\,^x\!\alpha-\hbar)(\,^x\!\alpha-2\hbar)\dots(\,^x\!\alpha-d\hbar)$.

Note that $(s(b_k^d))_{xs}= A_{x,xs}\frac{1}{f}$ and $(s(b_k^d))_{w}=0$ for all  $w$ of the form $xss_0s\dots$, where $s_0:= t^{-\alpha}s$. It thus follows that $A_{x,xs}/f$ only has singularities along $\,^{xs}\!\alpha,\ldots,\,^{xs}\!\alpha+(d-1)\hbar$. Since $\,^x\!\alpha A_{x,xs}$ is a linear polynomial, we see that $A_{x,xs}$ is proportional to $$\frac{\,^x\!\alpha-d\hbar}{\,^x\!\alpha}.$$
We apply (V)  to see that \begin{equation}\label{eq:A_xsx}
A_{x,xs}=\frac{\,^x\!\alpha-d\hbar}{\,^x\alpha}.
\end{equation}
This implies the claim of the lemma
by the last sentence of Step 3.
\begin{comment}
We also have that $s(b_k^d)$ is an element in the component of the localized Borel-Moore homology and so can be written as a sum of basis elements $b_k^r$ as in the Claim. Using this basis we can take away a multiple of some $b_l^d$ from $s(b_k^d)$ to eliminate its entry at $ws$. We thus get an element in $H^{BM}_{T\times\C^\times}(\Fl_{e_d})$ with  entry $(A_{ws,w}+A_{w,w})\frac{1}{f}$ at $w$ and $0$ at $ws,wt^{-\alpha},ws_{\alpha,1}\dots$ and thus we must have that the entry $(A_{ws,w}+A_{w,w})\frac{1}{f}$ only has at most simple poles along $({}^w(\alpha)-\hbar),\dots,({}^w(\alpha)-d\hbar)$. We also have that ${}^w(\alpha)A_{w,w},{}^w(\alpha)A_{ws,w}$ are degree $1$ polynomials, thus we must have $A_{ws,w}+A_{w,w}$ is a constant. We thus must have $A_{w,w}=\frac{d\hbar}{{}^w(\alpha)}+k$ for some constant $k\in\C$.\\
Now note that $s^2(g)=g$ for $g\in H^{BM}_{T\times\C^\times}(\Fl_{e_d})$, thus we must have the following relation
$$1=A_{w,w}A_{w,w}+A_{w,ws}A_{ws,w}=(\frac{d\hbar}{{}^w(\alpha)}+k)^2+\frac{({}^{ws}(\alpha)-d\hbar)}{{}^{ws}(\alpha)}\frac{({}^{w}(\alpha)-d\hbar)}{{}^{w}(\alpha)}=(\frac{d\hbar}{{}^w(\alpha)}+k)^2+\frac{({}^{w}(\alpha)^2-d^2\hbar^2)}{{}^{w}(\alpha)^2}$$
We thus must have that $\frac{d^2\hbar^2}{{}^w(\alpha)^2}=(\frac{d\hbar}{{}^w(\alpha)}+k)^2$. Thus we must have that $k=0$.\\
Thus we have that the Springer action for a simple reflection $s$ are given by
\begin{equation*}
    s:(f_y)\mapsto (\frac{d\hbar}{{}^y(\alpha)}f_y+\frac{({}^{ys}(\alpha)-d\hbar)}{{}^{ys}(\alpha)}f_{ys})
\end{equation*}
\end{comment}
\end{proof}

\end{document}